\documentclass[a4paper,11pt,twoside]{article}
\usepackage[utf8x]{inputenc}
\usepackage[english]{babel}

\usepackage[margin=2cm]{geometry}

\usepackage{isomath}

\usepackage{enumitem}

\usepackage{float}

\usepackage{graphicx}

\usepackage{tabulary}

\usepackage{tensor}

\usepackage{url,epsfig}
\usepackage {lettrine}
\usepackage {epigraph}
\usepackage {sidecap}	      
\usepackage {caption}
\usepackage{lmodern}
\usepackage{booktabs}

\usepackage{csquotes}   

\usepackage{amsmath,amsfonts,amssymb,amsthm}

\usepackage{bm}

\usepackage{stmaryrd}

 \usepackage{mathtools}
\usepackage{comment}
\usepackage{bbm}
\usepackage{latexsym}

\usepackage{enumitem}

\usepackage[hidelinks]{hyperref} 
 
\usepackage{mathrsfs}   

\usepackage{dsfont}

\usepackage[disable]{todonotes}

\theoremstyle{plain}
\newtheorem{teorema}{Theorem}[section]
\newtheorem{proposizione}[teorema]{Proposition}
\newtheorem{lemma}[teorema]{Lemma}
\newtheorem{corollario}[teorema]{Corollary}
\newtheorem*{teorema*}{Theorem}
\newtheorem*{proposizione*}{Proposition}
\newtheorem*{lemma*}{Lemma}
\newtheorem*{corollario*}{Corollary}

\theoremstyle{definition}
\newtheorem{definizione}[teorema]{Definition}

\newtheorem*{definizione*}{Definition}
\newtheorem*{esempio*}{Example}

\theoremstyle{remark}
\newtheorem{nota}[teorema]{Remark}

\newtheorem*{nota*}{Remark}
\newtheorem*{osservazione*}{Remark}
\newtheorem*{esercizio*}{Esercizio}

\newcommand{\I}{\mathrm{i}\mkern1mu}
\newcommand{\transpose}{\intercal}

\DeclareMathAlphabet{\mathpzc}{OT1}{pzc}{m}{it}
\newcommand{\m}[1]{\mathcal{#1}}
\newcommand{\bb}[1]{\mathbb{#1}}

\newcommand{\mrm}[1]{\mathrm{#1}}

\newcommand{\f}[1]{\mathfrak{#1}}
\newcommand{\scr}[1]{\mathscr{#1}}

\newcommand{\diff}{\partial}         
\newcommand{\bdiff}{\bar{\partial}} 

\newcommand{\cotJ}{T^*\!\!\!\scr{J}}

\DeclarePairedDelimiter\card{\lvert}{\rvert} 
\DeclarePairedDelimiter\norm{\lVert}{\rVert} 
 
\DeclarePairedDelimiter{\set}{\{}{\}}
\newcommand{\tc}{\mathrel{}\mathclose{}\middle|\mathopen{}\mathrel{}}

\numberwithin{equation}{section} 

\title{Scalar curvature and an infinite-dimensional hyperk\"ahler reduction}
\author{Carlo Scarpa and Jacopo Stoppa}
\date{\vspace{-5ex}}

\begin{document}

\maketitle

\abstract{\noindent We discuss a natural extension of the K\"ahler reduction of Fujiki and Donaldson, which realises the scalar curvature of K\"ahler metrics as a moment map, to a hyperk\"ahler reduction. Our approach is based on an explicit construction of hyperk\"ahler metrics due to Biquard and Gauduchon. This extension is reminiscent of how one derives Hitchin's equations for harmonic bundles, and yields real and complex moment map equations which deform the constant scalar curvature K\"ahler (cscK) condition. In the special case of complex curves we recover previous results of Donaldson. We focus on the case of complex surfaces. In particular we show the existence of solutions to the moment map equations on a class of ruled surfaces which do not admit cscK metrics.

\listoftodos

\setcounter{tocdepth}{2}
\tableofcontents

\section{Introduction}

Let $M$ be a compact K\"ahler manifold. The problem of finding a K\"ahler metric $g$ with prescribed cohomology class $[\omega_g]$ and constant scalar curvature
\begin{equation}\label{cscKequ}
s(g) = \hat{s}
\end{equation} 
has been intensively studied in complex differential geometry for the last few decades. A particularly fruitful parallel has been established between \eqref{cscKequ} (the cscK equation) and the Hermitian Yang--Mills (HYM) equation for a Hermitian metric $h$ on a holomorphic vector bundle $E$ over, say, a complex curve with a fixed K\"ahler form $(X, \omega)$,
\begin{equation}\label{HYMequ}
F(h) = \mu\,\mrm{Id}\otimes\omega.
\end{equation} 
Remarkably both equations can be realised as the zero moment map condition for a suitable infinite-dimensional K\"ahler reduction. For the {\sc HYM} equation \eqref{HYMequ} this goes back to \cite{AtiyahBott_YangMills}, the Atiyah--Bott characterisation of curvature as the moment map for the Hamiltonian action of unitary gauge transformations on the space of compatible $\bar{\diff}$-operators $\scr{A}$. The {\sc HYM} equation arises when looking for zeroes of the moment map along the orbits of the complexified action. In the case of the cscK equation Fujiki (\cite{Fujiki_moduli}) and Donaldson (\cite{Donaldson_scalar}) constructed a Hamiltonian action of the group of Hamiltonian symplectomorphisms $\m{G} = \mrm{Ham}(M, \omega_0)$ on the space $\scr{J}$ of almost complex structures compatible with a fixed symplectic form $\omega_0$, endowed with a natural symplectic (in fact K\"ahler) structure. It turns out that the moment map for this action, evaluated on Hamiltonians $h$, is given by
\begin{equation*}
\mu_J(h) = \int_M 2 (s(g_J) -\hat{s}) h \,\frac{\omega_0^n}{n!}.
\end{equation*}
Donaldson in \cite{Donaldson_scalar} then shows how the cscK equation \eqref{cscKequ} (with fixed $J$ and varying K\"ahler metric) arises when looking for zeroes of the moment map along the orbits of the complexified \emph{infinitesimal} action. 

An important feature of the moment map approach in the Hermitian Yang--Mills case is that the Atiyah-Bott K\"ahler reduction can be upgraded naturally to a hyperk\"ahler reduction of the holomorphic cotangent space $T^*\scr{A}$, as was shown by Hitchin (\cite{Hitchin_self_duality}). The real, respectively complex moment map, along orbits of the complexification, give Hitchin's harmonic bundle equations, 
\begin{align}\label{harmonicBundle}
\nonumber F(h) + [\phi, \phi^{*_h}] &= \mu\,\mrm{Id}\otimes\omega\\
\bar{\diff}\phi &= 0,
\end{align} 
involving a Higgs field $\phi \in \mrm{Hom}(E, E\otimes T^*X)$. The harmonic bundle equations \eqref{harmonicBundle} lead to a very rich theory, especially, but not only, in the case of complex curves. 

Thus it seems natural to ask if equations parallel to \eqref{harmonicBundle} can be derived and studied in the context of the cscK problem \eqref{cscKequ}. In fact this has been achieved by Donaldson (\cite{Donaldson_hyperkahler}) and Hodge (\cite{Hodge_phd_thesis}), in the special case of complex curves, as we discuss below in some detail. 

The present paper begins a more systematic study of this problem for higher dimensional manifolds. In the rest of this Introduction we summarise our main results. 

Naturally the first step is to upgrade the (Hamiltonian) action $\m{G}\curvearrowright \scr{J}$ to an action $\m{G}\curvearrowright \cotJ$ preserving a hyperk\"ahler structure. It is well known that, for a K\"ahler manifold $M$, there exists a hyperk\"ahler metric in a neighbourhood of the zero section of $T^*M$, see \cite{Feix_hyperkahler} and \cite{Kaledin_hyperkahler}; in the work of Donaldson such a structure is constructed for $\cotJ$, in an ad-hoc way, in the special case of a complex curve. Here we use instead an explicit construction, due to Biquard and Gauduchon (\cite{Biquard_Gauduchon}), of a canonical $G$-invariant hyperk\"ahler metric in the neighbourhood of the zero section of $T^*(G/H)$, the cotangent bundle of a Hermitian symmetric space of noncompact type, to obtain the required hyperk\"ahler structure in higher dimensions.

\begin{teorema}\label{teorema:HKThmIntro}
A neighbourhood of the zero section in the holomorphic cotangent bundle $\cotJ$ is endowed with a natural hyperk\"ahler structure. This is induced by regarding $\scr{J}$ as the space of sections of a $\mrm{Sp}(2n)$-bundle with fibres diffeomorphic to $\mrm{Sp}(2n)/\mrm{U}(n)$, and by the Biquard-Gauduchon canonical $\mrm{Sp}(2n)$-invariant hyperk\"ahler metric on a neighbourhood of the zero section in $T^*(\mrm{Sp}(2n)/\mrm{U}(n))$. The induced action $\m{G}\curvearrowright\cotJ$ preserves this hyperk\"ahler structure.
\end{teorema}
A review of the results of Biquard and Gauduchon can be found in \S \ref{PrelimBiqGau}. The construction of the hyperk\"ahler metric and the proof of Theorem \ref{teorema:HKThmIntro} are given in Section \ref{sec:cotJ_hyperk}.

Our next result studies the induced action $\m{G}\curvearrowright\cotJ$. We let $\Omega$ denote the Fujiki-Donaldson K\"ahler form on $\scr{J}$. We write $\bm{I}$, $\bm{J}$ for the complex structures underlying the hyperk\"ahler structure on $\cotJ$, with corresponding K\"ahler forms $\bm{\Omega_I}$, $\bm{\Omega}_{\bm{J}}$. By a slight abuse of notation we will denote by $(\cotJ,\bm{\Omega}_{\bm{I}})$ the open neighbourhood of the zero section in $\cotJ$ on which $\bm{\Omega}_{\bm{I}}$ is well-defined. Let $\bm{\Theta}$ be the canonical complex symplectic form on $\cotJ$. Points $(J,\alpha) \in \cotJ$ are pairs of an almost complex structure and a section $\alpha \in \mrm{End}(T^*M)$ (satisfying the compatibility conditions). We write $\alpha^\transpose$ for the dual endomorphism. Finally we recall that the Biquard-Gauduchon metric is expressed in terms of a canonical $\mrm{Sp}(2n)$-invariant function $\rho$ on $T^*(\mrm{Sp}(2n)/\mrm{U}(n))$.

\begin{teorema}\label{teorema:HamThmIntro}
The action $\m{G}\curvearrowright\cotJ$ is Hamiltonian with respect to the canonical symplectic form $\bm{\Theta}$; a moment map $\f{m}_{\bm{\Theta}}$ is given by
\begin{equation*}
{\f{m}_{\bm{\Theta}}}_{(J,\alpha)}(h)=-\int_M\frac{1}{2}\mrm{Tr}(\alpha^\transpose\m{L}_{X_h}J)\,\frac{\omega_0^n}{n!}.
\end{equation*}  
Moreover the action $\m{G}\curvearrowright(\cotJ,\bm{\Omega}_{\bm{I}})$ is Hamiltonian; a moment map $\f{m}_{\bm{\Omega}_{\bm{I}}}$ is given by
\begin{equation*}
\f{m}_{\bm{\Omega}_{\bm{I}}}=\mu\circ\pi+\f{m}
\end{equation*}
where $\mu$ is the moment map for the action $\m{G}\curvearrowright(\scr{J},\Omega)$, $\pi:\cotJ\to J$ is the projection and $\f{m}:\cotJ\to\mrm{Lie}(\m{G})^*$ is defined by 
\begin{equation*}
\f{m}_{(J,\alpha)}(h)=\int_{x\in M}\mrm{d}^c\rho_{(J(x),\alpha(x))}\left(\m{L}_{X_h}J,\m{L}_{X_h}\alpha\right)\,\frac{\omega_0^n}{n!}.
\end{equation*}
\end{teorema}
The proof of Theorem \ref{teorema:HamThmIntro} is given in Section \ref{sec:moment_map}, see in particular Lemma \ref{lemma:moment_map_theta} and Lemma \ref{lemma:mappa_momento_OmegaI}.

As we recalled above $\mu(J)$ is dual to the function $2(s(g_J) - \hat{s})$ under the natural $L^2$ product. Thus the analogue, in the cscK context, of the real moment map equation in \eqref{harmonicBundle}, is given by the problem of finding $(J, \alpha)$ such that the real moment map vanishes,
\begin{equation*}
\f{m}_{\bm{\Omega}_{\bm{I}}}(h) = \mu_J(h) + \f{m}_{(J,\alpha)}(h) = 0 \textrm{ for all } h\in C^{\infty}_0(M).
\end{equation*}
Similarly the holomorphicity of the Higgs field $\phi$ becomes the vanishing of the complex moment map, 
\begin{equation*}
{\f{m}_{\bm{\Theta}}}_{(J,\alpha)}(h)=0 \textrm{ for all } h\in C^{\infty}_0(M).
\end{equation*}
It turns out that one can easily compute the dual function to the \emph{complex} moment map under the $L^2$ pairing, at least when $J$ is integrable, i.e. under this identification we have
\begin{equation*}
\f{m}_{\bm{\Theta}}(J,\alpha)=-\mrm{div}\left({\bar{\diff}}^*\bar{\alpha}^\transpose\right).
\end{equation*}
Notice in particular that harmonic representatives of first order deformations of the complex structure  always provide solutions. 

On the contrary considerable more work is needed to turn the vanishing of the \emph{real} moment map into an explicit partial differential equation. We achieve this here for complex curves in Section \ref{section:moment_map_curve}, recovering Donaldson's equations for the hyperk\"ahler reduction, and for complex surfaces in Section \ref{sec:moment_map_complex_surface}. 

In what follows all metric quantities are computed with respect to $g_J$.  

\begin{teorema}\label{teorema:CurveThmIntro}
Let $M$ be a compact complex curve. Let $\psi$, $Q$ be the function and complex vector field on $M$, depending on a point in $\cotJ$, defined respectively by
\begin{align*}
\psi(\alpha) &= \frac{1}{1+\sqrt{1-\frac{1}{4}\norm{\alpha}^2}},\\
Q(J, \alpha) &= \frac{1}{2}\operatorname{Re}\left(g_J(\nabla^a\alpha,\bar{\alpha})\diff_a\right).
\end{align*}  
Then we have the identification
\begin{equation*}
\begin{split}
\f{m}_{\bm{\Omega}_{\bm{I}}}(J,\alpha)=2\,s(g_J)-2\,\hat{s}+\Delta\left(\mrm{log}\left(1+\sqrt{1-\frac{1}{4}\norm{\alpha}^2}\right)\right)+\mrm{div}\left(\psi(\alpha)\,Q(J,\alpha)\right)
\end{split}
\end{equation*} 
under the $L^2$ product.
\end{teorema}

In order to recover Donaldson's result we lower one index of $\alpha=\alpha\indices{_a^{\bar{b}}}\,\diff_{\bar{z}^b}\otimes\mrm{d}z^a$, using the metric $g_J$, obtaining the (symmetric) quadratic differential $\tau$. Then the complex moment map equation is equivalent to 
\begin{equation*}
\mrm{div}\left(({\nabla^{0,1}}^*\bar{\tau})^\sharp\right)=0
\end{equation*}
and holds automatically when $\tau$ is a \emph{holomorphic} quadratic differential. Similarly in this holomorphic case we have
\begin{equation*}
g(\nabla^a\tau,\bar{\tau})\diff_a=g^{a\bar{b}}g^{c\bar{e}}g^{d\bar{f}}\,\nabla_{\bar{b}}\tau_{cd}\,\tau_{\bar{e}\bar{f}}\,\diff_a=0,
\end{equation*}
so the real moment map equation becomes
\begin{equation*}
2\,s(g_J)-2\,\hat{s}+\Delta\left(\mrm{log}\left(1+\sqrt{1-\norm{\tau}^2}\right)\right)=0.
\end{equation*}
Fixing $J$ and varying $g$ instead along the orbits of the formal complefixication, this is exactly the equation that was used by Donaldson in \cite[Lemma $18$]{Donaldson_hyperkahler} to define a hyperk\"ahler structure on the cotangent bundle of the Teichm\"uller space of the curve $M$. T. Hodge (\cite{Hodge_phd_thesis}) proved existence and uniqueness of solutions for each fixed holomorphic $\tau$, at least under some boundedness assumptions on $\tau$.

We proceed to discuss the case of complex surfaces. In this case we prefer to write the moment maps in terms of an endomorphism $A$ of the \emph{real} tangent bundle given by
\begin{equation*}
A = \mrm{Re}(\alpha^\transpose).
\end{equation*}
We need some auxiliary notation. It is convenient to define the quantities
\begin{equation*}
\delta^{\pm}(A)=\frac{1}{2}\left(\frac{\mrm{Tr}(A^2)}{2}\pm\left(\left(\frac{\mrm{Tr}(A^2)}{2}\right)^2-4\,\mrm{det}(A)\right)^{\frac{1}{2}}\right).
\end{equation*}
Similarly to the case of curves we introduce two real spectral functions of the endomorphism $A$, given by
\begin{equation*}
\begin{split}
\psi(A)&=\frac{1}{\sqrt{4-2\,\delta^+(A)}+\sqrt{4-2\,\delta^-(A)}};\\
\tilde{\psi}(A)&=\frac{1}{\left(\sqrt{4-2\,\delta^+(A)}+\sqrt{4-2\,\delta^-(A)}\right)\left(2+\sqrt{4-2\,\delta^+(A)}\right)\left(2+\sqrt{4-2\,\delta^-(A)}\right)}.
\end{split}
\end{equation*}
Finally we write $\tilde{A}$ for the \emph{adjugate} of the endomorphism $A$, and $A = A^{1,0} + A^{0,1}$, $\tilde{A} = \tilde{A}^{1,0} + \tilde{A}^{0,1}$ for the type decomposition of the complexifications. 

\begin{teorema}\label{teorema:mm_reale_ruled}
Let $M$ be a compact complex surface. Let $X$ be the vector field on $M$, depending on a point of $\cotJ$, defined by
\begin{equation*}
\begin{split}
X(J,A)=&-\psi(A)\,\mrm{grad}\left(\frac{\mrm{Tr}(A^2)}{2}\right)+4\,\psi(A)\operatorname{Re}\left(g(\nabla^aA^{0,1},A^{1,0})\diff_a \right)-2\,\nabla^*(\psi(A)\,A^2)\\
&-4\,\left(\tilde{\psi}(A)\,\mrm{grad}\left(\mrm{det}(A)\right)+4\,\tilde{\psi}(A) \operatorname{Re}\left(g(\nabla^aA^{0,1},\tilde{A}^{1,0})\diff_a\right)+2\,\mrm{det}(A)\,\mrm{grad}(\tilde{\psi}(A))\right).
\end{split}
\end{equation*}
Then, when $J$ is integrable, we have the identification
\begin{equation*}
\begin{split}
\f{m}_{\bm{\Omega}_{\bm{I}}}(J,\alpha)=2(s(g_J)-\hat{s}) + \mrm{div} X(J, A)
\end{split}
\end{equation*} 
under the $L^2$ product.
\end{teorema}

We also obtain a similar but more complicated explicit expression for non-integrable $J$. The Theorem is proved in Section \ref{sec:moment_map_complex_surface}. Note that although this general expression for the vector field $X(J, A)$ is rather involved, it simplifies considerably when the endomorphism $A$ does not have maximal rank, yielding in this case   
\begin{equation*}
X(J, A) = 
-\frac{\mrm{grad} \left(\frac{1}{4}\mrm{Tr}(A^2)\right)}{1+\sqrt{1-\frac{1}{4}\mrm{Tr}(A^2)}}+\frac{2\operatorname{Re}\left( g (\nabla^aA^{0,1},A^{1,0})\diff_a\right)}{1+\sqrt{1-\frac{1}{4}\mrm{Tr}(A^2)}}-\nabla^*\left(\frac{A^2}{1+\sqrt{1-\frac{1}{4}\mrm{Tr}(A^2)}}\right).
\end{equation*}
The resulting real moment map in this low-rank case is quite similar to the one for Riemann surfaces given in Theorem \ref{teorema:CurveThmIntro}. 

Following the well-known case of the cscK equation, it is natural to study the system of partial differential equations obtained by fixing the complex structure $J$ in $\scr{J}$ and varying instead the metric $g$ in a fixed K\"ahler class and the ``Higgs term'' $\alpha$ in some class of first-order deformations of the complex structure $J$. Just as in the cscK case this can be understood as a formal (infinitesimal) complexification of the action of $\m{G}$, as is described in \S\ref{ComplexifiedSection}.
The resulting real and complex moment map equations form the system
\begin{align}\label{HcscKsystem}
\nonumber 2\,s(g)+ \mrm{div} X(g, A) &= 2\,\hat{s}\\
\mrm{div}\left({\bar{\diff}}^*_g A^{1,0}\right) &= 0,
\end{align}
reminiscent of Hitchin's harmonic bundle equations \eqref{harmonicBundle}. We refer to this system as the \emph{HcscK equations}.  

An important aspect of the theory of Higgs bundles is that a slope-unstable bundle $E$ may still carry a harmonic metric, for a suitable choice of Higgs field. Our last result in Section \ref{sec:ruled_surface_realmm} establishes an analogue of this fact in the context of the cscK equation, albeit in a not completely satisfactory way; details are explained in \S\ref{sec:ruled_surface_realmm}.

\begin{teorema}\label{teorema:solution_mm_reale}
Fix a compact complex curve $\Sigma$ of genus at least $2$, endowed with the hyperbolic metric $g_{\Sigma}$. Let $M$ be the ruled surface $M = \bb{P}(\m{O}\oplus T\Sigma)$, with projection $\pi\!: M \to \Sigma$ and relative hyperplane bundle $\m{O}(1)$, endowed with the K\"ahler class 
\begin{equation*}
[\omega_m] = [\pi^*\omega_{\Sigma}] + m\, c_1(\m{O}(1)),\, m > 0.
\end{equation*}
Then for all sufficiently small $m$ the HcscK equations \eqref{HcscKsystem} can be solved on $(M, [\omega_m])$. 
\end{teorema}

On the other hand it is well-known that, for all positive $m$, $(M,[\omega_m])$ does not admit a cscK metric (see \cite[\S $3.3$ and \S $5.2$]{Szekelyhidi_phd}).

\bigskip

The paper is organised as follows. Section \ref{PrelimSec} contains some preliminary material on the Hermitian symmetric space $\mrm{Sp}(2n)/\mrm{U}(n)$ (in particular, on its natural identifications with Siegel's upper half space and with the space of compatible linear complex structures), on the space of almost complex structures $\scr{J}$, and on the Biquard-Gauduchon construction. In Section \ref{sec:cotJ_hyperk} we use these results to construct a hyperk\"ahler structure on $\cotJ$ and to derive the implicit expression of the moment maps. Section \ref{sec:curva} discusses the case of a complex curve, while Section \ref{sec:complex_surface} is devoted to a general complex surface. Finally in Section \ref{sec:ruled_surface} we study the case of a ruled surface in more detail.

\bigskip

\noindent\textbf{Acknowledgements.} We are grateful to Olivier Biquard, Francesco Bonsante, Ruadha\'i Dervan, Mario Garcia-Fernandez, Julien Keller and Richard Thomas for some discussions related to the present paper. The research leading to these results has received funding from the European Research Council under the European Union's Seventh Framework Programme (FP7/2007-2013) / ERC Grant agreement no. 307119. 

\section{Preliminary results}\label{PrelimSec}

\subsection{Notation and conventions}

The imaginary unit is ``$\I$''; if $(M,J)$ is a complex manifold of complex dimension $n$ we use $i,j,k\dots$ as indices for tensors defined on the underlying real manifold, so $i,j\dots\in\set{1,2,\ldots,2n}$. For complex tensors instead we use $a,b,c,\dots$ as indices ranging from $1$ to $n$. We always use the Einstein convention on repeated indices.

\smallskip

Our conventions for the Laplacian are the following: $\Delta=\mrm{d}\,\mrm{d}^*+\mrm{d}^*\mrm{d}$, and in particular for a function $\varphi$ we get $\Delta(\varphi)=-\mrm{div}\,\mrm{grad}(\varphi)$. In complex coordinates we find, for a K\"ahler metric, $\Delta(\varphi)=-2g^{a\bar{b}}\diff_a\diff_{\bar{b}}\varphi$. The ``complex Laplacian'' is $\Delta_{\bar{\diff}}=\Delta_{\diff}=\frac{1}{2}\Delta$.

\smallskip

When working with left actions of a Lie group $G$ on a manifold $M$, we'll denote them by
\begin{equation*}
\begin{split}
G\times M&\to M\\
(g,x)&\mapsto \sigma_g(x)\\
\mbox{or simply }\ (g,x)&\mapsto g.x
\end{split}
\end{equation*}
We let $\f{g} =T_eG$ be the Lie algebra of the group $G$, identified with the space of left-invariant vector fields on $G$.

If we have a left action $G\curvearrowright M$, we define for $a\in\f{g}$ the \emph{fundamental vector field} $\hat{a}$ on $M$ as
\begin{equation*}
\hat{a}_x =\frac{\mrm{d}}{\mrm{d}t}\Bigr|_{t=0}\left(\mrm{exp}(-t\,a).x\right)\in T_xM.
\end{equation*}
The vector field $\hat{a}$ is also called the \emph{infinitesimal action} of $a$ on $M$. The minus sign in this definition is due to the fact that, with this definition, the map
\begin{equation*}
\begin{split}
\f{g}&\to\Gamma(TM)\\
a&\mapsto\hat{a}
\end{split}
\end{equation*}
is a Lie algebra homomorphism (see \cite[Proposition $3.8$, Appendix $5$]{LibermannMarle}).

\smallskip

Now, let $(M,\omega)$ be a symplectic manifold. For a function $f\in\m{C}^\infty(M)$ we define the \emph{Hamiltonian vector field} $X_f$ as
\begin{equation*}
\mrm{d}f=-X_f\lrcorner\omega.
\end{equation*}
Here the symbol $\lrcorner$ is the contraction of the first component, \emph{i.e.} 
\begin{equation*}
-X\lrcorner\omega=-\omega(X,-)=\omega(-,X).
\end{equation*}
The \emph{Poisson bracket} of two functions $f,g\in\m{C}^\infty(M)$ is defined as
\begin{equation*}
\{f,g\}=\omega(X_f,X_g).
\end{equation*}
This is a Lie bracket on $\m{C}^\infty(M)$, and the Hamiltonian construction $f\mapsto X_f$ is a Lie algebra homomorphism between $(\m{C}^\infty(M),\{-,-\})$ and $(\Gamma(TM),\left[-,-\right])$. 

Bringing together the last two paragraphs, consider now a symplectic left action $G\curvearrowright(M,\omega)$. We say that the action is \emph{Hamiltonian} if there is a \emph{moment map}
\begin{equation*}
\mu:M\to\f{g}^*
\end{equation*}
that is equivariant with respect to $G\curvearrowright M$ and the co-adjoint action of $G$ on $\f{g}^*$, and such that $\langle\mu,a\rangle$ is a Hamiltonian function of the vector field $\hat{a}$ on $M$. In a more concise way:
\begin{align*}
& \forall g\in G,\,\forall x\in M,\,\forall a\in\f{g}\quad\quad \langle\mu_{g.x},a\rangle=\langle\mu_x,\mrm{Ad}_{g^{-1}}(a)\rangle;&\\
& \forall g\in G,\,\forall a\in\f{g}\quad\quad\quad\quad\quad\quad \mrm{d}\left(x\mapsto\langle\mu_x,a\rangle\right)=-\hat{a}\lrcorner\omega.&
\end{align*}

\subsection{Some matrix spaces}\label{section:Siegel}

Consider the symplectic vector space $(\bb{R}^{2n},\Omega_0)$, where $\Omega_0$ is the canonical symplectic form, i.e. the matrix 
\begin{equation*}
\Omega_0=\begin{pmatrix}0 & \mathbbm{1}_m \\ -\mathbbm{1}_m & 0\end{pmatrix}.
\end{equation*}
We recall that the \emph{symplectic group} $\mrm{Sp}(2n)$ is defined as
\begin{equation*}
\mrm{Sp}(2n)=\set*{A\in\mrm{GL}(2n,\bb{R})\tc A^{\intercal}\Omega_0 A=\Omega_0}.
\end{equation*}
This is a connected real Lie group, and we are particularly interested on some actions of $\mrm{Sp}(2n)$.

By the usual identification of $\bb{C}^n$ with $\bb{R}^{2n}$ as real vector spaces, we can see $\mrm{GL}(n,\bb{C})$ as the subgroup of $\mrm{GL}(2n,\bb{R})$ consisting of all the real invertible $2n\times 2n$ matrices that commute with the standard complex structure on $\bb{R}^{2n}$, which is defined by $\Omega_0$. The groups $\mrm{Sp}(2n)$, $\mrm{SO}(2n)$ and $\mrm{U}(n)$ are tied together by the well known result:
\begin{equation*}
\mrm{Sp}(2n)\cap \mrm{SO}(2n)=\mrm{Sp}(2n)\cap \mrm{U}(n)=\mrm{SO}(2n)\cap \mrm{U}(n)=\mrm{U}(n).
\end{equation*}
The coset space $\mrm{Sp}(2n)/\mrm{U}(n)$ will play a fundamental role in what follows. It carries a natural K\"ahler metric, coming from its identification with \emph{Siegel's upper half space} $\f{H}$, and at the same time it can be identified naturally with the space $\m{AC}^+$ of linear complex structures compatible with a linear symplectic form.


\begin{definizione}
Siegel's upper half space $\f{H}(n)$ is the set of all symmetric $n\times n$ complex matrices whose imaginary part is positive definite.
\end{definizione}

Some reference texts for the properties of $\f{H}$ are \cite{Siegel_halfspace, DallaVolta_matrici}. Siegel's upper half space is a generalization of the well-known hyperbolic plane, and these two spaces share many interesting geometric properties.

In particular, $\f{H}$ is a complex manifold, with complex structure given simply by multiplication by $\I$. It will be more notationally convenient, however, to consider on $\f{H}$ the \emph{conjugate} complex structure, i.e. we will define the complex structure on $\f{H}$ to be the multiplication by $-\I$. The reason for this choice will become clear when we will use it to define a complex structure on $\m{AC}^+$, see Proposition \ref{prop:metrica_AC+_Siegel}.

On $\f{H}$ there is also a K\"ahler structure; the metric tensor at a point $Z=X+\I\,Y$ is
\begin{equation}\label{eq:Siegel_metric}
\mrm{d}s^2_{Z}=\mrm{trace}\left(Y^{-1}\mrm{d}Z\,Y^{-1}\mrm{d}\overline{Z}\right)
\end{equation}
where $\mrm{d}Z$ and $\mrm{d}\overline{Z}$ are the (symmetric) matrices of differentials $(\mrm{d}z_{ab})_{1\leq a,b\leq m}$ and its conjugate. We refer to \cite{Siegel_halfspace} for the details. This metric has a local potential of the form
\begin{equation*}
\frac{\diff}{\diff z\indices{^a_b}}\frac{\diff}{\diff\bar{z}\indices{^c_d}}\log\mrm{det}(Y)
\end{equation*}
see for example \cite[\S $5$]{DallaVolta_matrici}.

The symplectic group $\mrm{Sp}(2n)$ acts on $\f{H}(n)$ by an analogue of the M\"obius transformations. For $P=\begin{pmatrix}A & B\\ C & D\end{pmatrix}\in\mrm{Sp}(2n)$ and $Z\in\f{H}(n)$ one defines
\begin{equation*}
P.Z =(AZ+B)(CZ+D)^{-1}.
\end{equation*}
This is a well-defined left action on $\f{H}(n)$ that preserves the metric \eqref{eq:Siegel_metric}.

\begin{proposizione}[Theorem $1$ in \cite{Siegel_halfspace}]
The action of $\mrm{Sp}(2n)$ on $\f{H}$(n) is transitive. Moreover, every holomorphic bijection $\f{H}(n)\to\f{H}(n)$ is a M\"obius transformation.
\end{proposizione}

Consider the stabilizer of $\I\mathbbm{1}_{n}\in\f{H}(n)$ under this action. It is clear that the matrix $\begin{pmatrix}A & B\\ C & D\end{pmatrix}$ stabilizes $\I\mathbbm{1}$ if and only if $\I A+B=\I D-C$, i.e. $B+C=0$ and $A=D$. Hence the stabilizer is $\mrm{Sp}(2n)\cap\mrm{GL}(n,\bb{C})=\mrm{U}(n)$, with the previous identifications.

\subsubsection{The space $\m{AC}^+$}\label{section:mAC}

Let $J\in\mrm{Sp}(2n)$ be a linear almost complex structure preserving $\Omega_0$. Then the product $\Omega_0J$ is a nondegenerate symmetric matrix, defining a bilinear form $\beta_J$. We are interested in the set of all almost complex structures $J\in\mrm{Sp}(2n)$ such that $\beta_J$ is positive definite, and we define
\begin{equation*}
\begin{split}
\m{AC}(2n)& =\set*{J\in\mrm{Sp}(2n)\tc J^2=-\mathbbm{1}}\\
\m{AC}^+(2n)& =\set*{J\in\mrm{Sp}(2n)\tc J^2=-\mathbbm{1},\ \beta_J>0}.
\end{split}
\end{equation*}

Notice that the matrix $-\Omega_0$ is an element of $\m{AC}^+(2n)$, and $\beta_{-\Omega}$ is just the usual Euclidean product.

\begin{lemma}\label{lemma:stabilizzatore}
Let $\mrm{Sp}(2n)$ act on $\m{AC}(2n)$ by conjugation. Then the stabilizer of any $J\in\m{AC}(2n)$ is $\mrm{Sp}(2n)\cap\mrm{SO}(\beta_J)$.
\end{lemma}

In particular, the stabilizer of $-\Omega_0$ is $\mrm{Sp}(2n)\cap\mrm{O}(2n)=\mrm{U}(n)$.

\begin{proposizione}\label{proposizione:transitive_action}
The action of $\mrm{Sp}(2n)$ on $\m{AC}^+(2n)$ is transitive.
\end{proposizione}

Then for any $J\in\m{AC}^+(2n)$ there is some $P\in\mrm{Sp}(2n)$ which conjugates $J$ to $-\Omega_0$; a possible choice of $P$ is given by
\begin{equation*}
P_J =(\Omega_0 J)^{1/2}.
\end{equation*}
Proposition \ref{proposizione:transitive_action} and Lemma \ref{lemma:stabilizzatore} tell us that we can identify $\m{AC}^+(2n)$ with the quotient
\begin{equation*}
\mrm{Sp}(2m)/\left(\mrm{SO}(2m)\cap\mrm{Sp}(2m)\right)\cong\mrm{Sp}(2m)/\mrm{U}(m).
\end{equation*}
Let $\phi:\m{AC}^+\to\f{H}$ be the diffeomorphism that is given by composing the two identifications of $\m{AC}^+$ and $\f{H}$ with $\mrm{Sp}(2n)/\mrm{U}(n)$. These identifications are defined by fixing the reference points $-\Omega_0\in\m{AC}^+$ and $\I\mathbbm{1}\in\f{H}$, so that $\phi$ is given by the composition
\begin{center}
\begin{tabular}{c c c c c}
$\m{AC}^+$ & $\rightarrow$ & $\mrm{Sp}(2n)/\mrm{U}(n)$ & $\rightarrow$ & $\f{H}$\\
$J$ & $\mapsto$ & $P_J^{-1}\mrm{U}(n)$ & $\mapsto$ & $P_J^{-1}.(\I\mathbbm{1})$
\end{tabular}
\end{center}
and $\phi$ is a smooth isomorphism of $\mrm{Sp}(2n)$--spaces, i.e. it is a diffeomorphism that commutes with the $\mrm{Sp}(2n)$ actions. Using this identification of the two spaces we obtain a K\"ahler structure on $\m{AC}^+$. A straightforward computation of the differential of $\Phi$ at the point $-\Omega_0$ gives

\begin{proposizione}\label{prop:metrica_AC+_Siegel}
Endow $\m{AC}^+$ with the complex structure and K\"ahler metric pulled back from Siegel's upper half space $\f{H}$. Then the complex structure on $T_J\m{AC}^+$ is given by
\begin{equation*}
A\mapsto JA.
\end{equation*}
Moreover
\begin{equation*}
\phi^*(\mrm{d}s^2)_J(A,B)=\frac{1}{2}\mrm{Tr}(AB).
\end{equation*}  
\end{proposizione}

\begin{nota}\label{nota:traccia_tripla}
Notice that for $A,B,C\in T_J\m{AC}^+$, the trace of the product matrix $ABC$ vanishes; indeed
\begin{equation*}
\begin{split}
\mrm{Tr}(ABC)=-\mrm{Tr}(JJABC)=\mrm{Tr}(JABCJ)=\mrm{Tr}(JJABC)=-\mrm{Tr}(ABC).
\end{split}
\end{equation*}
This will be useful later to compute the moment map for the action in Section \ref{sec:moment_map}.
\end{nota}

The next result is an expression for the curvature of this metric on $\m{AC}^+$, that is obtained by considering its homogeneous space structure.

\begin{proposizione}\label{prop:curvatura_AC}
Consider $\m{AC}^+(2n)$ with the K\"ahler metric induced by its identification with $\mrm{Sp}(2n)/\mrm{U}(n)$ (and with $\f{H}(2n)$). The curvature of this metric at the point $-\Omega_0$ is given by
\begin{equation*}
R_{-\Omega_0}(A,B)(C)=-\frac{1}{4}\Big[[A,B],C\Big].
\end{equation*}
\end{proposizione}

\paragraph*{The holomorphic cotangent space of $\m{AC}^+$.} To write the moment map equations in Section \ref{sec:cotJ_hyperk} it will be convenient to have an expression for the complex structure of ${T^{1,0}}^*\m{AC}^+$. Consider $J\in\m{AC}^+$; the cotangent space of $\m{AC}^+$ at $J$ is the space of all $\alpha:V^*\to V^*$ such that 
\begin{equation*}
\begin{split}
J\alpha^\transpose+&\alpha^\transpose J=0\\
J^\transpose\Omega_0\alpha^\transpose+&\alpha\Omega_0 J^\transpose=0.
\end{split}
\end{equation*}
We want to consider the \emph{holomorphic} cotangent space at $J$, i.e. $\alpha\in V^*\otimes\bb{C}$ that satisfies
\begin{equation*}
\begin{split}
J^\transpose\alpha&=\I\alpha\\
\alpha\,J^\transpose&=-\I\alpha\\
J^\transpose\Omega_0\mrm{Re}(\alpha)^\transpose&+\mrm{Re}(\alpha)\Omega_0 J=0.
\end{split}
\end{equation*}
The tangent space $T_{J,\alpha}\left({T^{1,0}}^*\!\m{AC}^+\right)$ is defined by the set of all pairs $(\dot{J},\dot{\alpha})$ with $\dot{J}\in T_J\m{AC}^+$ and $\dot{\alpha}\in V^*\otimes\bb{C}$ such that
\begin{enumerate}
\item $\dot{J}^\transpose\alpha+J^\transpose\dot{\alpha}=\I\dot{\alpha}$
\item $\alpha\,\dot{J}^\transpose+\dot{\alpha}\,J^\transpose=-\I\dot{\alpha}$
\item $\dot{J}^\transpose\Omega_0\mrm{Re}(\alpha)^\transpose+J^\transpose\Omega_0\mrm{Re}(\dot{\alpha})^\transpose+\mrm{Re}(\dot{\alpha})\Omega_0 J+\mrm{Re}(\alpha)\Omega_0\dot{J}=0$.
\end{enumerate}

A lengthy calculation using the identification between $\m{AC}^+$ and $\f{H}^+$ gives the following expression for the canonical complex structure of ${T^{1,0}}^*\!\!\m{AC}^+$:
\begin{equation}\label{eq:canonical_compl_str}
\begin{split}
T_{J,\alpha}\left({T^{1,0}}^*\!\!\m{AC}^+\right)&\longrightarrow T_{J,\alpha}\left({T^{1,0}}^*\!\!\m{AC}^+\right)\\
(\dot{J},\dot{\alpha})&\mapsto(J\dot{J},\dot{\alpha}J^\transpose+\dot{J}^\transpose\alpha).
\end{split}
\end{equation}

\subsection{The space $\scr{J}$}

Let $(M,J_0,\omega_0)$ be a compact K\"ahler manifold, of complex dimension $n$. We are interested in the space
\begin{equation*}
\scr{J}=\set*{J\in\Gamma(\mrm{End}TM)\tc J^2=-\mrm{Id},\ \omega_0(J-,J-)=\omega_0(-,-)\mbox{ and }\omega_0(J-,-)>0}
\end{equation*}
of all almost complex structures on $M$ that are compatible with the symplectic form $\omega_0$.

For any point $x_0\in M$, there is a neighbourhood $U\in\m{U}(x_0)$ and a coordinate system $\bm{u}:U\to\bb{R}^{2n}$ such that $\omega_0(\bm{u})$ is expressed as the canonical $2$-form on $\bb{R}^{2n}$ (in other words, $\bm{u}$ is a local system of Darboux coordinates around $x_0$); hence for all $x\in U$ and for all $J\in\scr{J}$, the matrix associated to $J_p$ in the coordinate system $\bm{u}$ is an element of $\m{AC}^+$. Notice that, for a different system of Darboux coordinates $\bm{v}$, the ``change of coordinates matrix'' $\frac{\diff\bm{v}}{\diff\bm{u}}$ is a $\mrm{Sp}(2n)$-valued function. Considering the matrices associated to $J_x$ in the two Darboux coordinate systems we have
\begin{equation*}
J_x(\bm{v})=\frac{\diff\bm{v}}{\diff\bm{u}}(x)\,J_x(\bm{u})\,\left(\frac{\diff\bm{v}}{\diff\bm{u}}\right)^{-1}\!\!\!\!\!\!\!(x)
\end{equation*}
so the two different elements of $\m{AC}^+$ differ by the action of an element of $\mrm{Sp}(2n)$ on $\m{AC}^+$. We have all the ingredients to define a $\mrm{Sp}(2n)$-bundle with fibre $\m{AC}^+$ on the manifold $M$, that is trivialized in Darboux coordinates. We denote by $\m{E}\xrightarrow{\pi}M$ this fibre bundle, and it's clear that $\scr{J}=\mrm{\Gamma}(M,\m{E})$.

This description of the infinite-dimensional manifold $\scr{J}$ as a space of sections is quite convenient for describing extra structures on $\scr{J}$; for example, for any $J\in\scr{J}$ the tangent space at $J$ is
\begin{equation*}
T_J\scr{J}=T_J\Gamma(M,\m{E})=\Gamma(M,J^*(\mrm{Vert}\,\m{E}))
\end{equation*}
where $\mrm{Vert}\,\m{E}$ is the vertical distribution of $\m{E}$, the kernel of the projection on the base $\pi:\m{E}\to M$. For any $x\in M$, $J^*(\mrm{Vert}\m{E})_x=\mrm{Vert}_{J(x)}\m{E}\cong T_{J(x)}\m{AC}^+$; here the identification is done by fixing a Darboux coordinate system around $x$, i.e. by locally trivializing $\m{E}$. In other words, any $A\in T_J\scr{J}$ is itself a section of a fibre bundle on $M$ that is trivial over any system of Darboux coordinates, and in any such trivialization $A(x)\in T_{J(x)}\m{AC}^+$. This description of $T_J\scr{J}$ can be made more intrinsic by noticing that any such $A$ must be itself an endomorphism of $TM$, so that
\begin{equation*}
\begin{split}
T_J\scr{J}=\lbrace A\in\Gamma(M,\mrm{End}(TM))&\mid AJ+JA=0\mbox{ and }\omega_0(A-,J-)+\omega_0(J-,A-)=0\rbrace.
\end{split}
\end{equation*}
The second condition, $\omega_0(A-,J-)+\omega_0(J-,A-)=0$, tells us that the bilinear form $(v,w)\mapsto g_J(Av,w)$ is symmetric. Then, in a system of local coordinates for $M$, the conditions for an endomorphism $A$ to be in $T_J\scr{J}$ are equivalent to these useful identities:
\begin{equation}\label{eq:indentita_tangenti}
\begin{split}
J\indices{^i_j}A\indices{^j_k}&=-A\indices{^i_j}J\indices{^j_k}\\
g(J)_{ij}A\indices{^j_k}&=g(J)_{kj}A\indices{^j_i}.
\end{split}
\end{equation}

Using the various geometric structures on $\m{AC}^+$, we can induce similar structures on $\scr{J}$; let's see how this is done for the K\"ahler structure of $\m{AC}^+$. First of all, we define a complex structure $\bb{J}:T\scr{J}\to T\scr{J}$ as follows: fix $J\in\scr{J}$ and $A\in T_J\scr{J}$; for any $x\in M$ consider a trivialization of $\m{E}$ around $x$, giving the usual identification $A(x)\in T_{J(x)}\m{AC}^+$; on this vector space we have the complex structure described in the previous Section. It is given by $A(x)\mapsto J(x)A(x)=(JA)(x)$, so we define $(\bb{J}A)(x) =(JA)(x)$ for every $x\in M$. Notice moreover that the final result is independent from the choice of the trivialization, since the action of $\mrm{Sp}(2n)$ on $\m{AC}^+$ preserves the complex structure. Then
\begin{align*}
\bb{J}:T_J\scr{J}&\to T_J\scr{J}\\
A&\mapsto JA
\end{align*}
defines an almost complex structure on $\scr{J}$. The same approach works to define a metric; for $A,B\in T_J\scr{J}$ and $x\in M$ the number $\frac{1}{2}\mrm{Tr}(A(x)B(x))$ depends just on $x$, and not on the particular trivialization chosen to see $A(x)$, $B(x)$ as matrices, since the action of $\mrm{Sp}(2n)$ on $\m{AC}^+$ is isometric. We can then define a metric
\begin{align*}
\bb{G}:T_J\scr{J}\times T_J\scr{J}&\to\bb{R}\\
(A,B)&\mapsto\frac{1}{2}\int_{x\in M}\mrm{Tr}(A_xB_x)\,\frac{\omega_0^n}{n!}
\end{align*}
and all the ``algebraic'' relations of $\bb{J}$, $\bb{G}$ carry over from those of the metric and the complex structure on $\m{AC}^+$; in particular $\bb{G}(\bb{J}-,\bb{J}-)=\bb{G}(-,-)$, and so we obtain a $2$-form on $\scr{J}$,
\begin{equation*}
\Omega_J(A,B) =\bb{G}_J(\bb{J}A,B)=\frac{1}{2}\int_{x\in M}\mrm{Tr}(J_xA_xB_x)\,\frac{\omega_0^n}{n!}.
\end{equation*}

\begin{nota}
If we denote by $g_J$ the Hermitian metric on $M$ defined by $\omega_0$ and $J\in\scr{J}$, then $g_J(A,B)=\mrm{Tr}(AB)$ for any $A,B\in T_J\scr{J}$. Indeed
\begin{equation*}
g_J(A,B)=g\indices{^j^k}g\indices{_i_l}A\indices{^i_j}B\indices{^l_k}=g\indices{^j^k}g\indices{_i_j}A\indices{^i_l}B\indices{^l_k}=A\indices{^i_l}B\indices{^l_i}
\end{equation*}
where we have used \eqref{eq:indentita_tangenti} in the second equality. So we can rewrite the expression of $\bb{G}$ in a way that makes more explicit the role of the point $J$, i.e.
\begin{equation*}
\begin{split}
\bb{G}_J(A,B)&=\frac{1}{2}\int_{x\in M}g_J(A,B)_x\,\frac{\omega_0^n}{n!}=\\
&=\frac{1}{2}\int_{x\in M}\omega_0(A,JB)_x\,\frac{\omega_0^n}{n!}.
\end{split}
\end{equation*}
\end{nota}

\begin{teorema}\label{teorema:kahler_J}
With the almost complex structure $\bb{J}$ and the metric $\bb{G}$, $\scr{J}$ is an infinite-dimensional (formally) K\"ahler manifold.
\end{teorema}
This theorem is actually a particular case of a more general result. Indeed, it holds \emph{for any} fibre bundle $N\to M$ over a manifold with a fixed volume form whose fibres are K\"ahler manifolds, see \cite[Theorem $2.4$]{Koiso_complex_structure}.

\smallskip

The cotangent bundle of $\scr{J}$ can also be described in terms of the fibre bundle $\m{E}\to M$; indeed, since $T_J\scr{J}=\Gamma(M,J^*(\mrm{Vert}\,\m{E}))$, we also have
\begin{equation*}
T^*_J\scr{J}=\Gamma(M,J^*(\mrm{Vert}\,\m{E}^*)).
\end{equation*}
A more explicit description can be obtained by locally trivializing the bundle and identifying $\m{E}_x$ with $\m{AC}^+$:
\begin{equation*}
T^*_J\!\scr{J}=\set*{\alpha\in\Gamma(\mrm{End}(T^*M))\tc J^\transpose\circ\alpha+\alpha\circ J^\transpose=0\mbox{ and }g_J(\alpha^\transpose-,-)\mbox{ is symmetric}}.
\end{equation*}
Indeed these conditions on $\alpha$ tell us that in a system of Darboux coordinates on $U\subset M$, $\alpha(x)\in T_{J(x)}\m{AC}^+$ for every $x\in U$. The pairing between $T^*_J\scr{J}$ and $T_J\scr{J}$ is
\begin{equation*}
\langle\alpha,A\rangle=\frac{1}{2}\int_M\alpha\indices{_i^j}A\indices{^i_j}\,\frac{\omega_0^n}{n!}.
\end{equation*}
Later on, we will need the \emph{holomorphic} cotangent bundle of $\scr{J}$, that we will still denote by $\cotJ$; the context will make clear what space we are working on. The $(1,0)$-part of $T^*_J\scr{J}$ consists of those $\alpha\in T^*_J\scr{J}\otimes\mathbb{C}$ that satisfy $J^\transpose\circ\alpha=\I\,\alpha$. If $J$ is integrable and we fix a system of coordinates on $M$ that are holomorphic with respect to $J$, then an element $\alpha\in {T^{1,0}}^*_J\scr{J}$ in these coordinates is written as
\begin{equation*}
\alpha=\alpha\indices{_a^{\bar{b}}}\,\diff_{\bar{b}}\otimes\mrm{d}z^a.
\end{equation*}

\subsection{Characterisations of hyperk\"ahler manifolds}

\begin{definizione}
Let $(M,g)$ be a Riemannian manifold, and let $I$, $J$ be two almost complex structures on $M$ such that
\begin{enumerate}
\item $IJ=-JI$;
\item $g(I-,I-)=g(J-,J-)=g(-,-)$;
\item $\forall x\in M,\,\forall v\in T_xM\quad g(Iv,Jv)=0$.
\end{enumerate}
Then $(M,g,I,J)$ is a \emph{hyperk\"ahler manifold} if $(M,g,I)$ and $(M,g,J)$ are both K\"ahler.
\end{definizione}
In this case, by letting $K:=IJ$ we have that for any $\bm{u}\in\bb{S}^2$ also $(M,g,u_1I+u_2J+u_3K)$ is K\"ahler, hence the name. The standard notation is to call $\omega_1$, $\omega_2$ and $\omega_3$ (or $\omega_I$, $\omega_J$ and $\omega_K$) the three $2$-forms defined respectively by $g\circ I$, $g\circ J$ and $g\circ K$. Moreover, we let $\omega_c:=\omega_2+\I\omega_3$; this is a (complex-valued) $2$-form on $M$, and an important remark is that \emph{$\omega_c$ is a $(2,0)$ holomorphic symplectic form, relatively to the complex structure $I$}.

This lemma gives us a useful criterion to prove that some structures are hyperk\"ahler.

\begin{lemma}[Lemma $6.8$ in \cite{Hitchin_self_duality}]\label{lemma:closed_hyperkahler}
Let $(M,g)$ be a Riemannian manifold, and assume that $I,J$ are almost complex structures on $M$ satisfying conditions $1,2,3$ of the above definition. Then $(M,g,I,J)$ is hyperk\"ahler if and only if
\begin{equation*}
\mrm{d}\omega_1=\mrm{d}\omega_2=\mrm{d}\omega_3=0
\end{equation*}
where the $\omega_i$s are defined as above. 
\end{lemma}
In other words, the three forms being closed is enough to ensure the integrability of $I$, $J$ and $K$. We remark that this conditions follows from an algebraic manipulation of the Newlander-Nirenberg criterion, so it holds also in the infinite-dimensional setting -- guaranteeing at least the \emph{formal} integrability of the complex structures.

Another important criterion we will use is the following, that is taken from the discussion in \cite{Biquard_Gauduchon}.

\begin{lemma}\label{lemma:hyperekahler_criterio}
Let $(M,g)$ be a Riemannian manifold, and let $I$ be a complex structure on $M$, compatible with $g$. Assume that we also have a $(2,0)$ symplectic form on $M$, $\omega_c$. Then we can always define a tensor $J$ on $M$ by the condition $g(J-,-)=\mrm{Re}\,{\omega_c}(-,-)$. Assume that
\begin{enumerate}
\item $\mrm{d}\omega_1=0$, for $\omega_1=g(I-,-)$;
\item $J^2=-\mrm{Id}$.
\end{enumerate}
Then $(M,g,I,J)$ is a hyperk\"ahler manifold, and the three $2$-forms defined by $g$ and $I$, $J$ and $K=IJ$ are, respectively, $\omega_1$, $\mrm{Re}\,\omega_c$ and $\mrm{Im}\,\omega_c$.
\end{lemma}

\begin{proof}
First of all notice that $\omega_2$ is closed, since $\omega_c$ is closed and by definition $\omega_2=\mrm{Re}\,\omega_c$. Now we check the various algebraic identities between $g,I$ and $J$.

Compatibility of $g$ and $J$: for all $v,w$ we have, using the (anti-) symmetries of $g$ and $\omega_c$ 
\begin{equation*}
g(Jv,Jw)=\mrm{Re}(\omega_c(v,Jw))=-\mrm{Re}(\omega_c(Jw,v))=-g(JJw,v)=g(v,w).
\end{equation*}

Anticommutativity of $I$ and $J$: of course $IJ+JI=0$ if and only if $g(IJv+JIv,w)=0$ for every pair of tangent vectors $v,w$. From the definition of $J$ we have
\begin{equation*}
g(IJv+JIv,w)=-g(Jv,Iw)+g(JIv,w)=-\mrm{Re}(\omega_c(v,Iw))+\mrm{Re}(\omega_c(Iv,w))
\end{equation*}
and since $\omega_c$ is of type $(2,0)$ relatively to $I$
\begin{equation*}
-\mrm{Re}(\omega_c(v,Iw))+\mrm{Re}(\omega_c(Iv,w))=-\mrm{Re}(\I\omega_c(v,w))+\mrm{Re}(\I\omega_c(v,w))=0.
\end{equation*}
From these two conditions it is now trivial to check that for any tangent vector $v$, $g(Iv,Jv)=0$.

By Lemma \ref{lemma:closed_hyperkahler}, the only thing that remains to be checked is that, if we let $K=IJ$ and $\omega_3=g(K-,-)$, we have $\mrm{d}\omega_3=0$. However, as we have also seen above
\begin{equation*}
\begin{split}
g(Kv,w)&=g(IJv,w)=-g(JIv,w)=-\mrm{Re}(\omega_c(Iv,w))=\\ &=-\mrm{Re}(\I\omega_c(v,w))=\mrm{Im}(\omega_c(v,w))
\end{split}
\end{equation*}
so the closedness of $\omega_3$ follows from that of $\omega_c$.
\end{proof}
It is important to highlight the fact that the proof of Lemma \ref{lemma:hyperekahler_criterio} is purely algebraic, provided that $\omega_c$ and $\omega_1$ are closed; we do not need to resort to computations in local coordinates. Hence, this criterion for checking the hyperk\"ahler condition also holds in the infinite-dimensional setting; this is where we intend to apply it in Section \ref{sec:cotJ_hyperk}.

\subsection{A result of Biquard and Gauduchon}\label{PrelimBiqGau}

Here we recall the construction of Biquard and Gauduchon in \cite{Biquard_Gauduchon} of a hyperk\"ahler metric on the cotangent bundle of any hermitian symmetric space $\Sigma=G/H$. Assume that $\Sigma$ has a complex structure $I$ and a Hermitian metric $h$. For any $x\in\Sigma$ we have an identification of ${T^{1,0}}^*\Sigma$ and $T\Sigma$ given by taking the metric dual of the real part of $\xi\in{T^{1,0}_x}^*\Sigma$. Under this identification, for every $\xi\in{T^{1,0}}^*_x\Sigma$, we can consider the endomorphism $IR(I\xi,\xi)$ of $T_x\Sigma$ associated to the Riemann curvature tensor $R$. Since this is self-adjoint we can consider its spectral functions; we are interested in particular in the function $f:\bb{R}_{>0}\to\bb{R}$ defined by
\begin{equation}\label{eq:funzione_f}
f(x):=\frac{1}{x}\left(\sqrt{1+x}-1-\log\frac{1+\sqrt{1+x}}{2}\right).
\end{equation}
\begin{teorema}[\cite{Biquard_Gauduchon}]\label{teorema:BiquardGauduchon}
Let $\left(\Sigma=G/H,I,h\right)$ be a Hermitian symmetric space of compact type, and let $\omega_c$ be the canonical symplectic form on $T^*\Sigma$. Then there is a unique $G$-invariant hyperk\"ahler metric $g$ on $({T^{1,0}}^*\Sigma,I,\omega_c)$ such that the restriction of $g$ to the zero-section of ${T^{1,0}}^*\Sigma$ coincides with the Hermitian metric of $\Sigma$.

Moreover, we have an explicit expression for this metric: if we identify $T^*\Sigma$ and $T\Sigma$ using the metric on the base, the K\"ahler form is given by $\omega_I=\pi^*\omega_{\Sigma}+\mrm{d}\mrm{d}^c\rho$, where $\rho$ is the function on $T\Sigma$ defined by
\begin{equation}\label{eq:rho_def}
\rho(x,\xi)=h_x\left(f(-IR(I\xi,\xi))\xi,\xi\right).
\end{equation}
Here $f$ is the function defined by \eqref{eq:funzione_f}, evaluated on the self-adjoint endomorphism $-IR(I\xi,\xi)$.

If instead $\Sigma$ is of noncompact type, the same statement holds in an open neighbourhood $N\subseteq {T^{1,0}}^*\Sigma$ of the zero section. This neighbourhood is the set $N$ of all $\xi$ such that the modulus of the eigenvalues of $-IR(I\xi,\xi)$ is less than $1$.
\end{teorema}

In particular this theorem applies to the quotient $\mrm{Sp}(2n)/\mrm{U}(n)$, a symmetric space that is diffeomorphic to Siegel's upper half space $\f{H}(n)$. If we endow $\mrm{Sp}(2n)/\mrm{U}(n)$ with the K\"ahler structure coming from $\f{H}(n)$ we obtain a K\"ahler symmetric space of noncompact type, to which we can apply Theorem \ref{teorema:BiquardGauduchon}. Then $T^*(\mrm{Sp}(2n)/\mrm{U}(n))$ has a hyperk\"ahler metric, at least in a neighbourhood of the zero section. Moreover, also $\m{AC}^+$ is diffeomorphic to $\mrm{Sp}(2n)/\mrm{U}(n)$, and the K\"ahler structure on $\m{AC}^+$ is induced from the one of $\mrm{Sp}(2n)/\mrm{U}(n)$ using this isomorphism. Then we can also carry the hyperk\"ahler structure of $T^*(\mrm{Sp}(2n)/\mrm{U}(n))$ to $T^*\!\m{AC}^+$.

Let's denote by $(g,I,\omega)$ the K\"ahler structure of $\m{AC}^+$; then it is natural to also denote by $I$ the complex structure on $T^*\!\m{AC}^+$, and we let $\theta$ be the canonical $2$-form. Theorem \ref{teorema:BiquardGauduchon} guarantees that $\hat{g}:=\pi^*\omega+2\I\diff\bar{\diff}\rho$ is a hyperk\"ahler metric on $T^*\!\m{AC}^+$.

\begin{nota}
Biquard and Gauduchon consider the full cotangent bundle; for notation reasons, for us it will be more convenient to just consider the \emph{holomorphic} cotangent bundle of $\m{AC}^+$ and $\scr{J}$, but this won't cause issues, thanks to the usual canonical identifications of the two. Moreover, Biquard and Gauduchon in \cite{Biquard_Gauduchon} use the convention $R(X,Y)=\nabla_{[X,Y]}-[\nabla_X,\nabla_Y]$, rather than the more usual $R(X,Y)=[\nabla_X,\nabla_Y]-\nabla_{[X,Y]}$, that is the one we are going to use. This is why we introduced that minus sign in equation \eqref{eq:rho_def}.
\end{nota}

\section{An infinite dimensional hyperk\"ahler reduction}\label{sec:cotJ_hyperk}

\subsection{A hyperk\"ahler structure on $\cotJ$}\label{sec:hyperkahler}

In this section we prove Theorem \ref{teorema:HKThmIntro}, constructing the required hyperk\"ahler structure on $\cotJ$, using the results of the previous Section.

Firstly, we realise $\cotJ$ itself as a space of sections of a $\mrm{Sp}(2n)$-bundle, with fibres diffeomorphic to $T^*\!\m{AC}^+$. The action of $\mrm{Sp}(2n)$ on $T^*\!\m{AC}^+$ induced by the action on $\m{AC}^+$ is again by conjugation. More precisely, for $h\in\mrm{Sp}(2n)$ and $(J,\alpha)\in T^*\!\m{AC}^+$ we have
\begin{equation*}
h.(J,\alpha)=(h\,J\,h^{-1},(h^{-1})^{\transpose}\alpha\,h^\transpose)
\end{equation*}
and that is precisely also the change that the matrices associated to $(J,\alpha)\in\cotJ$ in a Darboux coordinate system undergo under a change to another Darboux coordinate system. Hence, as was the case for $\scr{J}$, we can write $\cotJ$ as the space of sections of some $\mrm{Sp}(2n)$-bundle $\hat{\m{E}}\xrightarrow{\pi}M$. Notice that we have a natural $\mrm{Sp}(2n)$-bundle map $F:\hat{\m{E}}\to\m{E}$, covering the identity on $M$, that is induced by the projection $p:T^*\!\m{AC}^+\to\m{AC}^+$. Define $F:\hat{\m{E}}\to\m{E}$ as follows: for $\xi\in\hat{\m{E}}$, let $x=\pi(\xi)$ and fix a system of Darboux coordinates $\bm{u}:U\to\bb{R}^{2n}$ around $x$; consider then the trivializations $\Phi_{\bm{u}}:\hat{\m{E}}_{\restriction U}\to U\times T^*\!\m{AC}^+$ and $\phi_{\bm{u}}:\m{E}_{\restriction U}\to U\times\m{AC}^+$ and let
\begin{equation*}
F(\xi):=\phi_{\bm{u}}^{-1}\circ (\mrm{id}\times p)\circ\Phi_{\bm{u}}(\xi).
\end{equation*}
Then it's immediate to check that the definition of $F$ does not depend upon the choice of Darboux coordinates on $M$, since the action of $\mrm{Sp}(2n)$ on $T^*\!\m{AC}^+$ is the one induced by the action on $\m{AC}^+$. This map accounts for the fact that from a section $s$ of $\hat{\m{E}}$ we can always get a section $J=F(s)$ of $\m{E}$ and a section $\alpha$ of $J^*(\mrm{Vert}\,\m{E}^*)$.

Next, with a view to applying Lemma \ref{lemma:hyperekahler_criterio}, we introduce the following tensors on $\cotJ $:
\begin{enumerate}
\item[$\cdot$] a Riemannian metric $\bm{G}$;
\item[$\cdot$] a complex structure $\bm{I}$ compatible with $\bm{G}$;
\item[$\cdot$] a symplectic form $\bm{\Omega}_c$ of type $(2,0)$ with respect to $\bm{I}$.
\end{enumerate}
By Lemma \ref{lemma:hyperekahler_criterio}, to prove that this defines a hyperk\"ahler structure on $\cotJ$ it suffices to show that
\begin{enumerate}
\item $\bm{J}^2=-1$, where $\bm{J}$ is defined by $\bm{G}(\bm{J}-,-)=\mrm{Re}(\bm{\Omega}_c)$;
\item $\mrm{d}\bm{\Omega_I}=0$, where $\bm{\Omega_I}(-,-) =\bm{G}(\bm{I}-,-)$.
\end{enumerate}

Since $\scr{J}$ already has a complex structure $\bb{I}$, we define $\bm{I}$ as the complex structure induced on $\cotJ $ by $\bb{I}$; explicitly, from equation \eqref{eq:canonical_compl_str} we have
\begin{equation*}
\forall (J,\alpha)\in \cotJ ,\ \forall (A,\varphi)\in T_{(J,\alpha)}(\cotJ )\quad \bm{I}_{(J,\alpha)}(A,\varphi):=(JA,\varphi J^\transpose+A^\transpose\alpha). 
\end{equation*}
Let $(I,\theta,\hat{g})$ be the triple of a complex structure, canonical $2$-form and hyperk\"ahler metric on $T^*\!\m{AC}^+$ described in Section \ref{PrelimBiqGau}. The $2$-form $\bm{\Theta}$ on $\cotJ $ will be  
\begin{equation}\label{eq:2forma_canonica}
\forall (J,\alpha)\in \cotJ ,\ \forall v,w\in T_{(J,\alpha)}(\cotJ )\quad \bm{\Theta}_{(J,\alpha)}(v,w):=\int_{x\in M}\theta_x(v_x,w_x)\,\frac{\omega_0^n}{n!}
\end{equation}
where as usual we are taking around each $x\in M$ a trivialization of the fibre bundle (i.e. a system of Darboux coordinates). It's not obvious that this expression is actually independent from the choice of the trivialization; it will be shown in Lemma \ref{lemma:2canonica_J}. A point to remark is that $\bm{\Theta}$ is automatically of type $(2,0)$ with respect to $\bm{I}$, since $\theta$ is of type $(2,0)$ with respect to the complex structure of $T^*\!\m{AC}^+$.

The natural candidate to be the hyperk\"ahler metric is the metric $\bm{G}$ induced on $\cotJ$ from the Biquard-Gauduchon metric on $T^*\!\m{AC}^+$
\begin{equation}\label{eq:metrica_hyper_J}
\bm{G}_{(J,\alpha)}(v,w):=\int_{x\in M}\hat{g}_x(v_x,w_x)\,\frac{\omega_0^n}{n!}
\end{equation}
but again we should check that this expression is independent from the choice of Darboux coordinates around each point. Assuming for the moment that it is, the fact that $\bm{I}$ and $\bm{G}$ are compatible follows immediately from the compatibility of $I$ and $\hat{g}$ on $T^*\!\m{AC}^+$; moreover, the $2$-form $\bm{\Omega}_{\bm{I}}$ is
\begin{equation}\label{eq:forma_hyper_J}
{\bm{\Omega}_{\bm{I}}}_{(J,\alpha)}(v,w):=\int_{x\in M}(\omega_I)_x(v_x,w_x)\,\frac{\omega_0^n}{n!}
\end{equation}
where $\omega_I$ is the $2$-form defined in Theorem \ref{teorema:BiquardGauduchon}. Notice also that it is enough to check that \eqref{eq:forma_hyper_J} does not depend on the choice of coordinates to guarantee that also \eqref{eq:metrica_hyper_J} does not. Again under the (provisional) assumption that \eqref{eq:forma_hyper_J} is well-defined, we notice that condition $(1)$ above is automatically satisfied. Indeed, the complex structure $\bm{J}$ is pointwise induced from the analogue complex structure $J$ of $T^*\!\m{AC}^+$, from which it inherits algebraic properties like $J^2=-1$.

Summing up these considerations, to prove Theorem \ref{teorema:HKThmIntro} we just have to verify that $\bm{\Theta}$ and $\bm{\Omega_I}$ are well-defined and closed.

First we prove the well-definedness of $\bm{\Omega_I}$. Notice that, since the action of $\mrm{Sp}(2n)$ on $\m{AC}^+$ is isometric and holomorphic, both $\rho$ and $\diff\bdiff\rho$ are $\mrm{Sp}(2n)$-invariant.

\begin{lemma}
The $2$-form $\bm{\Omega}_{\bm{I}}$ of equation \eqref{eq:forma_hyper_J} is well-defined.
\end{lemma}

\begin{proof}
Choose $(J,\alpha)\in\cotJ$, $v,w\in T_{(J,\alpha)}(\cotJ)$ and a Darboux coordinate system $\bm{u}$. In this coordinate system $\bm{u}$ the bundle $\Xi$ trivializes, and we have to check that, for $x\in\mrm{dom}(\bm{u})$, the expression
\begin{equation*}
\pi^*\omega_{(J(x),\alpha(x))}(v(x),w(x))+\mrm{d}\mrm{d}^c\rho_{(J(x),\alpha(x))}(v(x),w(x))
\end{equation*}
does not depend upon the choice of the coordinate system $\bm{u}$. If $\bm{v}$ is a different Darboux coordinate system, the matrix $\varphi:=\frac{\diff\bm{v}}{\diff\bm{u}}$ is a $\mrm{Sp}(2n)$-valued function and the previous expression becomes, in the new coordinate system,
\begin{equation*}
\begin{split}
\pi^*\omega_{\varphi(x).(J(x),\alpha(x))}&(\varphi(x).v(x),\varphi(x).w(x))+\\
&+\mrm{d}\mrm{d}^c\rho_{\varphi(x).(J(x),\alpha(x))}(\varphi(x).v(x),\varphi(x).w(x)).
\end{split}
\end{equation*}
Since both terms are $\mrm{Sp}(2n)$-invariant this proves the claim.
\end{proof}

The closedness of both forms is guaranteed by the following theorem.

\begin{teorema}\label{teorema:differenziazione}
Let $k$ be a $r$-form on $T^*\!\m{AC}^+$ invariant under the $\mrm{Sp}(2n)$-action, and let $K$ be a $r$-form on $\cotJ$ such that
\begin{equation*}
\begin{split}
\forall(J,\alpha)\in\cotJ,&\ \forall v_1,\dots,v_r\in T_{(J,\alpha)}\cotJ\\ &K_{(J,\alpha)}(v_1,\dots,v_r)=\int_{x\in M} k_{(J(x),\alpha(x))}(v_1(x),\dots,v_r(x))\,\frac{\omega_0^n}{n!}
\end{split}
\end{equation*}
where the second expression is computed by taking a local trivialization of $\Xi$ around each $x\in M$. Then
\begin{equation*}
\mrm{d}K_{(J,\alpha)}(\dots)=\int_{x\in M}\mrm{d}k_{(J(x),\alpha(x))}(\dots)\,\frac{\omega_0^n}{n!}.
\end{equation*}
\end{teorema}

\begin{nota}
In fact we just need this result for $r=0,1,2$. For $r=0$ the result is elementary: for $(J,\alpha)\in\cotJ$ and $v\in T_{(J,\alpha)}(\cotJ)$, let $(J_t,\alpha_t)$ be a path in $\cotJ$ such that $v=\frac{\mrm{d}}{\mrm{d}t}\Bigr|_{t=0}(J_t,\alpha_t)$. Then
\begin{equation*}
\begin{split}
\mrm{d}K_{(J,\alpha)}(v)&=v(K)_{(J,\alpha)}=\frac{\mrm{d}}{\mrm{d}t}\Bigr|_{t=0}\int_{x\in M}k(J_t(x),\alpha_t(x))\,\frac{\omega_0^n}{n!}=\\
&=\int_{x\in M}\frac{\mrm{d}}{\mrm{d}t}\Bigr|_{t=0}k(J_t(x),\alpha_t(x))\,\frac{\omega_0^n}{n!}=\int_{x\in M}\mrm{d}k_{(J(x),\alpha(x))}(v(x))\,\frac{\omega_0^n}{n!}.
\end{split}
\end{equation*}
Here the last equality holds since the matrix $v(x)$ associated to $v$ in a Darboux coordinate system around $x\in M$ is given by $\frac{\mrm{d}}{\mrm{d}t}\Bigr|_{t=0}(J_t(x),\alpha_t(x))$.
\end{nota}

\begin{proof}[Proof of Theorem \ref{teorema:differenziazione}.]
We spell out the proof for  $r=1$; the other cases are very similar. It will be convenient to introduce some additional notation: for $x\in M$ and a system of Darboux coordinates $\bm{u}$ around $x$, let $\Phi^x_{\bm{u}}$ be the map
\begin{equation*}
\begin{split}
\Phi^x_{\bm{u}}:\cotJ&\to T^*\!\m{AC}^+\\
(J,\alpha)&\mapsto (J(x),\alpha(x))
\end{split}
\end{equation*}
given by locally trivializing the fibre bundle over the coordinate system $\bm{u}$.

For $v\in T_p(\cotJ)$ a tangent vector on $\cotJ$, we can extend it to a vector field $V$ on an open neighbourhood of $p\in\cotJ$ in such a way that $V$ is \emph{constant} in a system of ``local coordinates'' for $\cotJ$. For the details about how to find local coordinates for $\cotJ$, see \cite[proof of Theorem $1.2$]{Koiso_complex_structure}. Moreover, this extension $V$ is such that $\mrm{d}\Phi^x_{\bm{u}}(V)$ is a vector field on $T^*\!\m{AC}^+(2n)$, itself constant in a system of coordinates for $T^*\!\m{AC}^+(2n)$.

Now fix $p\in\cotJ$,  $v,w\in T_p(\cotJ)$. If we extend $v, w$ to constant vectors $V,W$ as described in the previous paragraph, we can compute
\begin{equation*}
\mrm{d}K_p(v,w)=v_p(K(W))-w_p(K(V))-K([V,W])
\end{equation*}
however, $[V,W]=0$ since the vector fields are constant; for the other two terms we have, if $v=\diff_t\Bigr|_{t=0}p_t$:
\begin{equation*}
\begin{split}
v_p(K(W))=\frac{\mrm{d}}{\mrm{d}t}\Bigr|_{t=0}\int_{x\in M} k_{\Phi^x_{\bm{u}}(p_t)}\left(\left(\mrm{d}\Phi^x_{\bm{u}}\right)_{p_t}(W)\right)\,\frac{\omega_0^n}{n!}=\int_{x\in M} \left(\mrm{d}\Phi^x_{\bm{u}}\right)_p(v)\left(k(\mrm{d}\Phi^x_{\bm{u}}(W))\right)\,\frac{\omega_0^n}{n!}
\end{split}
\end{equation*}
so we find
\begin{equation*}
\begin{split}
\mrm{d}K_p(v,w)=&\int_{x\in M} \Big[\left(\mrm{d}\Phi^x_{\bm{u}}\right)_p(v)\left(k(\mrm{d}\Phi^x_{\bm{u}}(W))\right)-\left(\mrm{d}\Phi^x_{\bm{u}}\right)_p(w)\left(k(\mrm{d}\Phi^x_{\bm{u}}(V))\right)-\\
&\quad\quad\quad-k_{\Phi^x_{\bm{u}}(p)}\left(\left[\mrm{d}\Phi^x_{\bm{u}}(V),\mrm{d}\Phi^x_{\bm{u}}(W)\right]\right)\Big]\frac{\omega_0^n}{n!}=\\
=&\int_{x\in M} \mrm{d}k_{\Phi^x_{\bm{u}}(p)}\left(\mrm{d}\Phi^x_{\bm{u}}(v),\mrm{d}\Phi^x_{\bm{u}}(w)\right)\,\frac{\omega_0^n}{n!}.
\end{split}
\end{equation*}
\end{proof}

Another consequence of Theorem \ref{teorema:differenziazione} is that $\bm{\Theta}$ has a more natural description, and in particular it is well-defined, concluding the proof of Theorem \ref{teorema:HKThmIntro}.

\begin{lemma}\label{lemma:2canonica_J}
The $2$-form $\bm{\Theta}$ defined in equation \eqref{eq:2forma_canonica} is the canonical $2$-form of $\cotJ$.
\end{lemma}

\begin{proof}
We recall that for any manifold $X$, the tautological $1$ form $\tau_X$ is a $1$-form defined on the total space of $T^*X\xrightarrow{\pi}X$ by
\begin{equation*}
\begin{split}
T^*_{x,\alpha}X&\to\bb{R}\\
v&\mapsto\alpha(\pi_*v)
\end{split}
\end{equation*}
and is related to the canonical $2$-form $\theta_X$ of $T^*X$ by $\theta_X=-\mrm{d}\tau_X$. Denote simply by $\tau$ the tautological $1$-form of $T^*\!\m{AC}^+$, just as $\theta$ is the canonical $2$-form. Let also $\bm{\tau}$ be the tautological form of $\cotJ$. Then from the definitions it follows immediately that for any $(J,\alpha)\in\cotJ$ and $(A,\varphi)\in T_{(J,\alpha)}(\cotJ)$
\begin{equation*}
\bm{\tau}_{(J,\alpha)}((A,\varphi))=\alpha(A)=\int_{x\in M}\frac{1}{2}\mrm{Tr}(\alpha_x A_x^\transpose)\,\frac{\omega_0^n}{n!}=\int_{x\in M}\tau_{(J_x,\alpha_x)}(A_x,\varphi_x)\,\frac{\omega_0^n}{n!}.
\end{equation*}
By Theorem \ref{teorema:differenziazione} it's clear that this identity proves that $\bm{\Theta}=-\mrm{d}\bm{\tau}$.
\end{proof}

\subsection{The infinite-dimensional Hamiltonian action}\label{sec:moment_map}

Let $(M,J_0,\omega_0)$ be a compact K\"ahler manifold. In this Section we prove Theorem \ref{teorema:HamThmIntro}, showing that the action of $\m{G} =\mrm{Ham}(M,\omega_0)$ induced on $\cotJ$ from the action on $\scr{J}$ is Hamiltonian with respect to both the real symplectic form $\bm{\Omega}_{\bm{I}}$ and the complex symplectic form $\bm{\Theta}$. 

The group $\m{G}$ acts on $\scr{J}$ by pull-backs: more precisely, for $\varphi\in\m{G}$ and $J\in\scr{J}$ we define
\begin{equation*}
\varphi.J =(\varphi^{-1})^*J=\varphi_*\circ J\circ\varphi_*^{-1}.
\end{equation*}
Notice that, since elements $\varphi$ of $\m{G}$ preserve $\omega_0$, in any system of Darboux coordinates on $M$ the tensor $\varphi_*$ is given by a $\mrm{Sp}(2n)$-valued function. It follows that the action preserves the structures $\Omega$, $\bb{J}$ on $\scr{J}$. The action induced by $\m{G}$ on $\cotJ$ is given by
\begin{equation*}
\varphi.(J,\alpha)=\left((\varphi^{-1})^*J,(\varphi^{-1})^*\alpha\right)=\left(\varphi_*\circ J\circ\varphi_*^{-1},(\varphi^{-1})^*\circ\alpha\circ\varphi^*\right)
\end{equation*}
and again it preserves $\bm{\Theta}$, $\bm{I}$ and $\bm{\Omega}_{\bm{I}}$.

For a function $h\in\m{C}^\infty_0(M)=\mrm{Lie}(\m{G})$, the infinitesimal action of $h$ on $\cotJ$ is
\begin{equation}\label{eq:azione_infinitesima}
\hat{h}_{(J,\alpha)}=\left(\m{L}_{X_h}J,\m{L}_{X_h}\alpha\right)\in T_{(J,\alpha)}(\cotJ).
\end{equation}

First we recall a simple result that will be used to prove Theorem \ref{teorema:HamThmIntro}.

\begin{lemma}\label{lemma:moment_map_exact}
Let $G$ be a Lie group acting on the left on a manifold $X$, and assume that the action preserves a $1$-form $\chi$; let also $\eta=\mrm{d}\chi$. Then the map
\begin{equation*}
\begin{split}
X&\to Lie(G)^*\\
x&\mapsto\f{m}_x
\end{split}
\end{equation*}
defined by $\f{m}_x(a) =\chi_x(\hat{a}_x)$ satisfies
\begin{equation*}
\mrm{d}(\f{m}(a))=-\hat{a}\lrcorner\eta.
\end{equation*}
Moreover, $\f{m}$ is $G$-equivariant with respect to the action of $G$ on $X$ and the co-adjoint action on $\mrm{Lie}(G)^*$. In particular if $\eta$ is a symplectic form then $\f{m}$ is a moment map for $G\curvearrowright X$.
\end{lemma}

\begin{proof}
The first part is a simple consequence of Cartan's formula:
\begin{equation*}
0=\m{L}_{\hat{a}}\chi=\hat{a}\lrcorner\mrm{d}\chi+\mrm{d}(\hat{a}\lrcorner\chi)=\hat{a}\lrcorner\eta+\mrm{d}(\f{m}_x(a)).
\end{equation*}
As for the $G$-equivariance, fix $g\in G$ and $a\in\mrm{Lie}(G)$. Then for every $x\in X$ (here $\sigma$ denotes the left action $G\curvearrowright X$)
\begin{equation*}
\begin{split}
\f{m}_{g.x}(a)&=\chi_{g.x}(\hat{a}_{g.x})=\chi_{g.x}\left((\mrm{d}\sigma_g)_x\left(\widehat{\mrm{Ad}_{g^{-1}}(a)}_x\right)\right)=\left(\sigma_g^*\chi\right)_x\left(\widehat{\mrm{Ad}_{g^{-1}}(a)}_x\right)=\\ &=\chi_x\left(\widehat{\mrm{Ad}_{g^{-1}}(a)}_x\right)=\mrm{Ad}_{g^{-1}}^*\f{m}_x(a)
\end{split}
\end{equation*}
where we have used again the fact that $\chi$ is $G$-invariant.
\end{proof}

As a consequence, we obtain the following results for the action $\m{G}\curvearrowright\cotJ$.

\begin{lemma}\label{lemma:moment_map_theta}
The action $\m{G}\curvearrowright\cotJ$ is Hamiltonian with respect to the canonical symplectic form $\bm{\Theta}$; a moment map $\f{m}_{\bm{\Theta}}$ is given by
\begin{equation}\label{eq:moment_map_theta}
{\f{m}_{\bm{\Theta}}}_{(J,\alpha)}(h)=-\int_M\frac{1}{2}\mrm{Tr}(\alpha^\transpose\m{L}_{X_h}J)\,\frac{\omega_0^n}{n!}.
\end{equation}
\end{lemma}

\begin{proof}
Since $\bm{\Theta}=-\mrm{d}\bm{\tau}$ and $\m{G}$ preserves $\bm{\tau}$, we can apply Lemma \ref{lemma:moment_map_exact} to find that $-\bm{\tau}_{(J,\alpha)}(\hat{h})$ is a moment map for $\m{G}\curvearrowright(\cotJ,\bm{\Theta})$.
\end{proof}
Let us now consider the action with respect to the real symplectic form.
\begin{lemma}\label{lemma:mappa_momento_OmegaI}
The action $\m{G}\curvearrowright(\cotJ,\bm{\Omega}_{\bm{I}})$ is Hamiltonian; a moment map $\f{m}_{\bm{\Omega}_{\bm{I}}}$ is given by
\begin{equation*}
\f{m}_{\bm{\Omega}_{\bm{I}}} =\mu\circ\pi+\f{m}
\end{equation*}
where $\mu$ is the moment map for the action $\m{G}\curvearrowright(\scr{J},\Omega)$, $\pi:\cotJ\to J$ is the projection and $\f{m}:\cotJ\to\mrm{Lie}(\m{G})^*$ is defined by 
\begin{equation}\label{eq:moment_map_omega}
\f{m}_{(J,\alpha)}(h)=\int_{x\in M}\mrm{d}^c\rho_{(J(x),\alpha(x))}\left(\m{L}_{X_h}J,\m{L}_{X_h}\alpha\right)\,\frac{\omega_0^n}{n!}.
\end{equation}
\end{lemma}

With our choice of notation and conventions, the moment map $\mu$ for $\m{G}\curvearrowright(\scr{J},\Omega)$ is given by
\begin{equation*}
\mu(J)=2\,s(J)-2\,\hat{s}
\end{equation*}
where we are identifying $\m{C}^\infty_0(M)$ with its dual via the usual $L^2$ pairing on functions. For a proof of this result, see \cite{Donaldson_scalar}, \cite[chapter $4$]{Tian_libro}, \cite{Fujiki_moduli}, \cite[section $6.1$]{Szekelyhidi_libro} and \cite[Proposition $2.2.1$]{Szekelyhidi_phd}. (Note that there are various different sign conventions, as well as different conventions with the pairings involved).

\begin{proof}[Proof of Lemma \ref{lemma:mappa_momento_OmegaI}]
Since $\bm{\Omega}_{\bm{I}}=\pi^*\Omega+\int_M\mrm{d}\mrm{d}^c\rho\,\frac{\omega_0^n}{n!}$, the first step is to show that for all $h\in\m{C}^\infty_0$
\begin{equation*}
\mrm{d}(\f{m}(h))=-\hat{h}\lrcorner\int_M\mrm{d}\mrm{d}^c\rho\,\frac{\omega_0^n}{n!}.
\end{equation*}
To prove this, we can use Lemma \ref{lemma:moment_map_exact} and Theorem \ref{teorema:differenziazione}. Indeed, if we define $\chi =\int_M\mrm{d}^c\rho\,\frac{\omega_0^n}{n!}$ then $\mrm{d}\chi=\int_M\mrm{d}\mrm{d}^c\rho\,\frac{\omega_0^n}{n!}$. We already saw that the action of $\m{G}$ preserves $\chi$, and so Lemma \ref{lemma:moment_map_exact} tells us that $\f{m}$ defined by
\begin{equation*}
\f{m}_{(J,\alpha)}(h)=\int_{x\in M}\mrm{d}^c\rho_{(J(x),\alpha(x))}\left(\m{L}_{X_h}J,\m{L}_{X_h}\alpha\right)\frac{\omega_0^n}{n!}
\end{equation*}
has the properties we need.
\end{proof}

The results of Lemma \ref{lemma:moment_map_theta} and Lemma \ref{lemma:mappa_momento_OmegaI} conclude the proof of Theorem \ref{teorema:HamThmIntro}.

Clearly one would like to obtain more explicit expressions for the moment maps under the natural $L^2$ pairing. This is not too difficult for the \emph{complex} moment map, at least if $J$ is integrable. 
\begin{lemma} Suppose $J$ is integrable. Then we have
\begin{equation*}
{\f{m}_{\bm{\Theta}}}_{(J,\alpha)}(h) = \left\langle h,-\mrm{div}\left(\bdiff^*\bar{\alpha}^\transpose\right)\right\rangle.
\end{equation*}
\end{lemma}
\begin{proof} We compute
\begin{equation*}
\begin{split}
{\f{m}_{\bm{\Theta}}}_{(J,\alpha)}(h)&=-\frac{1}{2}\int_M\alpha\indices{_a^{\bar{b}}}\left(\m{L}_{X_h}J\right)\indices{^a_{\bar{b}}}\,\frac{\omega_0^n}{n!}=\I\int_M\alpha\indices{_a^{\bar{b}}}\diff_{\bar{b}}(X_h)\indices{^a}\,\frac{\omega_0^n}{n!}=
-\I\int_M\alpha\indices{_a^{\bar{b}}}\nabla_{\bar{b}}(X_h)\indices{^a}\,\frac{\omega_0^n}{n!}=\\
&=-\I\int_M\nabla_{\bar{b}}\left(\alpha\indices{_a^{\bar{b}}}(X_h)\indices{^a}\right)\frac{\omega_0^n}{n!}+\I\int_M(X_h)\indices{^a}\nabla_{\bar{b}}\alpha\indices{_a^{\bar{b}}}\,\frac{\omega_0^n}{n!}=
-\int_Mg^{a\bar{c}}\nabla_{\bar{c}}h\,\nabla_{\bar{b}}\alpha\indices{_a^{\bar{b}}}\,\frac{\omega_0^n}{n!}=\\
&=\int_M h\,g^{a\bar{b}}\nabla_{\bar{c}}\nabla_{\bar{b}}\alpha\indices{_a^{\bar{c}}}\,\frac{\omega_0^n}{n!}=\left\langle h,-\mrm{div}\left({\nabla^{0,1}}^*\bar{\alpha}^\transpose\right)\right\rangle.
\end{split}
\end{equation*}
\end{proof}
Unfortunately it is more difficult to obtain an explicit expression for the \emph{real} moment map. We will do this for complex curves and surfaces in the next sections.

\paragraph*{The Biquard-Gauduchon function for $T^*\!\m{AC}^+(2n)$.} An explicit expression for the real moment map requires an explicit expression for the Biquard-Gauduchon function $\rho$.

For a fixed vector field $A$ on $T\m{AC}^+(2n)$ we consider the endomorphism $\Xi(A)$ of $T\m{AC}^+(2n)$ defined by
\begin{equation*}
\Xi(A):B\mapsto -\bb{J}\left(R(\bb{J}A,A)(B)\right).
\end{equation*}
By Proposition \ref{prop:curvatura_AC}, at the point $-\Omega_0\in\m{AC}^+(2n)$, $\Xi(A)$ can be written as:
\begin{equation}\label{eq:da_diagonalizzare}
\Xi_{-\Omega_0}(A)(B)=\Omega_0\left(-\frac{1}{4}\Big[[-\Omega_0\,A,A],B\Big]\right)=-\frac{1}{2}\left(A^2B+B\,A^2\right).
\end{equation}
Then the Biquard-Gauduchon function at $A$ is a spectral function of the endomorphism $\Xi(A)$. 

\subsubsection{The complexified action}\label{ComplexifiedSection}

The classical Kempf-Ness Theorem in Geometric Invariant Theory characterises orbits of a complexified Hamiltonian action containing a zero of the moment map in terms of algebro-geometric stability. This is not applicable in our situation, in particular since the group of Hamiltonian symplectomorphisms does not admit a complexification (see e.g. the discussion in \cite[Remark $35$]{Wang_Futaki} and \cite[\S$1.3.3$]{GarciaFernandez_PHD}). However, there is a way to circumvent this, at least at the formal level. The general idea is well-known and goes back to \cite{Donaldson_scalar}: while a complexification of the group $\m{G}$ does not exist, we can find subvarieties of $\scr{J}$ that play the role of \emph{complexified orbits} for the action of $\m{G}$.

The same result holds also in our case: the infinitesimal action of $h\in\mrm{Lie}(\m{G})=\m{C}^{\infty}(M,\bb{R})$ on $\cotJ$ is
\begin{equation*}
\hat{h}_{J,\alpha}=\left(\m{L}_{X_h}J,\m{L}_{X_h}\alpha\right)
\end{equation*}
and since $\cotJ$ has a complex structure $\bm{I}$ we can infinitesimally complexify the action of $\m{G}$ by setting for $h\in\mrm{Lie}(\m{G})^{\bb{C}}=\m{C}^\infty(M,\bb{C})$
\begin{equation*}
\widehat{h}_{J,\alpha}:=\widehat{\mrm{Re}(h)}_{J,\alpha}+\bm{I}_{J,\alpha}\widehat{\mrm{Im}(h)}_{J,\alpha}
\end{equation*}
So we can define a distribution on $\cotJ$ that is a complexification of the infinitesimal action
\begin{equation*}
\begin{split}
\scr{D}_{J,\alpha}=&\set*{\hat{h}_{J,\alpha}\tc h\in\m{C}^\infty(M,\bb{R})}\cup\set*{\widehat{\I h}_{J,\alpha}\tc h\in\m{C}^\infty(M,\bb{R})}=\\
=&\set*{\left(\m{L}_{X_h}J,\m{L}_{X_h}\alpha\right)\tc h\in\m{C}^\infty(M,\bb{R})}\cup\set*{\left(J\m{L}_{X_h}J,(\m{L}_{X_h}\alpha)J^\transpose+(\m{L}_{X_h}J^\transpose)\alpha\right)\tc h\in\m{C}^\infty(M,\bb{R})}.
\end{split}
\end{equation*}
It is still possible to show that this distribution is integrable, at least for an integrable pair $(J,\alpha)$, and we can consider its integral leaves at a point $(J,\alpha)\in\cotJ$ as the complexified orbit of $(J,\alpha)$. The final result is that for $(J,\alpha)\in\cotJ$ one can construct a map from the K\"ahler class of $\omega_0$ to $\cotJ$ that describes the complexified orbit of $(J,\alpha)$. We omit the details here, as this map will not be used in the present paper. 

\paragraph{The complexified equations.} Following the classical case of the cscK equation, the ``formal complexification'' of the orbits of $\m{G}\curvearrowright\cotJ$ makes it natural to regard our system
\begin{equation}\label{eq:HCSCK_complex_relaxed}
\begin{cases}
\f{m}_{\bm{\Omega}_{\bm{I}}}(\omega,J,\alpha)=0\\
\f{m}_{\bm{\Theta}}(\omega,J,\alpha)=0
\end{cases}
\end{equation}
as equations for a deformation $\alpha$ of the complex structure and a form $\omega$, to be found in the K\"ahler class of $\omega_0$, keeping instead the complex structure $J$ fixed. Of course, we have the additional condition that $\alpha$ and $\omega$ should be compatible, i.e. $\omega_0(\alpha^\transpose-,J-)+\omega_0(J-,\alpha^\transpose-)=0$.

\section{The case of complex curves}\label{sec:curva}

In this Section we examine the moment map equations when the base manifold $M$ is a Riemann surface, proving Theorem \ref{teorema:CurveThmIntro}. We will also recover Donaldson's equations (see \cite{Donaldson_hyperkahler}) by considering the complexified equations \eqref{eq:HCSCK_complex_relaxed}.

\subsection{The space $T^*\!\m{AC}^+(2)$.}

In the special case when $n = 1$, an element of $\m{AC}^+$ is a matrix $J=\begin{pmatrix}a &b \\ c& -a\end{pmatrix}$ of determinant $1$, and a tangent vector in $T_J\m{AC}^+(2)$ is a matrix $A=\begin{pmatrix} A_1 & A_2\\ A_3 &-A_1\end{pmatrix}$ with
\begin{equation*}
A_2=\frac{1+a^2}{c^2}A_3-2\frac{a}{c}A_1.
\end{equation*}
In particular that $A^2=-\mrm{det}(A)\mrm{Id}$, so that $\norm{A}^2=\frac{1}{2}\mrm{Tr}(A^2)=-\mrm{det}(A)$. Then in the $n=1$ case the map $\Xi(A)$ of equation \eqref{eq:da_diagonalizzare} becomes
\begin{equation*}
\begin{split}
T_{-\Omega_0}\m{AC}^+(2)&\to T_{-\Omega_0}\m{AC}^+(2)\\
B&\mapsto\mrm{det}(A)\,B=-\norm{A}^2\,B
\end{split}
\end{equation*}
so it is simply a scalar map, with spectrum $\set*{-\norm{A}^2}$.

We can use this map to find the Biquard-Gauduchon form $\mrm{d}^c\rho$ on ${T^{1,0}}^*\m{AC}^+(2)$; recall however that we have to consider the space $T\m{AC}^+(2)$, that is isomorphic to ${T^{1,0}}^*\m{AC}^+(2)$ under the map
\begin{equation*}
\begin{split}
{T^{1,0}}^*\m{AC}^+(2)&\to T\m{AC}^+(2)\\
\alpha&\mapsto\mrm{Re}(\alpha)^\transpose.
\end{split}
\end{equation*}
The Biquard-Gauduchon function on $T\m{AC}^*$ is $\rho(J,A)=\langle f(-J\,R(JA,A))A,A\rangle$, where we are using the canonical metric on $\m{AC}^+(2)$ (induced from the Poincar\'e upper half plane) and $f$ is defined by equation \eqref{eq:funzione_f}. We have just seen that
\begin{equation*}
-J\,R(JA,A)=\mrm{det}(A)\cdot\mrm{Id}
\end{equation*}
so
\begin{equation*}
f(-J\,R(JA,A))=\left(\frac{1}{\mrm{det}(A)}\left(\sqrt{1+\mrm{det}(A)}-1-\log\left(\frac{1+\sqrt{1+\mrm{det}(A)}}{2}\right)\right)\right)\cdot\mrm{Id}
\end{equation*}
and the Biquard-Gauduchon function is
\begin{equation*}
\rho(J,A)=\langle f(-J\,R(JA,A))A,A\rangle=1-\sqrt{1+\mrm{det}(A)}+\log\left(\frac{1+\sqrt{1+\mrm{det}(A)}}{2}\right).
\end{equation*}

Consider now a tangent vector $V\in T_{(J,A)}(T\m{AC}^+(2))$, $V=(\dot{J}_0,\dot{A}_0)$. According to our previous computations, the differential of $\rho$ acts on $V$ as
\begin{equation}\label{eq:diff_rho_curva}
\begin{split}
\mrm{d}\rho_{(J,A)}(V)&=-\frac{1}{2}\frac{\diff_t\mrm{det}(A_t)}{1+\sqrt{1+\mrm{det}(A_0)}}=-\frac{1}{2}\frac{\mrm{Tr}(\dot{A}_0\mbox{adj}(A_0))}{1+\sqrt{1+\mrm{det}(A_0)}}=\frac{1}{2}\frac{\mrm{Tr}(\dot{A}_0A_0)}{1+\sqrt{1+\mrm{det}(A_0)}}
\end{split}
\end{equation}
where we used Jacobi's formula for the derivative of the determinant in terms of the \emph{adjugate} endomorphism.

Some remarks on this object, $\mrm{adj}(A)$, are in order: for matrices $M_1,M_2$ we have $\mrm{adj}(M_1M_2)=\mrm{adj}(M_2)\mrm{adj}(M_1)$. Moreover, for an invertible matrix $G$, $\mrm{adj}(G)=\mrm{det}(G)\,G^{-1}$. Then
\begin{equation*}
\mrm{adj}(G\,A\,G^{-1})=G\,\mrm{adj}(A)G^{-1}
\end{equation*}
and this means that, if $A\in\Gamma(M,\mrm{End}(TM))$, $\mrm{adj}(A)$ is a well-defined section of $\mrm{End}(TM)$.
\begin{lemma}\label{lemma:aggiunto_TJ}
If $A\in T_J\scr{J}$, then also $\mrm{adj}(A)$ belongs to $T_J\scr{J}$.
\end{lemma}

\begin{proof}
We have to check that $\mrm{adj}(A)J+J\,\mrm{adj}(A)=0$ and that $g_J(\mrm{adj}(A)-,-)$ is a symmetric bilinear form. The first identity can be obtained as follows, recalling that $\mrm{adj}(J)=-J$:
\begin{equation*}
J\,\mrm{adj}(A)=-\mrm{adj}(AJ)=\mrm{adj}(JA)=-\mrm{adj}(A)J.
\end{equation*}
The second identity can be checked pointwise: fix $p\in M$, and choose a local coordinate system $\bm{x}$ around $p$ such that $g_J(p)$ in this coordinate system is the standard Euclidean product. Abusing notation let $A$ be the matrix associated to $A(p)$ in the coordinate system $\bm{x}$; then $A$ is a symmetric matrix, since $g_J(A-,-)$ is symmetric. But then
\begin{equation*}
\mrm{adj}(A)^\transpose=\mrm{adj}(A^\transpose)=\mrm{adj}(A)
\end{equation*}
and so $g_J(\mrm{adj}(A)-,-)$ is also a symmetric matrix, at the point $p$.
\end{proof}

\subsection{The real moment map for a curve}\label{section:moment_map_curve}

The expression for $\mrm{d}\rho$ on $T\m{AC}^+(2)$ that we just computed allows to rewrite the implicit definition of $\f{m}$ in equation \eqref{eq:moment_map_omega} as
\begin{equation*}
\f{m}_{(J,\alpha)}(h)=\int_{x\in M}\mrm{d}^c\rho_{(J(x),\alpha(x))}\left(\m{L}_{X_h}J,\m{L}_{X_h}\alpha\right)\,\frac{\omega_0^n}{n!}.
\end{equation*}
However, as was remarked earlier, under our identifications we should compute
\begin{equation*}
\mrm{d}^c\rho_{J,\mrm{Re}(\alpha)^\transpose}\left(\m{L}_{X_h}J,\mrm{Re}\left(\m{L}_{X_h}\alpha\right)^\transpose\right)=\mrm{d}\rho_{J,\mrm{Re}(\alpha)^\transpose}\left(-J\m{L}_{X_h}J,\mrm{Re}\left((\m{L}_{X_h}J)^\transpose\alpha+(\m{L}_{X_h}\alpha)J^\transpose\right)^\transpose\right)
\end{equation*}
with $\mrm{d}\rho$ as in \eqref{eq:diff_rho_curva}, since the identification between $T\m{AC}^+$ and ${T^{1,0}}^*\!\!\m{AC}^+$ is conjugate-linear in the second component. It is more convenient to write $A=\mrm{Re}(\alpha)^\transpose\in T_J\scr{J}$, so that equation \eqref{eq:diff_rho_curva} gives
\begin{equation*}
\mrm{d}\rho_{J,\mrm{Re}(\alpha)^\transpose}\left(-J\m{L}_{X_h}J,\mrm{Re}\left((\m{L}_{X_h}J)^\transpose\alpha+(\m{L}_{X_h}\alpha)J^\transpose\right)^\transpose\right)=\frac{1}{2}\frac{\mrm{Tr}\left(A(\m{L}_{X_h}J)\,A\right)+\mrm{Tr}(J(\m{L}_{X_h}A)\,A)}{1+\sqrt{1+\mrm{det}(A)}}
\end{equation*}
Since $\m{L}_{X_h}J\in T_J\scr{J}$, Remark \ref{nota:traccia_tripla} implies
\begin{equation*}
\mrm{Tr}\left(A(\m{L}_{X_h}J)\,A\right)+\mrm{Tr}(J(\m{L}_{X_h}A)\,A)=-\mrm{Tr}((\m{L}_{X_h}A)\,J\,A)
\end{equation*}
and we can write
\begin{equation}\label{eq:diff_rho_dim2}
\f{m}_{(J,\alpha)}(h)=-\frac{1}{2}\int_M\frac{\mrm{Tr}((\m{L}_{X_h}A) JA)}{1+\sqrt{1+\mrm{det}(A)}}\,\frac{\omega_0^n}{n!}.
\end{equation}
If we could write this expression in the form $\int_M f\,h\,\frac{\omega_0^n}{n!}$ for some function $f$, then we could use the $L^2$ pairing of $\m{C}^\infty_0(M)$ to identify $\f{m}$ with $f-\int_Mf$. To get this result, consider the function
\begin{equation}\label{eq:F_F1F2F3F4}
\begin{split}
F:\m{C}^\infty_0(M)&\to\m{C}^\infty_0(M)\\
h&\mapsto\mrm{Tr}((\m{L}_{X_h}A) JA).
\end{split}
\end{equation}
We have $\f{m}_{(J,\alpha)}=-\frac{1}{2}\left\langle\frac{1}{1+\sqrt{1+\mrm{det}(A)}},F(h)\right\rangle$; so if we can find a formal adjoint $F^*$ of $F$ with respect to the $L^2$ pairing, we could write $\f{m}=-\frac{1}{2}F^*\left(\frac{1}{1+\sqrt{1+\mrm{det}(A)}}\right)$. Notice that we can write $F$ as a composition $F=F_3\circ F_2\circ F_1$, with
\begin{align*}
F_1:\m{C}^\infty_0(M)&\to\Gamma(M,TM)\\
h&\mapsto X_h\\
F_2:\Gamma(M,TM)&\to\Gamma(M,\mrm{End}(TM))\\
X&\mapsto\m{L}_XA\\
F_3:\Gamma(M,\mrm{End}(TM))&\to\m{C}^\infty_0(M)\\
P&\mapsto \mrm{Tr}(PJA).
\end{align*}
Moreover the formal adjoints of $F_1$ and $F_3$ with respect to the pairing induced by the metric $g_J:=\omega_0(-,J-)$ are given explicitly by
\begin{align*}
F_1^*(X)=\mrm{div}(JX);\\
F_3^*(f)=-f\,AJ.
\end{align*}
It remains to compute the formal adjoint of $F_2$.

\begin{lemma}\label{lemma:aggiunto_F2}
For any $Q\in\Gamma(\mrm{End}(TM))$, $X\in\Gamma(TM)$ and $A\in T_J\scr{J}$ we have
\begin{equation*}
\langle\m{L}_XA,Q\rangle=\langle Q,\nabla_XA\rangle+\langle AQ-QA,\nabla X\rangle.
\end{equation*}
\end{lemma}
Here the pairings and the connection are those defined by the metric $g_J$.
\begin{proof}
Fix an element $Q$ of $\Gamma(\mrm{End}(TM))$, and consider the product
\begin{equation}\label{eq:pairing_LX}
g_J(\m{L}_XA,Q)=g\indices{^i^j}g\indices{_k_l}Q\indices{^l_j}X\indices{^m}\diff\indices{_m}A\indices{^k_i}-g\indices{^i^j}g\indices{_k_l}Q\indices{^l_j}A\indices{^m_i}\diff\indices{_m}X\indices{^k}+g\indices{^i^j}g\indices{_k_l}Q\indices{^l_j}A\indices{^k_m}\diff\indices{_i}X\indices{^m}.
\end{equation}
We can exchange the usual derivatives with covariant derivatives (using the Levi-Civita connection of $g_J$), but we have to introduce Christoffel symbols; the proof consists in showing that the sum of all the terms that must be introduced in fact vanishes, and this is done recalling that $g_J(-,A-)$ is symmetric (cf. equation \eqref{eq:indentita_tangenti}).

The first right hand side term of equation \eqref{eq:pairing_LX} can then be written as
\begin{equation}\label{eq:LX_1}
\begin{split}
g\indices{^i^j}g\indices{_k_l}Q\indices{^l_j}X\indices{^m}\diff\indices{_m}A\indices{^k_i}&=g\indices{^i^j}g\indices{_k_l}Q\indices{^l_j}X\indices{^m}\nabla\indices{_m}A\indices{^k_i}-g\indices{^i^j}g\indices{_k_l}Q\indices{^l_j}X\indices{^m}A\indices{^p_i}\Gamma\indices{^k_m_p}+g\indices{^i^j}g\indices{_k_l}Q\indices{^l_j}X\indices{^m}A\indices{^k_q}\Gamma\indices{^q_m_i}=\\
&=g\indices{^i^j}g\indices{_k_l}Q\indices{^l_j}X\indices{^m}\nabla\indices{_m}A\indices{^k_i}-g\indices{^i^p}g\indices{_k_l}Q\indices{^l_j}X\indices{^m}A\indices{^j_i}\Gamma\indices{^k_m_p}+g\indices{^i^j}g\indices{_k_q}Q\indices{^l_j}X\indices{^m}A\indices{^k_l}\Gamma\indices{^q_m_i}
\end{split}
\end{equation}
while the other two terms become
\begin{equation}\label{eq:LX_2}
\begin{split}
-g\indices{^i^j}g\indices{_k_l}Q\indices{^l_j}A\indices{^m_i}\diff\indices{_m}X\indices{^k}&=-g\indices{^i^j}g\indices{_k_l}Q\indices{^l_j}A\indices{^m_i}\nabla\indices{_m}X\indices{^k}+g\indices{^i^j}g\indices{_k_l}Q\indices{^l_j}A\indices{^m_i}X\indices{^p}\Gamma\indices{^k_p_m}=\\
&=-g\indices{^i^j}g\indices{_k_l}Q\indices{^l_j}A\indices{^m_i}\nabla\indices{_m}X\indices{^k}+g\indices{^i^m}g\indices{_k_l}Q\indices{^l_j}A\indices{^j_i}X\indices{^p}\Gamma\indices{^k_p_m}
\end{split}
\end{equation}
\begin{equation}\label{eq:LX_3}
\begin{split}
g\indices{^i^j}g\indices{_k_l}Q\indices{^l_j}A\indices{^k_m}\diff\indices{_i}X\indices{^m}&=g\indices{^i^j}g\indices{_k_l}Q\indices{^l_j}A\indices{^k_m}\nabla\indices{_i}X\indices{^m}-g\indices{^i^j}g\indices{_k_l}Q\indices{^l_j}A\indices{^k_m}X\indices{^p}\Gamma\indices{^m_i_p}=\\ &=g\indices{^i^j}g\indices{_k_l}Q\indices{^l_j}A\indices{^k_m}\nabla\indices{_i}X\indices{^m}-g\indices{^i^j}g\indices{_k_m}Q\indices{^l_j}A\indices{^k_l}X\indices{^p}\Gamma\indices{^m_i_p}
\end{split}
\end{equation}
and adding up equations \eqref{eq:LX_1}, \eqref{eq:LX_2} and \eqref{eq:LX_3} we find
\begin{equation*}
\begin{split}
g_J(\m{L}_XA,Q)&=g\indices{^i^j}g\indices{_k_l}Q\indices{^l_j}X\indices{^m}\nabla\indices{_m}A\indices{^k_i}-g\indices{^i^j}g\indices{_k_l}Q\indices{^l_j}A\indices{^m_i}\nabla\indices{_m}X\indices{^k}+g\indices{^i^j}g\indices{_k_l}Q\indices{^l_j}A\indices{^k_m}\nabla\indices{_i}X\indices{^m}=\\
&=g_J(Q,\nabla_XA)-g_J(QA,\nabla X)+g_J(AQ,\nabla X).
\end{split}
\end{equation*}
\end{proof}

\begin{corollario}\label{corollario:aggiunto_F2}
The formal adjoint of $F_2$ is
\begin{equation*}
\begin{split}
F_2^*:\Gamma(\mrm{End}(TM))&\to\Gamma(TM)\\
Q&\mapsto C^2_1\left((\nabla A)Q\right)^\sharp+\nabla^*([A,Q]).
\end{split}
\end{equation*}
\end{corollario}

Here $\nabla^*$ is the formal adjoint of $\nabla$, $\nabla^*Q=-g\indices{^i^j}\nabla\indices{_i}Q\indices{^k_j}\diff\indices{_k}$, while $C^2_1$ denotes the contraction of the first lower index with the second upper index. More explicitly
\begin{equation*}
C^2_1\left((\nabla A)Q\right)^\sharp=g\indices{^n^p}Q\indices{^i_j}\nabla\indices{_p}A\indices{^j_i}\diff\indices{_n}.
\end{equation*}

We are finally in a good position to write the moment map and prove Theorem \ref{teorema:CurveThmIntro}. For notational convenience, we introduce the function
\begin{equation*}
\psi:=\frac{1}{1+\sqrt{1+\mrm{det}(A)}}.
\end{equation*}
Our computations so far show
\begin{equation*}
\begin{split}
\f{m}_{(J,\alpha)}(h)&=\left\langle\psi,-\frac{1}{2}\mrm{Tr}\left((\m{L}_{X_h}A)AJ\right)\right\rangle=\left\langle-\frac{1}{2}F_1^*F_2^*F_3^*(\psi),h\right\rangle=\\
&=\left\langle\frac{1}{2}\mrm{div}\left[\psi\,J\,C^2_1\left((\nabla A)AJ\right)^\sharp+2\,J\nabla^*(\psi A^2J)\right],h\right\rangle
\end{split}
\end{equation*}
so we can identify the function $\f{m}$, using the $L^2$-pairing, with
\begin{equation}\label{eq:momento_Omega_dim2_A}
\f{m}(J,\alpha)=\mrm{div}\left[\frac{\psi}{2}\,J\,C^2_1\left((\nabla A)AJ\right)^\sharp+J\nabla^*(\psi A^2J)\right].
\end{equation}
Notice that this expression implies already that $\f{m}_{(J,\alpha)}$ is a zero-average function, as we expected. But we can make further simplifications. First, recall that $A^2=-\mrm{det}(A)\mrm{Id}$. Since $A=\mrm{Re}(\alpha^\transpose)=\frac{\bar{\alpha}^\transpose+\alpha^\transpose}{2}$ we have $A^{0,1}=\frac{1}{2}\alpha^\transpose$ and $\mrm{det}(A)=-\frac{1}{2}\mrm{Tr}(A^2)=-\norm{A^{1,0}}_{g_J}^2=-\frac{1}{4}\norm{\alpha}_{g_J}^2$, so that $A^2=\frac{1}{4}\norm{\alpha}_{g_J}^2\,\mrm{Id}$.
Then we have (all metric quantities are computed w.r.t. $g_J$)
\begin{equation*}
J\nabla^*(\psi A^2J)=-\nabla^*\left(\frac{1}{4}\,\psi\,\norm{\alpha}^2\,\mrm{Id}\right)=\mrm{grad}\left(\frac{1}{4}\,\psi\,\norm{\alpha}^2\right)
\end{equation*}
since $J$ is integrable and $g_J$ is a K\"ahler metric. This shows
\begin{equation*}
\f{m}(J,\alpha)=\frac{1}{2}\mrm{div}\left(\psi J\left(C\indices{^2_1}\left((\nabla A)AJ\right)^\sharp\right)+2\,\mrm{grad}\left(\frac{1}{4}\,\psi\,\norm{\alpha}^2\right)\right).
\end{equation*}
Fix holomorphic coordinates with respect to $J$. Then
\begin{equation*}
\begin{split}
J\left(C^2_1\,\left((\nabla A)AJ\right)^\sharp\right)=&-\mrm{grad}\left(\mrm{Tr}(A^{1,0}A^{0,1})\right)+2\left(\mrm{Tr}(\nabla^aA^{0,1}\,A^{1,0})\diff_a+\mrm{Tr}(\nabla^{\bar{b}}A^{1,0}\,A^{0,1})\diff_{\bar{b}}\right)=\\
=&-\mrm{grad}\left(\frac{1}{4}\norm{\alpha}^2\right)+2\left(\frac{1}{4}g(\nabla^a\alpha,\bar{\alpha})\diff_a+\frac{1}{4}g(\nabla^{\bar{b}}\bar{\alpha},\alpha)\diff_{\bar{b}}\right)
\end{split}
\end{equation*}
so we can rewrite everything as
\begin{equation}\label{eq:mappa_momento_RS_parziale}
\begin{split}
\f{m}(J,\alpha)=&\frac{1}{2}\mrm{div}\left[-\psi\,\mrm{grad}\left(\frac{1}{4}\norm{\alpha}^2\right)+2\,\psi\left(\frac{1}{4}g(\nabla^a\alpha,\bar{\alpha})\diff_a+\frac{1}{4}g(\nabla^{\bar{b}}\bar{\alpha},\alpha)\diff_{\bar{b}}\right)+2\,\mrm{grad}\left(\frac{1}{4}\,\psi\,\norm{\alpha}^2\right)\right].
\end{split}
\end{equation}
Notice that
\begin{equation*}
-\psi\,\mrm{grad}\left(\frac{1}{4}\norm{\alpha}^2\right)+2\,\mrm{grad}\left(\frac{1}{4}\,\psi\,\norm{\alpha}^2\right)=-2\,\mrm{grad}\left(\mrm{log}\left(1+\sqrt{1-\frac{1}{4}\norm{\alpha}^2}\right)\right)
\end{equation*}
so that
\begin{equation}\label{eq:mappa_mom_Log}
\begin{split}
\f{m}(J,\alpha)=&\Delta\left(\mrm{log}\left(1+\sqrt{1-\frac{1}{4}\norm{\alpha}^2}\right)\right)+\mrm{div}\left[\psi\left(\frac{1}{4}g(\nabla^a\alpha,\bar{\alpha})\diff_a+\frac{1}{4}g(\nabla^{\bar{b}}\bar{\alpha},\alpha)\diff_{\bar{b}}\right)\right]
\end{split}
\end{equation}
The complete expression for the moment map relative to $\bm{\Omega}_{\bm{I}}$ is, according to Lemma \ref{eq:moment_map_omega}:
\begin{equation}\label{eq:momento_omega_alpha}
\begin{split}
\f{m}_{\bm{\Omega}_{\bm{I}}}(J,\alpha)=2\,s(J)-2\,\hat{s}+\Delta\left(\mrm{log}\left(1+\sqrt{1-\frac{1}{4}\norm{\alpha}^2}\right)\right)+\mrm{div}\left(\psi\,Q(J,\alpha)\right)
\end{split}
\end{equation}
where $Q(J,\alpha)$ is the vector field on $M$ defined by
\begin{equation*}
Q(J,\alpha):=\frac{1}{4}g(\nabla^a\alpha,\bar{\alpha})\diff_a+\frac{1}{4}g(\nabla^{\bar{b}}\bar{\alpha},\alpha)\diff_{\bar{b}}.
\end{equation*}

\subsection{Equations for a conformal potential}

We turn now to the complexified system of equations, \eqref{eq:HCSCK_complex_relaxed}. Notice that in dimension $1$ the compatibility between $\omega$ and $\alpha$ is a vacuous condition, since $\alpha\indices{_1^{\bar{1}}}g_{1\bar{1}}$ is certainly symmetric. From now on we fix the complex structure $J$ and a K\"ahler class $[\omega_0]$, and look for a metric $\omega\in[\omega_0]$ and a ``Higgs field'' $\alpha\in\mrm{Hom}({T^{0,1}}^*,{T^{1,0}}^*)$ such that $(\omega,\alpha)$ satisfy the following system of equations
\begin{equation}\label{eq:equazioni_curva_alpha}
\begin{cases}\frac{1}{4}\norm{\alpha}^2<1;\\
\mrm{div}(\bdiff^*\bar{\alpha}^\transpose)=0;\\
2\,s(\omega)-2\,\hat{s}+\Delta\left(\mrm{log}\left(1+\sqrt{1-\frac{1}{4}\norm{\alpha}^2}\right)\right)+\mrm{div}\left(\psi\,Q(J,\alpha)\right)=0,
\end{cases}
\end{equation}
where all the metric quantities are computed from the metric defined by $\omega$ and $J$. It is more convenient to write the equations in \eqref{eq:equazioni_curva_alpha} not in terms of $\alpha$ but rather in terms of the \emph{quadratic differential} $\tau$ defined by
\begin{equation*}
\tau:=\frac{1}{2}\alpha\indices{_a^{\bar{b}}}\,g_{\bar{b}c}\,\mrm{d}z^a\odot\mrm{d}z^c;
\end{equation*}
using this object, equations \eqref{eq:equazioni_curva_alpha} become
\begin{equation}\label{eq:equazioni_curva}
\begin{cases}
\norm{\tau}^2<1;\\
\mrm{div}({\nabla^{1,0}}^*\tau)^\sharp=0;\\
2\,s(\omega)-2\,\hat{s}+\Delta\left(\mrm{log}\left(1+\sqrt{1-\norm{\tau}^2}\right)\right)+\mrm{div}\left(\psi\,Q(\omega,\tau)\right)=0.
\end{cases}
\end{equation}
We can make the the second equation in \eqref{eq:equazioni_curva} more explicit by using holomorphic local coordinates (with respect to the fixed complex structure); recall that we are working on a Riemann surface, so we just have one index, when working in coordinates:
\begin{equation*}
\mrm{div}({\nabla^{1,0}}^*\tau)^\sharp=-g^{1\bar{1}}\diff_1\left(g^{1\bar{1}}\diff_{\bar{1}}\tau_{11}\right)
\end{equation*}
and so the second equation in \eqref{eq:equazioni_curva} is certainly satisfied when $\tau$ is a \emph{holomorphic} quadratic differential; the space of such objects has dimension $3\,g(M)-3$, so if $g(M)>1$ we are sure that there are holomorphic quadratic differentials. Notice that, while the second equation in \eqref{eq:equazioni_curva} depends on the choice of $\omega$ in the fixed K\"ahler class, the simpler condition ``$\tau$ is holomorphic'' does not; then our equations can be satisfied if we are able to show that the following equation has solutions, for a small enough holomorphic quadratic differential $\tau$
\begin{equation*}\label{eq:reale_curva}
2\,s(\omega)-2\,\hat{s}+\Delta\left(\mrm{log}\left(1+\sqrt{1-\norm{\tau}^2}\right)\right)+\mrm{div}\left(\psi\,Q(\omega,\tau)\right)=0.
\end{equation*}
Notice however that, under the assumption that $\tau$ is a holomorphic quadratic differential, we can simplify this equation, since $Q(f,\tau)=0$. Indeed
\begin{equation*}
g(\nabla^a\tau,\bar{\tau})\diff_a=g^{a\bar{b}}g^{c\bar{e}}g^{d\bar{f}}\,\nabla_{\bar{b}}\tau_{cd}\,\tau_{\bar{e}\bar{f}}\,\diff_a=0.
\end{equation*}
So the complexified moment map equation becomes
\begin{equation}\label{eq:F_tau_seconda}
2\,s(\omega)-2\,\hat{s}+\Delta\left(\mrm{log}\left(1+\sqrt{1-\norm{\tau}^2}\right)\right)=0.
\end{equation}
As was already mentioned in the Introduction, this equation has been already studied by Donaldson in \cite{Donaldson_hyperkahler} and by T. Hodge in \cite{Hodge_phd_thesis} (see also \cite{traut}). In particular, if the $\omega_0$--norm of $\tau$ and its derivative is small enough, then there is a unique solution $\omega$ of equation \eqref{eq:F_tau_seconda} that is in the conformal class of $\omega_0$.

\section{The case of complex surfaces}\label{sec:complex_surface}

In this Section we will derive explicit moment map equations when the base manifold $M$ is a complex surface. The first step is to find an explicit expression for the Biquard-Gauduchon function $\rho$ on $\mrm{T}^*\m{AC}^+(4)$. This is computationally quite heavy. Obtaining similar expressions in general is certainly one of the difficulties in working out the HcscK system explicitly in higher dimension.

\subsection{The Biquard-Gauduchon function for $T^*\!\m{AC}^+(4)$}

In this subsection we will compute the Biquard-Gauduchon function. This involves working out the spectrum of the self-adjoint operator \eqref{eq:da_diagonalizzare}. 


An element $A\in T_{-\Omega_0}\m{AC}^+(4)$ is a matrix that can be written  as $A=\begin{pmatrix}P & Q\\ Q& -P\end{pmatrix}$ for $P=(P\indices{^i_j})_{1\leq i,j\leq 2}$ and $Q=(Q\indices{^i_j})_{1\leq i,j\leq 2}$ some $2\times 2$ symmetric matrices.

\bigskip

The space of all such matrices is $6$-dimensional, and a possible basis is given by the matrices
\begin{equation*}
\begin{split}
E_1=\begin{pmatrix}
1 &0 &0 &0\\
0 &0 &0 &0\\
0 &0 &-1 &0\\
0 &0 &0 &0\\
\end{pmatrix},\quad
E_2=\begin{pmatrix}
0 &1 &0 &0\\
1 &0 &0 &0\\
0 &0 &0 &-1\\
0 &0 &-1 &0\\
\end{pmatrix},\quad
E_3=\begin{pmatrix}
0 &0 &0 &0\\
0 &1 &0 &0\\
0 &0 &0 &0\\
0 &0 &0 &-1\\
\end{pmatrix},\\
E_4=\begin{pmatrix}
0 &0 &1 &0\\
0 &0 &0 &0\\
1 &0 &0 &0\\
0 &0 &0 &0\\
\end{pmatrix},\quad
E_5=\begin{pmatrix}
0 &0 &0 &1\\
0 &0 &1 &0\\
0 &1 &0 &0\\
1 &0 &0 &0\\
\end{pmatrix},\quad
E_6=\begin{pmatrix}
0 &0 &0 &0\\
0 &0 &0 &1\\
0 &0 &0 &0\\
0 &1 &0 &0\\
\end{pmatrix}.
\end{split}
\end{equation*}
The matrix representing the map $M(A):B\mapsto -\frac{1}{2}\left(A^2 B+B\,A^2\right)$ with respect to this basis may be conveniently expressed in terms of the quantities 
\begin{equation*}
\begin{split}
k_1=&(P\indices{^1_1})^2+(P\indices{^1_2})^2+(Q\indices{^1_1})^2+(Q\indices{^1_2})^2\\
k_2=&P\indices{^1_2}\left(P\indices{^1_1}+P\indices{^2_2}\right)+Q\indices{^1_2}\left(Q\indices{^1_1}+Q\indices{^2_2}\right)\\
k_3=&(P\indices{^1_2})^2+(P\indices{^2_2})^2+(Q\indices{^1_2})^2+(Q\indices{^2_2})^2\\
k_4=&Q\indices{^1_2}(P\indices{^2_2}-P\indices{^1_1})+P\indices{^1_2}(Q\indices{^1_1}-Q\indices{^2_2})
\end{split}
\end{equation*}
and is given by 
\begin{equation*}
M(A)=-\frac{1}{2}\left(
\begin{array}{ccc|ccc}
 2\,k_1 & 2\,k_2 & 0 & 0 & -2\,k_4 & 0 \\
 k_2 & k_1+k_3 & k_2 & k_4 & 0 & -k_4 \\
 0 & 2\,k_2 & 2\,k_3 & 0 & 2\,k_4 & 0 \\
 \hline
 0 & 2\,k_4 & 0 & 2\,k_1 & 2\,k_2 & 0 \\
 -k_4 & 0 & k_4 & k_2 & k_1+k_3 & k_2 \\
 0 & -2\,k_4 & 0 & 0 & 2\,k_2 & 2\,k_3 \\
\end{array}
\right)
\end{equation*}
(the vertical and horizontal lines have been added to make the symmetries of the matrix more evident).
It is useful to observe the identities 
\begin{equation*}
\frac{1}{2}\mrm{Tr}(A^2)=k_1+k_3, \quad \mrm{det}(A)=k_1k_3-k_2^2-k_4^2.
\end{equation*}
The spectrum of $M(A)$ contains three eigenvalues, each with multiplicity $2$. A lengthy computation shows that they are given by
\begin{equation*}
\begin{split}
-\frac{1}{2}\Bigg(&k_1+k_3,\,k_1+k_3+\sqrt{k_1^2+4\,k_2^2-2\,k_1\,k_3+k_3^2+4\,k_4},\,k_1+k_3-\sqrt{k_1^2+4\,k_2^2-2\,k_1\,k_3+k_3^2+4\,k_4}\Bigg)
\end{split}
\end{equation*}
and by the previous observation they can be rewritten as
\begin{equation}\label{eq:autovalori_curv_AC4}
-\frac{1}{2}\Bigg(\frac{\mrm{Tr}(A^2)}{2},\,\frac{\mrm{Tr}(A^2)}{2}+\sqrt{\left(\frac{\mrm{Tr}(A^2)}{2}\right)^2-4\,\mrm{det}(A)},\,\frac{\mrm{Tr}(A^2)}{2}-\sqrt{\left(\frac{\mrm{Tr}(A^2)}{2}\right)^2-4\,\mrm{det}(A)}\Bigg).
\end{equation}
In order to get more compact expressions we introduce the auxiliary quantities
\begin{equation*}
\delta^{\pm}(A):=\frac{1}{2}\left(\frac{\mrm{Tr}(A^2)}{2}\pm\sqrt{\left(\frac{\mrm{Tr}(A^2)}{2}\right)^2-4\,\mrm{det}(A)}\right).
\end{equation*}
Then a set of eigenvectors for the eigenvalues in \eqref{eq:autovalori_curv_AC4} is given by
\begin{equation*}
\begin{split}
v_1=&\left(\frac{k_2}{k_4},\frac{k_3-k_1}{2\,k_4},-\frac{k_2}{k_4},1,0,1\right)^\transpose;\\
v_2=&\left(2\frac{k_4^2-k_2^2}{(k_1-k_3)k_4},\frac{k_2}{k_4},2\frac{k_4^2+k_2^2}{(k_1-k_3)k_4},4\frac{k_2}{k_3-k_1},1,0\right)^\transpose;
\end{split}
\end{equation*}
\begin{equation*}
\begin{split}
v_3=&\left(\frac{k_2\left(\delta^+(A)-k_3\right)}{k_4\left(\delta^-(A)-k_3\right)},-\frac{\delta^+(A)-2\,k_3}{k_4},-\frac{k_2}{k_4},\frac{\delta^+(A)-k_3}{\delta^-(A)-k_3},0,1\right)^\transpose;
\end{split}
\end{equation*}
\begin{equation*}
\begin{split}
v_4=&\left(-\frac{k_2^2-k_4^2}{k_4\left(\delta^-(A)-k_3\right)},\frac{k_2}{k_4},-\frac{\delta^-(A)-k_3}{k_4},-2\frac{k_2}{\delta^-(A)-k_3},1,0 \right)^\transpose;
\end{split}
\end{equation*}
\begin{equation*}
\begin{split}
v_5=&\left(\frac{k_2\left(\delta^-(A)-k_3\right)}{k_4\left(\delta^+(A)-k_3\right)},-\frac{\delta^-(A)-k_3}{k_4},-\frac{k_2}{k_4},\frac{\delta^-(A)-k_3}{\delta^+(A)-k_3},0,1\right)^\transpose;
\end{split}
\end{equation*}
\begin{equation*}
\begin{split}
v_6=&\left(\frac{k_4^2-k_2^2}{k_4\left(\delta^+(A)-k_3\right)},\frac{k_2}{k_4},-\frac{\delta^+(A)-k_3}{k_4},-2\frac{k_2}{\delta^+(A)-k_3},1,0 \right)^\transpose.
\end{split}
\end{equation*}
We finally have all the ingredients needed in the computation of the spectral function for $M(A)$, and of the Biquard-Gauduchon $\rho$ function itself. A direct computation gives
\begin{equation*}
\begin{split}
\rho(A)=2&-\sqrt{1-\delta^+(A)}-\sqrt{1-\delta^-(A)}+\mrm{log}\left(\frac{1}{2}+\frac{1}{2}\sqrt{1-\delta^+(A)}\right)+\mrm{log}\left(\frac{1}{2}+\frac{1}{2}\sqrt{1-\delta^-(A)}\right).
\end{split}
\end{equation*}
Recall that a priori this is an expression for the Biquard-Gauduchon function $\rho$ at the point $-\Omega_0$. However, since we know that $\rho$ is invariant under the action of $\mrm{Sp}(2n)$ and that the action is transitive, this is in fact valid on the whole $T\m{AC}^+(4)$.

\subsection{The real moment map for a complex surface.}\label{sec:moment_map_complex_surface}

Consider now a path $(J_t,A_t)\in T\m{AC}^+(4)$; the differential $\mrm{d}\rho_{(J_0,A_0)}(\dot{J}_0,\dot{A}_0)$ is
\begin{equation}\label{eq:diff_t_rho}
\begin{split}
\frac{\mrm{d}}{\mrm{d}t}\Bigr|_{t=0}\left(\rho(J_t,A_t)\right)=&\frac{\mrm{Tr}(A_0\dot{A}_0)}{\sqrt{4-2\,\delta^+(A_0)}+\sqrt{4-2\,\delta^-(A_0)}}-\\
-&\frac{4\,\mrm{Tr}(\mrm{adj}(A_0)\dot{A}_0)}{\left(\sqrt{4-2\,\delta^+(A_0)}+\sqrt{4-2\,\delta^-(A_0)}\right)\left(2+\sqrt{4-2\,\delta^+(A_0)}\right)\left(2+\sqrt{4-2\,\delta^-(A_0)}\right)}.
\end{split}
\end{equation}

Equation \eqref{eq:moment_map_omega} tells us that we should compute (see the discussion at the beginning of \S\ref{section:moment_map_curve})
\begin{equation*}
\int_M\mrm{d}\rho_{(J,A)}\left(-J(\m{L}_{X_h}J),\mrm{Re}\left((\m{L}_{X_h}J)^\transpose\alpha+(\m{L}_{X_h}\alpha)J^\transpose\right)^\transpose\right)\,\frac{\omega_0^n}{n!}
\end{equation*}
for $A=\mrm{Re}\left(\alpha^\transpose\right)$, where $(J,\alpha)\in\cotJ$. By Remark \ref{nota:traccia_tripla} we can write the integrand using equation \eqref{eq:diff_t_rho} as
\begin{equation}\label{eq:diff_rho_dim4}
\begin{split}
\mrm{d}\rho_{(J,A)}&\left(-J(\m{L}_{X_h}J),\mrm{Re}\left((\m{L}_{X_h}J)^\transpose\alpha+(\m{L}_{X_h}\alpha)J^\transpose\right)^\transpose\right)=-\frac{\mrm{Tr}(A(\m{L}_{X_h}A)J)}{\sqrt{4-2\,\delta^+(A)}+\sqrt{4-2\,\delta^-(A)}}+\\
&+\frac{4\,\mrm{Tr}(\mrm{adj}(A)(\m{L}_{X_h}A)J)}{\left(\sqrt{4-2\,\delta^+(A)}+\sqrt{4-2\,\delta^-(A)}\right)\left(2+\sqrt{4-2\,\delta^+(A)}\right)\left(2+\sqrt{4-2\,\delta^-(A)}\right)}.
\end{split}
\end{equation}
To find an explicit expression for $\f{m}_{(J,\alpha)}$ we should try write this as the $L^2$-pairing of $h$ with some function $m(J,\alpha)\in\m{C}^\infty_0(M)$. Equation \eqref{eq:diff_rho_dim4} gives
\begin{equation*}
\begin{split}
\f{m}_{(J,\alpha)}(h)=&-\int_M\frac{\mrm{Tr}(A(\m{L}_{X_h}A)J)}{\sqrt{4-2\,\delta^+(A)}+\sqrt{4-2\,\delta^-(A)}}\,\frac{\omega_0^n}{n!}+\\
&+4\int_M\frac{\mrm{Tr}(\mrm{adj}(A)(\m{L}_{X_h}A)J)}{\left(\sqrt{4-2\,\delta^+(A)}+\sqrt{4-2\,\delta^-(A)}\right)\left(2+\sqrt{4-2\,\delta^+(A)}\right)\left(2+\sqrt{4-2\,\delta^-(A)}\right)}\,\frac{\omega_0^n}{n!}
\end{split}
\end{equation*}
and these two terms are quite similar to the one we had in complex dimension $1$, see equation \eqref{eq:diff_rho_dim2}. The first term can be written as a pairing $\langle h,F(J,A)\rangle_{L^2(M)}$ in the same way we did for equation \eqref{eq:diff_rho_dim2} in subsection \ref{section:moment_map_curve}, while to get the same result for the second term we have to make small modifications.

Let $F$ be defined as in \eqref{eq:F_F1F2F3F4}, and let $\tilde{F}$ be defined as
\begin{equation}\label{eq:tildeF_F1F2F3F4}
\begin{split}
\tilde{F}:\m{C}^\infty_0(M)&\to\m{C}^\infty_0(M)\\
h&\mapsto\mrm{Tr}((\m{L}_{X_h}A)J\mrm{adj}(A)).
\end{split}
\end{equation}
Then
\begin{equation*}
\begin{split}
\f{m}_{(J,\alpha)}(h)=&-\left\langle F(h),\frac{1}{\sqrt{4-2\,\delta^+(A)}+\sqrt{4-2\,\delta^-(A)}}\right\rangle_{L^2(M)}+\\
&+4\left\langle\tilde{F}(h),\frac{1}{\left(\sqrt{4-2\,\delta^+(A)}+\sqrt{4-2\,\delta^-(A)}\right)\left(2+\sqrt{4-2\,\delta^+(A)}\right)\left(2+\sqrt{4-2\,\delta^-(A)}\right)}\right\rangle_{L^2(M)}.
\end{split}
\end{equation*}

The computation of the formal adjoint of $F$ that was carried out in subsection \ref{section:moment_map_curve}, particularly in Lemma \ref{lemma:aggiunto_F2} and Corollary \ref{corollario:aggiunto_F2}, actually holds in any dimension. We can use them also to compute the adjoint of $\tilde{F}$, by virtue of Lemma \ref{lemma:aggiunto_TJ}. The only difference is that, while $F=F_3\circ F_2\circ F_1$, we have instead $\tilde{F}=\tilde{F}_3\circ F_2\circ F_1$, with
\begin{align*}
\tilde{F}_3:\Gamma(M,\mrm{End}(TM))&\to\m{C}^\infty(M)\\
P&\mapsto \mrm{Tr}(PJ\,\mrm{adj}(A)).
\end{align*}
The formal adjoint of $\tilde{F}_3$ is readily computed as $\tilde{F}_3^*(f)=-f\,\mrm{adj}(A)J$.

Introduce the quantities
\begin{equation*}
\begin{split}
\psi&=\frac{1}{\sqrt{4-2\,\delta^+(A)}+\sqrt{4-2\,\delta^-(A)}};\\
\tilde{\psi}&=\frac{1}{\left(\sqrt{4-2\,\delta^+(A)}+\sqrt{4-2\,\delta^-(A)}\right)\left(2+\sqrt{4-2\,\delta^+(A)}\right)\left(2+\sqrt{4-2\,\delta^-(A)}\right)}.
\end{split}
\end{equation*}
Our computations so far show
\begin{equation*}
\begin{split}
\f{m}_{(J,\alpha)}(h)=&-\langle F(h),\psi\rangle+4\,\langle\tilde{F}(h),\tilde{\psi}\rangle=\\
=&\langle h,\mrm{div}\left[\psi\,J\,C^2_1\left((\nabla A)AJ\right)^\sharp+2\,J\nabla^*(\psi A^2J)\right]\rangle-\\
&-4\langle h,\mrm{div}\left[\tilde{\psi}\,J\,C^2_1\left((\nabla A)\mrm{adj}(A)J\right)^\sharp+2\,J\nabla^*(\tilde{\psi}\,\mrm{det}(A)J)\right]\rangle
\end{split}
\end{equation*}
and so we have an explicit expression for $\f{m}(J,\alpha)$, letting as usual $A=\mrm{Re}(\alpha)^\transpose$ (see equation \eqref{eq:momento_Omega_dim2_A}):
\begin{equation}\label{eq:momento_Omega_dim4_A}
\begin{split}
\f{m}(J,\alpha)=&\mrm{div}\left[\psi\,J\,C^2_1\left((\nabla A)AJ\right)^\sharp+2\,J\nabla^*(\psi A^2J)\right]-\\
&-4\,\mrm{div}\left[\tilde{\psi}\,J\,C^2_1\left((\nabla A)\mrm{adj}(A)J\right)^\sharp+2\,J\nabla^*(\tilde{\psi}\,\mrm{det}(A)J)\right]
\end{split}
\end{equation}
It is possible to simplify this result further, following closely what we did in the case of curves.

Assume that $J$ is integrable. Then $\nabla J=0$, hence
\begin{equation*}
\begin{split}
J\nabla^*(\psi\,A^2J)&=-\nabla^*(\psi\,A^2)\\
J\nabla^*(\tilde{\psi}\,\mrm{det}(A)J)&=-\nabla^*(\tilde{\psi}\,\mrm{det}(A)\,\mrm{Id})=\mrm{grad}(\tilde{\psi}\,\mrm{det}(A)).
\end{split}
\end{equation*}
It will be useful to have a more compact notation for $\mrm{adj}(A)$. We'll denote it by $\tilde{A}$ whenever working in local coordinates.
\begin{lemma}
Let $J\in\scr{J}$ be an integrable, compatible complex structure. For any $A\in T_J\scr{J}$
\begin{equation*}
J\,C^2_1\left((\nabla A)AJ\right)^\sharp=-\mrm{grad}\left(\frac{\mrm{Tr}(A^2)}{2}\right)+2\,\left(g(\nabla^aA^{0,1},A^{1,0})\diff_a+g(\nabla^{\bar{b}}A^{1,0},A^{0,1})\diff_{\bar{b}}\right);
\end{equation*}
\begin{equation*}
J\,C^2_1\left((\nabla A)\tilde{A}J\right)^\sharp=\left(g(\nabla^aA^{0,1},\tilde{A}^{1,0})-g(\nabla^aA^{1,0},\tilde{A}^{0,1})\right)\diff_a+\left(g(\nabla^{\bar{a}}A^{0,1},\tilde{A}^{1,0})-g(\nabla^{\bar{a}}A^{1,0},\tilde{A}^{0,1})\right)\diff_{\bar{a}}.
\end{equation*}
\end{lemma}
This is proved by precisely the same type of computations carried out at the end of Section \ref{section:moment_map_curve}. We omit the details.

Summarising our results in this Section, we have derived the expression
\begin{equation}\label{eq:HCSCK_surface_complete}
\begin{split}
\f{m}(J,\alpha)=&\mrm{div}\left[-\psi\,\mrm{grad}\left(\frac{\mrm{Tr}(A^2)}{2}\right)+2\,\psi\,\left(g(\nabla^aA^{0,1},A^{1,0})\diff_a+\mrm{c.c.}\right)-2\,\nabla^*(\psi\,A^2)\right]-\\
&-4\,\mrm{div}\left[\tilde{\psi}\left(g(\nabla^aA^{0,1},\tilde{A}^{1,0})-g(\nabla^aA^{1,0},\tilde{A}^{0,1})\right)\diff_a+\mrm{c.c.}+2\,\mrm{grad}(\tilde{\psi}\,\mrm{det}(A))\right]
\end{split}
\end{equation}
where ``$\mrm{c.c.}$'' denotes simply the complex conjugate of the term immediately before it.

\paragraph*{Low-rank case.} There are some conditions under which the expression for $\f{m}(J,\alpha)$ becomes much simpler. If $A$ does not have maximal rank then $\mrm{det}(A)=0$; moreover, since the rank of $A$ is even (the kernel of $A$ is $J$-invariant), if $\mrm{rk}(A)$ is not maximal then actually $\mrm{rk}(A)=0$ or $2$, so also $\mrm{adj}(A)=0$.

Hence if $\mrm{rank}(A)$ is not maximal we get
\begin{equation*}
\delta^{\pm}(A):=\frac{1}{2}\left(\frac{\mrm{Tr}(A^2)}{2}\pm\sqrt{\left(\frac{\mrm{Tr}(A^2)}{2}\right)^2-4\,\mrm{det}(A)}\right)=\frac{1}{2}\left(\frac{\mrm{Tr}(A^2)}{2}\pm\frac{\mrm{Tr}(A^2)}{2}\right)
\end{equation*}
so that $\delta^+(A)=\frac{1}{2}\mrm{Tr}(A^2)=\frac{1}{2}\norm{A}_{g_J}^2$, and we also find
\begin{equation*}
\psi=\frac{1}{\sqrt{4-2\,\delta^+(A)}+\sqrt{4-2\,\delta^-(A)}}=\frac{1}{2}\frac{1}{1+\sqrt{1-\frac{1}{4}\mrm{Tr}(A^2)}}.
\end{equation*}
So, in this low-rank case, we can write
\begin{equation}\label{eq:low_rank}
\begin{split}
\f{m}(J,\alpha)=&\mrm{div}\left[-\frac{\mrm{grad}\left(\frac{1}{4}\mrm{Tr}(A^2)\right)}{1+\sqrt{1-\frac{1}{4}\mrm{Tr}(A^2)}}+\frac{g(\nabla^aA^{0,1},A^{1,0})\diff_a+\mrm{c.c.}}{1+\sqrt{1-\frac{1}{4}\mrm{Tr}(A^2)}}-\nabla^*\left(\frac{A^2}{1+\sqrt{1-\frac{1}{4}\mrm{Tr}(A^2)}}\right)\right]=\\
=&\mrm{div}\left[-\frac{\mrm{grad}\left(\frac{1}{2}\norm{A^{1,0}}^2\right)}{1+\sqrt{1-\frac{1}{2}\norm{A^{1,0}}^2}}+\frac{g(\nabla^aA^{0,1},A^{1,0})\diff_a+\mrm{c.c.}}{1+\sqrt{1-\frac{1}{2}\norm{A^{1,0}}^2}}-\nabla^*\left(\frac{A^2}{1+\sqrt{1-\frac{1}{2}\norm{A^{1,0}}^2}}\right)\right]
\end{split}
\end{equation}
The resulting moment map is remarkably similar to the one we had in the Riemann surface case, see equation \eqref{eq:mappa_momento_RS_parziale}. In the rest of this paper we will focus on this low-rank case.

\section{The equations on a ruled surface}\label{sec:ruled_surface}

Let $\Sigma$ be a Riemann surface of genus $g(\Sigma)\geq 2$ and assume that $L\to\Sigma$ is a holomorphic line bundle equipped with a Hermitian fibre metric $h$. In this section we study our equations on the ruled surface $M=\bb{P}(\m{O}\oplus L)$ (the completion of $L$) using the \emph{momentum construction}; our main reference for this technique is \cite{Hwang_Singer}; see also \cite[chapter $5$]{Szekelyhidi_phd}). 

After this initial study we solve a ``complexified'' version of the equations in the particular case when $L$ is the anticanonical bundle of $\Sigma$. We remark that we solve just a subset of equations of the complexified HcscK system \eqref{eq:HCSCK_complex_relaxed}, namely for a fixed complex structure $J$ we'll find a K\"ahler form $\omega_\phi$ and a ``Higgs field" $\alpha$ that are a zero of the moment maps, but such that $\alpha$ and $\omega_\phi$ are \emph{not} compatible. In fact we will not solve the equations in general, but rather prove that in the ``\emph{adiabatic limit}'' in which the fibres are sufficiently small a solution exists. This is a well developed technique and we follow in particular the approach of \cite{Fine_phd}.  

\smallskip

For a fixed K\"ahler form $\omega_\Sigma$ on $\Sigma$, we consider K\"ahler forms on the total space of the bundle
\begin{equation*}
\bb{P}(L\oplus\m{O})\xrightarrow{\pi}\Sigma
\end{equation*}
that satisfy the \emph{Calabi ansatz}, i.e. we consider a form $\omega$ of the form
\begin{equation}\label{eq:metrica_ruled_surface}
\omega=\pi^*\omega_\Sigma+\I\,\diff\bar{\diff}f(t)
\end{equation}
where $t$ is the logarithm of the fibrewise norm function, and $f$ is a suitably convex real function. More explicitly, we fix a system of holomorphic coordinates $(z,\zeta)$ on $M$ that are adapted to the bundle structure, i.e. $z$ is a holomorphic coordinate on $\Sigma$ while $\zeta$ is a linear coordinate on the fibres of $L\to\Sigma$. Let $a(z)$ denote the local function on $\Sigma$ such that the Hermitian metric $h$ on $L$ is given by $h=a(z)\,\mrm{d}\zeta\,\mrm{d}\bar{\zeta}$; then $t:=\mrm{log}(a(z)\,\zeta\bar{\zeta})$ is a well-defined function on $L\setminus\Sigma$, and if $f$ satisfies some conditions on its second derivative then $\I\diff\bar{\diff}f(t)$ is a (globally) well-defined real $2$-form on the total space of $L$, that in some cases can be extended to $M$.

Let $F(h)$ be the curvature form of $h$. We choose $h$ such that $F(h)=-\omega_\Sigma$. Then in bundle-adapted holomorphic coordinates $\bm{w}=(z,\zeta)$ we have 
\begin{equation}\label{eq:metrica_CalabiAnsatz}
\begin{split}
\pi^*\omega_\Sigma+\I\,\diff\bar{\diff}f(t)=&(1+f'(t))\pi^*\omega_\Sigma+\\
&+\I\,f''(t)\left[\diff_zt\,\diff_{\bar{z}}t\,\mrm{d}z\wedge\mrm{d}\bar{z}+\frac{\diff_zt}{\bar{\zeta}}\mrm{d}z\wedge\mrm{d}\bar{\zeta}+\frac{\diff_{\bar{z}}t}{\zeta}\mrm{d}\zeta\wedge\mrm{d}\bar{z}+\frac{1}{\zeta\,\bar{\zeta}}\mrm{d}\zeta\wedge\mrm{d}\bar{\zeta}\right].
\end{split}
\end{equation}

It will be useful to change point of view to describe the curvature properties of the metric $\omega$. Rather than working with $f$ and $t$, define $\tau$ to be the function $\tau=f'(t)$, and let $F$ be the Legendre transform of $f$. If we define $\phi:=\frac{1}{F''}$, then we have
\begin{align*}
\tau=f'(t)\\
t=F'(\tau)\\
F(\tau)+f(t)=t\,\tau\\
f''(t)=\phi(\tau)
\end{align*}
so that the metric $\omega_\phi:=\omega$ is, with the notation of \eqref{eq:metrica_CalabiAnsatz}
\begin{equation}\label{eq:metrica_CalabiAnsatz_phi}
\begin{split}
\omega_\phi=(1+\tau)\pi^*\omega_\Sigma+\I\,\phi(\tau)\left(\diff_zt\,\diff_{\bar{z}}t\,\mrm{d}z\wedge\mrm{d}\bar{z}+\frac{\diff_zt}{\bar{\zeta}}\mrm{d}z\wedge\mrm{d}\bar{\zeta}+\frac{\diff_{\bar{z}}t}{\zeta}\mrm{d}\zeta\wedge\mrm{d}\bar{z}+\frac{1}{\zeta\,\bar{\zeta}}\mrm{d}\zeta\wedge\mrm{d}\bar{\zeta}\right)
\end{split}
\end{equation}
In particular, the matrices of the metric and its inverse in this system of coordinates are
\begin{equation*}
\left(g_{a\bar{b}}\right)_{1\leq a,b\leq 2}=\begin{pmatrix}(1+\tau)g_\Sigma+\phi(\tau)\,\diff_zt\,\diff_{\bar{z}}t & \phi(\tau)\frac{\diff_zt}{\bar{\zeta}} \\
\phi(\tau)\frac{\diff_{\bar{z}}t}{\zeta} & \frac{\phi(\tau)}{\zeta\,\bar{\zeta}} \end{pmatrix}
\end{equation*}
\begin{equation*}
\left(g^{a\bar{b}}\right)_{1\leq a,b\leq 2}=\begin{pmatrix}\frac{1}{(1+\tau)g_\Sigma}& -\frac{\bar{\zeta}\,\diff_{\bar{z}}t}{(1+\tau)g_\Sigma}\\
-\frac{\zeta\,\diff_zt}{(1+\tau)g_\Sigma} & \frac{\zeta\,\bar{\zeta}}{\phi(\tau)}+\frac{\zeta\,\bar{\zeta}\,\diff_zt\,\diff_{\bar{z}}t}{(1+\tau)g_\Sigma} \end{pmatrix}.
\end{equation*}
The main reasons for using $\phi(\tau)$ rather than $f(t)$ are given by Proposition \ref{prop:omega_phi_estende} and Proposition \ref{prop:curv_scal_phi}. Note that we are only stating a particular case of the more general results of Hwang-Singer in \cite{Hwang_Singer}.
\begin{proposizione}[\cite{Hwang_Singer}, see also \cite{Szekelyhidi_libro}]\label{prop:omega_phi_estende}
Assume that $\phi:[a,b]\to[0,\infty)$ is a function positive on the interior of $[a,b]$. Then $\omega_\phi$ defines a smooth metric on $M\setminus\Sigma_\infty$ if and only if $\phi(a)=0$, $\phi'(0)=1$. Moreover, $\omega_\phi$ extends to the whole of $M$ if and only if $\phi(a)=\phi(b)=0$ and $\phi'(a)=1$, $\phi'(b)=-1$.
\end{proposizione}

Then it will be useful to assume that $\tau$ takes values in an interval $[a,b]$. The convexity assumptions on $f$ imply that actually $\tau$ is increasing (as a function of $t$), and that $\tau_{\restriction_{\Sigma_0}}=a$, $\tau_{\restriction_{\Sigma_\infty}}=b$. Up to translations, we can assume that in fact $[a,b]=[0,m]$ for some $m\in\bb{R}_{>0}$. This $m$ has a direct geometric interpretation:

\begin{lemma}\label{lemma:volume_fibre}
The volume of a fibre of $\bb{P}(\m{O}\oplus L)\to\Sigma$ is $2\,\pi\,m$.
\end{lemma}

\begin{proof}
We just have to compute $\int_{F}i^*\omega_\phi$, where $F$ is a fibre of $\bb{P}(\m{O}\oplus L)\to\Sigma$ and $i:F\hookrightarrow\bb{P}(\m{O}\oplus L)$ is the inclusion. Fix a system of bundle-adapted coordinates $(z,\zeta)$ on $\bb{P}(\m{O}\oplus L)$, and let $r=\card{\zeta}$. Then $\diff_r\tau=2\,\phi(\tau)\,r^{-1}$, and so
\begin{equation*}
\int_{F}i^*\omega_\phi=\int_{\zeta\in\bb{C}}\I\,\frac{\phi(\tau)}{r^2}\,\mrm{d}\zeta\,\mrm{d}\bar{\zeta}=\int_{\bb{R}^2}2\,\frac{\phi(\tau)}{r^2}\,\mrm{d}x\,\mrm{d}y=\int_{[0,2\,\pi]\times\bb{R}}\diff_r\tau\,\mrm{d}\tau\,\mrm{d}\vartheta=2\,\pi\,m.
\end{equation*}
\end{proof}

\begin{proposizione}[\cite{Hwang_Singer}, see also \cite{Szekelyhidi_libro}]\label{prop:curv_scal_phi}
With the previous notation, the scalar curvature of $\omega_\phi$ is
\begin{equation*}
s(\omega_\phi)=\frac{1}{1+\tau}\pi^*s(\omega_\Sigma)-\phi''(\tau)-\frac{2}{1+\tau}\phi'(\tau).
\end{equation*}
\end{proposizione}

\medskip

To study the moment map equations we will also need an explicit expression for $\widehat{s(\omega_\phi)}$.

\begin{lemma}\label{lemma:media_curvatura}
If $\phi$ defines a K\"ahler metric on the whole ruled surface $\bb{P}(L\oplus\m{O})$ then
\begin{equation*}
\widehat{s(\omega_\phi)}=\frac{2}{m+2}\widehat{s(\omega_\Sigma)}+\frac{2}{m}.
\end{equation*}
\end{lemma}

\begin{proof}
We use the same notation of the proof of Lemma \ref{lemma:volume_fibre}. First notice that
\begin{equation*}
\omega_\phi^2=-(1+\tau)g_\Sigma\frac{\phi(\tau)}{r^2}\,\mrm{d}z\wedge\mrm{d}\bar{z}\wedge\mrm{d}\zeta\wedge\mrm{d}\bar{\zeta}
\end{equation*}
so that the volume of $M=\bb{P}(L\oplus\m{O})$ is
\begin{equation*}
\begin{split}
\mrm{Vol}_\phi(M)=&\int_M\frac{\omega_\phi^2}{2}=-\frac{1}{2}\int_\Sigma\mrm{d}z\,\mrm{d}\bar{z}\left[g_0\int_{\bb{C}}(1+\tau)\frac{\phi(\tau)}{r^2}\mrm{d}\zeta\,\mrm{d}\bar{\zeta}\right]=\int_\Sigma\mrm{d}z\,\mrm{d}\bar{z}\,g_0\left[\pi\,\I\left(1+\frac{m}{2}\right)m\right]=\\
=&\pi\frac{m(2+m)}{2}\,\mrm{Vol}_{\omega_\Sigma}(\Sigma).
\end{split}
\end{equation*}

In order to compute the integral of $s(\omega_\phi)$ recall that
\begin{equation*}
s(\omega_\phi)=\frac{s(\omega_\Sigma)}{1+\tau}-\phi''(\tau)-2\frac{\phi'(\tau)}{1+\tau}.
\end{equation*}
Then
\begin{equation*}
\begin{split}
\int_M s(\omega_\phi)\frac{\omega_\phi^2}{2}=&-\frac{1}{2}\int_M(1+\tau)g_\Sigma\frac{\phi(\tau)}{r^2}\left(\frac{s(\omega_\Sigma)}{1+\tau}-\phi''(\tau)-2\frac{\phi'(\tau)}{1+\tau}\right)\mrm{d}z\,\mrm{d}\bar{z}\,\mrm{d}\zeta\,\mrm{d}\bar{\zeta}=\\
=&-\frac{1}{2}\int_\Sigma\mrm{d}z\,\mrm{d}\bar{z}\,g_\Sigma\,s(\omega_\Sigma)\left[\int_{\bb{C}}\frac{\phi(\tau)}{r^2}\mrm{d}\zeta\,\mrm{d}\bar{\zeta}\right]+\frac{1}{2}\int_\Sigma\mrm{d}z\,\mrm{d}\bar{z}\,g_\Sigma\left[\int_{\bb{C}}\frac{2\,\phi(\tau)\,\phi'(\tau)}{r^2}\mrm{d}\zeta\,\mrm{d}\bar{\zeta}\right]\\
&+\frac{1}{2}\int_\Sigma\mrm{d}z\,\mrm{d}\bar{z}\,g_\Sigma\left[\int_{\bb{C}}\frac{(1+\tau)\phi(\tau)\,\phi''(\tau)}{r^2}\mrm{d}\zeta\,\mrm{d}\bar{\zeta}\right].
\end{split}
\end{equation*}
We split the computation in three parts. To compute the integrals over $\bb{C}$, we use polar coordinates.
\begin{equation*}
\int_{\bb{C}}\frac{\phi(\tau)}{r^2}\mrm{d}\zeta\,\mrm{d}\bar{\zeta}=-2\,\I\int_{\bb{C}}\frac{\phi(\tau)}{r}\mrm{d}\vartheta\,\mrm{d}r=-\I\int_0^{2\pi}\mrm{d}\vartheta\left[\int_0^\infty2\frac{\phi(\tau)}{r}\mrm{d}r\right]=-2\,\pi\,\I\,m
\end{equation*}
\begin{equation*}
\int_{\bb{C}}\frac{2\,\phi(\tau)\,\phi'(\tau)}{r^2}\mrm{d}\zeta\,\mrm{d}\bar{\zeta}=-2\,\I\int_0^{2\pi}\mrm{d}\vartheta\left[2\int_0^\infty\frac{\phi(\tau)\,\phi'(\tau)}{r}\mrm{d}r\right]=-2\,\I\int_0^{2\pi}\mrm{d}\vartheta\left[\phi(\tau)\right]^\infty_0=0
\end{equation*}
\begin{equation*}
\begin{split}
\int_{\bb{C}}\frac{(1+\tau)\phi(\tau)\,\phi''(\tau)}{r^2}&\mrm{d}\zeta\,\mrm{d}\bar{\zeta}=-\,\I\int_0^{2\pi}\mrm{d}\vartheta\left[2\int_0^\infty\frac{(1+\tau)\phi(\tau)\,\phi''(\tau)}{r}\mrm{d}r\right]=\\
=&-\I\int_0^{2\pi}\mrm{d}\vartheta\left[\int_0^\infty\diff_r\phi'(\tau)\mrm{d}r\right]-\I\int_0^{2\pi}\mrm{d}\vartheta\left[\int_0^\infty\diff_r(\phi'(\tau)\,\tau)-\diff_r\phi(\tau)\mrm{d}r\right]=\\
=&4\,\pi\,\I+2\,\pi\,\I\,m.
\end{split}
\end{equation*}
Putting everything together:
\begin{equation*}
\begin{split}
\int_M s(\omega_\phi)\frac{\omega_\phi^2}{2}=&-\frac{1}{2}\int_\Sigma\mrm{d}z\,\mrm{d}\bar{z}\,g_\Sigma\,s(\omega_\Sigma)\left[-2\,\pi\,\I\,m\right]+\frac{1}{2}\int_\Sigma\mrm{d}z\,\mrm{d}\bar{z}\,g_\Sigma\left[4\,\pi\,\I+2\,\pi\,\I\,m\right]=\\
=&\pi\,m\int_\Sigma s(\omega_\Sigma)\,\omega_\Sigma+(2\,\pi+\pi\,m)\,\mrm{Vol}_{\omega_\Sigma}(\Sigma).
\end{split}
\end{equation*}
Finally:
\begin{equation*}
\widehat{s(\omega_\phi)}=\frac{\int_M s(\omega_\phi)\,\frac{\omega^2_\phi}{2}}{\mrm{Vol}_\phi(M)}=2\frac{\pi\,m\int_\Sigma s(\omega_\Sigma)\,\omega_\Sigma+(2\,\pi+\pi\,m)\,\mrm{Vol}_{\omega_\Sigma}(\Sigma)}{\pi\,m(2+m)\,\mrm{Vol}_{\omega_\Sigma}(\Sigma)}=\frac{2}{2+m}\widehat{s(\omega_\Sigma)}+\frac{2}{m}.
\end{equation*}
\end{proof}

An analogous computation will give the K\"ahler class of $\omega_\phi$.

\begin{lemma}[See \S $4.4$ in \cite{Szekelyhidi_libro}]\label{lemma:classe_phi}
Consider on $\bb{P}(\m{O}\oplus L)$ the classes of a fibre $\m{C}$ and the infinity section $\Sigma_\infty$. Then the Poincaré dual to $[\omega_\phi]$ is
\begin{equation*}
\m{L}_m:=2\,\pi\left(\m{C}+m\,\Sigma_\infty\right).
\end{equation*}
\end{lemma}

\paragraph*{Transversally normal coordinates.} For many of the computations that we will have to make, it will be convenient to choose bundle-adapted holomorphic coordinates $\bm{w}=(z,\zeta)$ such that, for a fixed point $p\in\Sigma$, $\left(\diff_zt\right)(p)=0$. For brevity, we will call coordinates with these properties \emph{transversally normal at} $p$. Such a system of coordinates always exists, they are essentially just normal coordinates for the bundle metric $h$. In these coordinates the metric $\omega_\phi$ becomes (see equation \eqref{eq:metrica_CalabiAnsatz_phi})
\begin{equation*}
\omega(p)=(1+\tau)\pi^*\omega_\Sigma+\I\,\frac{\phi(\tau)}{\zeta\bar{\zeta}}\mrm{d}\zeta\wedge\mrm{d}\bar{\zeta}.
\end{equation*}
In particular, it will be convenient to use transversally normal coordinates whenever we have to compute objects that involve the Christoffel symbols of $\omega_\phi$, since in these coordinates $g_\phi$ and its inverse are diagonal.

\begin{lemma}
The Christoffel symbols of $\omega_\phi$ are
\begin{align*}
\Gamma^1_{11}=&2\frac{\phi(\tau)}{1+\tau}\diff_zt+\Gamma^1_{11}(\Sigma); & \Gamma^2_{11}=&\zeta\,(\diff_zt)^2\left(-2\frac{\phi(\tau)}{1+\tau}+\phi'(\tau)\right)+\zeta\,\diff^2_zt-\zeta\,\diff_zt\,\Gamma^1_{11}(\Sigma);\\
\Gamma^1_{21}=&\frac{\phi(\tau)}{(1+\tau)\zeta}; & \Gamma^2_{21}=&\diff_zt\left(-\frac{\phi(\tau)}{1+\tau}+\phi'(\tau)\right);\\
\Gamma^1_{22}=&0; & \Gamma^2_{22}=&\frac{\phi'(\tau)-1}{\zeta}.
\end{align*}
\end{lemma}
In particular, if we fix a point $p\in\Sigma$ and a system of transversally normal coordinates around it, the Christoffel symbols of $\omega_\phi$ at the point $p$ are
\begin{align}\label{christoffel}
\nonumber \Gamma^1_{11}=&\Gamma^1_{11}(\Sigma); & \Gamma^2_{11}=&0;\\
\nonumber \Gamma^1_{21}=&\frac{\phi(\tau)}{(1+\tau)\zeta}; & \Gamma^2_{21}=&0;\\
\Gamma^1_{22}=&0; & \Gamma^2_{22}=&\frac{\phi'(\tau)-1}{\zeta}.
\end{align}

\subsection{Deforming complex structures on the total space of a vector bundle}

The HcscK equations involve both a K\"ahler metric and a deformation of the complex structure. While in this ruled surface case we have already chosen to use K\"ahler metrics satisfying the Calabi ansatz \eqref{eq:metrica_ruled_surface}, we still have to choose which deformations of $\bb{P}(\m{O}\oplus L)$ to consider. The natural choice is to consider a deformation of the $\bdiff$-operator of $E:=\m{O}\oplus L$, so a matrix-valued form $\beta\in\mathcal{A}^{0,1}(\mrm{End}(E))$; this $\beta$ will induce a deformation $A\in\mrm{End}(TE)$ of the complex structure of the total space (which we still denote by $E$).

First, recall how a $\bdiff_E$-operator determines the complex structure $J_E$, see \cite[Proposition $1.3.7$]{Kobayashi_bundles}. Fix a local holomorphic coordinate $z$ on $\Sigma$ and a local frame $(s_1,s_2)$ on $E$. If we let $(w^1,w^2)$ be the usual coordinates on $\mathbb{C}^2$, by the choice of the local frame we can use $(z,w^1,w^2)$ as local complex coordinates on $E$. Denote by
\begin{equation*}
T\indices{^i_j}:=T\indices{_{\bar{1}}^i_j}\mathrm{d}\bar{z}
\end{equation*}
the local representative of the $\bar{\partial}_E$-operator. A complex structure on $E$ is uniquely determined by a decomposition $T_{\mathbb{C}}E=T^{1,0}E\oplus T^{0,1}E$; we define
\begin{equation*}
T^{1,0}E:=\mrm{span}_{\mathbb{C}}\left(\diff_{w^1},\diff_{w^2},\diff_z-\overline{T\indices{^i_j}}(\diff_z)\bar{w}^j\diff_{\bar{w}^i}\right).
\end{equation*}
A different choice of a local frame does not change this bundle; moreover, the integrability of $\bdiff_E$ (i.e. $\bdiff_E^2=0$) is equivalent to that of $T^{1,0}E$ (i.e. $[T^{1,0}E,T^{1,0}E]\subseteq T^{1,0}E$.)

Consider now the case in which we already have a holomorphic structure $\bdiff_E$, and we are deforming it as $\bdiff'_E:=\bdiff_E+\beta$ for some $\beta\in\mathcal{A}^{0,1}(\mrm{End}(E))$. Choose a local $\bdiff_E$-holomorphic frame $s_1,s_2$ for $E$. Then a local representative for $\bdiff'_E$ in this local frame is just the matrix $\beta$, and the previous construction gives us
\begin{equation*}
T_{\bdiff_E}^{1,0}E=\mrm{span}_{\mathbb{C}}\left(\diff_{w^1},\diff_{w^2},\diff_z\right),\quad\quad T_{\bdiff'_E}^{1,0}E=\mrm{span}_{\mathbb{C}}\left(\diff_{w^1},\diff_{w^2},\diff_z-\overline{\beta\indices{_{\bar{1}}^i_j}}\bar{w}^j\diff_{\bar{w}^i}\right).
\end{equation*}

Changing point of view, $\bdiff_E$ defines on the total space of $E$ a complex structure $J_E$, and if we slightly deform it to $J'_E:=J_E+\varepsilon\, A$ for some $A\in\Gamma(E,\mrm{End}(TE))$, to first order in $\varepsilon$ the holomorphic tangent bundle of $E$ with respect to $J'_E$ can be described as
\begin{equation*}
T_{J'_E}^{1,0}E=\left\lbrace v-\frac{\mrm{i}\,\varepsilon}{2}A(v)\mid v\in T_{J_E}^{1,0}E\right\rbrace.
\end{equation*}
Comparing the spaces $T_{J'_E}^{1,0}E$ and $T_{\bdiff'_E}^{1,0}E$, we see that $A$ induces the same deformation of $J_E$ as $\beta$ if and only if
\begin{equation}
\begin{split}
A^{1,0}(\diff_{\bar{w}^i})=&0\\
A^{1,0}(\diff_{\bar{z}})=&2\,\mrm{i}\,\beta\indices{^i_j}(\diff_{\bar{z}})w^j\diff_{w^i};
\end{split}
\end{equation}
we let $A(\beta)$ be the deformation of the complex structure defined by these equations.

\bigbreak

The next step is to see how a deformation of $\bdiff_E$, $\beta\in\m{A}^{0,1}(\mrm{End}(E))$ induces a deformation of the complex structure of $\bb{P}(E)$. From the previous discussion, we have a canonical way to induce a first-order deformation $A(\beta)\in\Gamma(\mrm{End}(TE))$ of the complex structure of $E$. Now, on $E$ we have the usual $\bb{C}^*$-action on the fibres, and $\bb{P}(E)$ is defined as
\begin{equation*}
\bb{P}(E):=\left(E\setminus M\right)/\bb{C}^*.
\end{equation*}

\begin{lemma}
Let $p:E\setminus M\to\bb{P}(E)$ be the usual projection, and fix $\beta\in\m{A}^{0,1}(\mrm{End}(E))$. Then $A=A(\beta)$ induces a deformation of the complex structure of $\bb{P}(E)$ as follows: for $[x]\in\bb{P}(E)$ and $v\in T_{[x]}\bb{P}(E)$ choose a $p$-lift $\hat{v}\in T_xE$ of $v$, and let
\begin{equation*}
A_{[x]}(v):=p_*A_{x}(\hat{v}).
\end{equation*}
\end{lemma}

\begin{proof}
We have to check that this expression does not depend upon the choice of the preimage of $[x]$ and of the lift $\hat{v}$ of $v$.

Fix holomorphic local frames of $\m{O}$ and $L$, so that we can locally describe $E$ as $M\times\bb{C}^2$, with coordinates $w^1,w^2$ on the fibres. We get homogeneous coordinates on the fibres of $\bb{P}(E)$ as $[w^1:w^2]$. If we fix a holomorphic coordinate $z$ on $M$, on the open subset of $\bb{P}(E)$ where $w^1\not=0$ we have local holomorphic coordinates $(z,\zeta)$, with $w=w^2/w^1$.

In this system of local coordinates the projection $p$ is written as $p(z,w^1,w^2)=\left(z,\frac{w^2}{w^1}\right)$, and (the $(1,0)$ part of) its differential is
\begin{equation*}
\mrm{d}p_{(z,w^1,w^2)}=\begin{pmatrix}
1 & 0 & 0 \\
0 & -\frac{w^2}{(w^1)^2}& \frac{1}{w^1}
\end{pmatrix}.
\end{equation*}
We have to check that for all $[x]\in\bb{P}(E)$ and all $\lambda\in\bb{C}^*$, if $\hat{v}_1\in T^{0,1}_{x}E$ and $\hat{v}_2\in T^{0,1}_{\lambda\,x}E$ are such that $p_*\hat{v}_1=p_*\hat{v}_2$, then also
\begin{equation*}
p_*A_{x}(\hat{v}_1)=p_*A_{\lambda\,x}(\hat{v}_2).
\end{equation*}
If $x=(z,w^1,w^2)$ and $\hat{v}_1=V\,\diff_{\bar{z}}+U^{\bar{i}}\diff_{\bar{w}^i}$ then
\begin{equation*}
\begin{split}
p_*A_{x}(\hat{v}_1)&=p_*\left(2\,\I\,V\,\beta\indices{^i_j}(\diff_{\bar{z}})\,w^j\diff_{w^i}\right)=2\,\I\,V\left(-\beta\indices{^1_j}(\diff_{\bar{z}})\,w^j\frac{w^2}{(w^1)^2}+\beta\indices{^2_j}(\diff_{\bar{z}})\,w^j\frac{1}{w^2}\right)\diff_\zeta
\end{split}
\end{equation*}
while, if $\hat{v}_2=\tilde{V}\diff_{\bar{z}}+\tilde{U}^{\bar{i}}\diff_{\bar{w}^i}$
\begin{equation*}
\begin{split}
p_*A_{\lambda\,x}(\hat{v}_2)&=p_*\left(2\,\I\,\tilde{V}\,\beta\indices{^i_j}(\diff_{\bar{z}})\,w^j\diff_{w^i}\right)=2\,\I\,\tilde{V}\left(-\beta\indices{^1_j}(\diff_{\bar{z}})\,w^j\frac{w^2}{(w^1)^2}+\beta\indices{^2_j}(\diff_{\bar{z}})\,w^j\frac{1}{w^2}\right)\diff_\zeta
\end{split}
\end{equation*}
but if $\hat{v}_1$ and $\hat{v}_2$ have the same image under $p_*$, $V=\tilde{V}$.
\end{proof}
Let $v=v^{\bar{1}}\diff_{\bar{z}}+v^{\bar{2}}\diff_{\bar{\zeta}}\in T^{0,1}_{(z,\zeta)}\bb{P}(E)$, and consider $\hat{v}=v^{\bar{1}}\diff_{\bar{z}}+v^{\bar{2}}\diff_{\bar{w}^2}\in T^{0,1}_{(z,1,\zeta)}(E)$. By our definition,
\begin{equation*}
\begin{split}
p_*A(\hat{v})=&2\,\I\,v^{\bar{1}}\left(-\beta\indices{^1_1}(\diff_{\bar{z}})\,\zeta-\beta\indices{^1_2}(\diff_{\bar{z}})\,\zeta^2+\beta\indices{^2_1}(\diff_{\bar{z}})+\beta\indices{^2_2}(\diff_{\bar{z}})\,\zeta\right)\diff_\zeta.
\end{split}
\end{equation*}
So, if we denote still by $A$ the deformation of the complex structure of $\bb{P}(E)$ we have
\begin{equation}\label{eq:deformazione_proiettivo}
A^{1,0}=2\,\I\left[(\beta\indices{_{\bar{1}}^2_2}-\beta\indices{_{\bar{1}}^1_1})\,\zeta-\beta\indices{_{\bar{1}}^1_2}\,\zeta^2+\beta\indices{_{\bar{1}}^2_1}\right]\mrm{d}\bar{z}\otimes\diff_\zeta.
\end{equation}

Notice that when we decompose $\beta\in\m{A}^{0,1}(\m{O}\oplus L)$ as
\begin{equation*}
\beta=\begin{pmatrix}
\beta\indices{^1_1} & \beta\indices{^1_2}\\
\beta\indices{^2_1} &\beta\indices{^2_2}
\end{pmatrix}
\end{equation*}
then $\beta\indices{^1_1}\in\m{A}^{0,1}(\m{O})\cong\m{A}^{0,1}(\Sigma,\bb{C})$, $\beta\indices{^1_2}\in\m{A}^{0,1}(L^*)$, $\beta\indices{^2_1}\in\m{A}^{0,1}(L)$ and $\beta\indices{^2_2}\in\m{A}^{0,1}(\mrm{End}(L))\cong\m{A}^{0,1}(\Sigma,\bb{C})$.

The expression \eqref{eq:deformazione_proiettivo} for $A^{1,0}$ holds just on the set $\bb{P}(\m{O}\oplus L)\setminus\Sigma_\infty$. If instead we change coordinates to $\bb{P}(\m{O}\oplus L)\setminus\Sigma_0$, we simply have to exchange the roles of $\beta\indices{^1_2}$ and $\beta\indices{^2_1}$. Indeed, equation \eqref{eq:deformazione_proiettivo} was obtained by fixing a system of bundle-adapted holomorphic coordinates $(z,\zeta)$ on $L$; if we perform the change of variables $\eta=\zeta^{-1}$ we obtain
\begin{equation*}
A^{1,0}=-2\,\I\left[\left(\beta\indices{_{\bar{1}}^2_2}-\beta\indices{_{\bar{1}}^1_1}\right)\eta-\beta\indices{_{\bar{1}}^1_2}+\beta\indices{_{\bar{1}}^2_1}\eta^2\right]\mrm{d}\bar{z}\otimes\diff_{\eta}.
\end{equation*}
After all, the construction of $\bb{P}(\m{O}\oplus L)$ can be interpreted as glueing the total spaces of $L$ and $L^*$ along their open subsets $L\setminus\Sigma$ and $L^*\setminus\Sigma$.

\bigskip

\begin{nota}\label{nota:surf_eq_complex_nonequiv}
Our choice of deformation of the complex structure $A$ is not compatible with $\omega_\phi$ for any $\phi$. Indeed $A^2=A^{1,0}A^{0,1}+A^{0,1}A^{1,0}=0$, and if $A$ and $\omega_\phi$ were compatible then we would find 
\begin{equation*}
\norm{A}^2_{g_\phi}=\mrm{Tr}(A^2)=0
\end{equation*}
but $A\not=0$. Hence, in this Section we study the complexified equations
\begin{equation*}
\begin{cases}
\f{m}_{\bm{\Omega_I}}\left(\omega,A(\beta)\right)=0;\\
\f{m}_{\bm{\Theta}}\left(\omega,A(\beta)\right)=0.
\end{cases}
\end{equation*}
for $A(\beta)$ as in \ref{eq:deformazione_proiettivo}, and so we'll find a solution to the complexified system \eqref{eq:HCSCK_complex_relaxed} without the compatibility condition.

Hence, we are tacitly assuming that we have extended the moment maps $\f{m}_{\bm{\Omega_I}}$, $\f{m}_{\bm{\Theta}}$ to the space of metrics $g$ for which $g(\alpha^\transpose-,-)$ is not necessarily symmetric. We have shown above that ${\f{m}_{\bm{\Theta}}}_{(J,\alpha)}(h) = \left\langle h,-\mrm{div}\left(\bdiff^*\bar{\alpha}^\transpose\right)\right\rangle$, which clearly has a tautological extension to all $g$ in the K\"ahler class. But the choice of an extension of the real moment map $\f{m}_{\bm{\Omega_I}}$ is more flexible. 

The crucial point is that, by Lemma \ref{lemma:mappa_momento_OmegaI}, $\f{m}_{\bm{\Omega_I}}$ is computed in terms of a spectral function of $A = \mrm{Re}(\alpha^\transpose)$. This function can be expressed in several different, equivalent ways by using a compatible metric $g$, that is, one for which $g(\alpha^\transpose-,-)$ is symmetric. In our present situation where this compatibility condition might not hold, these equivalent expressions give rise to potentially different extensions of $\f{m}_{\bm{\Omega_I}}$. A simple example is given by the spectral quantity $\mrm{Tr}(A^2)$. A computation shows that for \emph{compatible} $g$ this may be expressed equivalently as $\norm{A}^2_g$. So when $g$ and $A$ are not compatible $\norm{A}^2_g$ gives an alternative extension of the spectral quantity $\mrm{Tr}(A^2)$.
\end{nota}

\paragraph*{Choice of complexification.} The two expressions appearing in \eqref{eq:low_rank} were derived in close analogy to the case of curves. However in the present case they are no longer equivalent, as we discussed in Remark \ref{nota:surf_eq_complex_nonequiv}. This leads to a few different possibilities for the formal complexification. In the rest of this paper we examine the natural choices given by the two expressions in \eqref{eq:low_rank}. So in terms of the endomorphism $A$ the alternative possibility for the real moment map is   
\begin{equation}\label{eq:low_rank_norm}
\begin{split}
\f{m}(J,\alpha)&=\mrm{div}\left[
-\frac{\mrm{grad}\left(\frac{1}{2}\norm{A^{1,0}}^2_g\right)}{1+\sqrt{1-\frac{1}{2}\norm{A^{1,0}}^2_g}}+ \frac{g(\nabla^aA^{0,1},A^{1,0})\diff_a+\mrm{c.c.}}{1+\sqrt{1-\frac{1}{2}\norm{A^{1,0}}^2_g}}- \nabla^*\left(\frac{A^2}{1+\sqrt{1-\frac{1}{2}\norm{A^{1,0}}^2_g}}\right)
\right].
\end{split}
\end{equation}

\subsection{The complex moment map}

In this Section we'll find sufficient conditions on $\beta\in\m{A}^{0,1}(\mrm{End}(\m{O}\oplus L))$ such that the pair $\left(\omega_\phi,A(\beta)\right)$ satisfies the complex moment map equation. We work with a fixed metric $\omega_\phi$ for a prescribed (arbitrary) momentum profile $\phi$. 

Our strategy is to carry out the necessary computations in \emph{transversally normal local coordinates} and \emph{without} assuming that $A=A(\beta)$, but rather for some arbitrary $A^{1,0}=A\indices{^2_{\bar{1}}}\mrm{d}\bar{z}\otimes\diff_\zeta$. At the end of this Section we show that, when $L$ is the anticanonical bundle and for suitable choices of $A=A(\beta)$, our computations actually globalise to the whole ruled surface.

Recall that, for a deformation of complex structures $\dot{J}_0$ and a K\"ahler form $\omega$, the complex moment map equation is
\begin{equation*}
\mrm{div}\left(\bdiff^*\!\dot{J}_0^{1,0}\right)=0.
\end{equation*}
\begin{lemma}
With the previous notation,
\begin{equation*}
\bdiff^*A^{1,0}=-\frac{\phi(\tau)}{\zeta\,g_0\,(1+\tau)^2}A\indices{^2_{\bar{1}}}\diff_z-\frac{1}{(1+\tau)g_0}\left(\diff_zA\indices{^2_{\bar{1}}}+A\indices{^2_{\bar{1}}}\,\diff_zt\,\left(1-\frac{\phi(\tau)}{1+\tau}\right)-\zeta\,\diff_zt\,\diff_{\zeta}A\indices{^2_{\bar{1}}}\right)\diff_\zeta.
\end{equation*}
\end{lemma}

\begin{proof}
It's just a matter of computing carefully, starting from
\begin{equation*}
\bdiff^*A^{1,0}=-g^{a\bar{b}}\nabla_aA\indices{^c_{\bar{b}}}\,\diff_c.
\end{equation*}
The covariant derivatives of $A$ satisfy
\begin{align*}
\nabla_1A\indices{^1_{\bar{1}}}&=A\indices{^2_{\bar{1}}}\Gamma^1_{12};      \quad\quad&&\nabla_2A\indices{^1_{\bar{1}}}=A\indices{^2_{\bar{1}}}\Gamma^1_{22}=0;\\
\nabla_1A\indices{^1_{\bar{2}}}&=0;      \quad\quad&&\nabla_2A\indices{^1_{\bar{2}}}=0;\\
\nabla_1A\indices{^2_{\bar{1}}}&=\diff_zA\indices{^2_{\bar{1}}}+A\indices{^2_{\bar{1}}}\Gamma^2_{21};      \quad\quad&&\nabla_2A\indices{^2_{\bar{1}}}=\diff_{\zeta}A\indices{^2_{\bar{1}}}+A\indices{^2_{\bar{1}}}\Gamma^2_{22};\\
\nabla_1A\indices{^2_{\bar{2}}}&=0;      \quad\quad&&\nabla_2A\indices{^2_{\bar{2}}}=0.
\end{align*}
By \eqref{christoffel} we can rewrite $\bdiff^*A^{1,0}$ as
\begin{equation*}
\begin{split}
\bdiff^*A^{1,0}=&-g^{1\bar{1}}\nabla_1A\indices{^c_{\bar{1}}}\,\diff_c-g^{2\bar{1}}\nabla_2A\indices{^2_{\bar{1}}}\,\diff_\zeta=-g^{1\bar{1}}\nabla_1A\indices{^1_{\bar{1}}}\diff_z-\left(g^{1\bar{1}}\nabla_1A\indices{^2_{\bar{1}}}+g^{2\bar{1}}\nabla_2A\indices{^2_{\bar{1}}}\right)\diff_\zeta=\\
=&-\frac{\phi(\tau)}{\zeta\,g_0\,(1+\tau)^2}A\indices{^2_{\bar{1}}}\diff_z-\frac{1}{(1+\tau)g_0}\left(\diff_zA\indices{^2_{\bar{1}}}+A\indices{^2_{\bar{1}}}\,\diff_zt\,\left(1-\frac{\phi(\tau)}{1+\tau}\right)-\zeta\,\diff_zt\,\diff_{\zeta}A\indices{^2_{\bar{1}}}\right)\diff_\zeta.
\end{split}
\end{equation*}
\end{proof}
We proceed to calculate the divergence of $\bdiff^*A^{1,0}$. By definition
\begin{equation*}
\mrm{div}(\bdiff^*A^{1,0})=\nabla_a(\bdiff^*A^{1,0})^a=\diff_a(\bdiff^*A^{1,0})^a+(\bdiff^*A^{1,0})^b\Gamma^a_{ab}.
\end{equation*}
We compute the two terms separately. We will need the quantities
\begin{equation*}
\begin{split}
D_1(\tau):=&-\frac{\phi(\tau)}{(1+\tau)^2}=\phi(\tau)\,\diff_\tau\left(\frac{1}{1+\tau}\right)\\
D_2(\tau):=&\phi(\tau)\,\diff_\tau D_1(\tau).
\end{split}
\end{equation*}
The first term is the sum of 
\begin{equation*}
\diff_1(\bdiff^*A^{1,0})^1=\diff_z\left(D_1(\tau)\,\frac{A\indices{^2_{\bar{1}}}}{\zeta\,g_0}\right)=D_2(\tau)\,\frac{\diff_zt}{\zeta\,g_0}A\indices{^2_{\bar{1}}}-D_1(\tau)\,\frac{\Gamma^1_{11}(\Sigma)}{\zeta\,g_0}A\indices{^2_{\bar{1}}}+D_1(\tau)\,\frac{1}{\zeta\,g_0}\diff_z A\indices{^2_{\bar{1}}} 
\end{equation*}
and
\begin{equation*}
\begin{split}
\diff_2(\bdiff^*A^{1,0})^2=&\diff_\zeta\left(-\frac{\diff_zA\indices{^2_{\bar{1}}}}{(1+\tau)g_0}-\frac{A\indices{^2_{\bar{1}}}\,\diff_zt}{(1+\tau)g_0}-D_1(\tau)\,\frac{A\indices{^2_{\bar{1}}}\,\diff_zt}{g_0}+\frac{\zeta\,\diff_zt}{(1+\tau)g_0}\diff_{\zeta}A\indices{^2_{\bar{1}}}\right)=\\
=&-D_1(\tau)\,\frac{\diff_zA\indices{^2_{\bar{1}}}}{\zeta\,g_0}-\frac{\diff_\zeta\diff_zA\indices{^2_{\bar{1}}}}{(1+\tau)g_0}
-D_1(\tau)\frac{A\indices{^2_{\bar{1}}}\,\diff_zt}{\zeta\,g_0}-D_2(\tau)\,\frac{A\indices{^2_{\bar{1}}}\,\diff_zt}{\zeta\,g_0}+\frac{\zeta\,\diff_zt}{(1+\tau)g_0}\diff_{\zeta}\diff_{\zeta}A\indices{^2_{\bar{1}}}.
\end{split}
\end{equation*}
The sum is given by
\begin{equation*}
\begin{split}
\diff_a(\bdiff^*A^{1,0})^a=&-D_1(\tau)\,\frac{\Gamma^1_{11}(\Sigma)}{\zeta\,g_0}A\indices{^2_{\bar{1}}}
-\frac{\diff_\zeta\diff_zA\indices{^2_{\bar{1}}}}{(1+\tau)g_0}
-D_1(\tau)\,\frac{A\indices{^2_{\bar{1}}}\,\diff_zt}{\zeta\,g_0}+\frac{\zeta\,\diff_zt}{(1+\tau)g_0}\diff_{\zeta}\diff_{\zeta}A\indices{^2_{\bar{1}}}=\\
=&-(\bdiff^*A^{1,0})^1\Gamma^1_{11}(\Sigma)
-\frac{\diff_\zeta\diff_zA\indices{^2_{\bar{1}}}}{(1+\tau)g_0}
-D_1(\tau)\,\frac{A\indices{^2_{\bar{1}}}\,\diff_zt}{\zeta\,g_0}+\frac{\zeta\,\diff_zt}{(1+\tau)g_0}\diff_{\zeta}\diff_{\zeta}A\indices{^2_{\bar{1}}}.
\end{split}
\end{equation*}
On the other hand the second term in $\mrm{div}(\bdiff^*A^{1,0})$ is given by
\begin{equation*}
\begin{split}
(\bdiff^*A^{1,0})^b\Gamma^a_{ab}=&(\bdiff^*A^{1,0})^1\Gamma^1_{11}+(\bdiff^*A^{1,0})^1\Gamma^2_{21}+(\bdiff^*A^{1,0})^2\Gamma^1_{12}+(\bdiff^*A^{1,0})^2\Gamma^2_{22}=\\
=&(\bdiff^*A^{1,0})^1\Gamma^1_{11}(\Sigma)+D_1(\tau)\,\frac{\diff_zt}{\zeta\,g_0}\,A\indices{^2_{\bar{1}}}\,\left(\phi'(\tau)+\frac{\phi(\tau)}{1+\tau}\right)+D_1(\tau)\,\frac{\diff_zA\indices{^2_{\bar{1}}}}{\zeta\,g_0}+\\
&+D_1(\tau)\,\frac{\diff_zt}{\zeta\,g_0}A\indices{^2_{\bar{1}}}-D_1(\tau)\,\frac{\phi(\tau)\,\diff_zt}{\zeta\,g_0}A\indices{^2_{\bar{1}}}-D_1(\tau)\,\frac{\diff_zt\,\diff_\zeta A\indices{^2_{\bar{1}}}}{g_0}+(\bdiff^*A^{1,0})^2\frac{\phi'(\tau)-1}{\zeta}.
\end{split}
\end{equation*}
These computations show that we have 
\begin{equation*}
\begin{split}
\mrm{div}(\bdiff^*A^{1,0})=&-\frac{\diff_\zeta\diff_zA\indices{^2_{\bar{1}}}}{(1+\tau)g_0}+\frac{\zeta\,\diff_zt}{(1+\tau)g_0}\diff_{\zeta}\diff_{\zeta}A\indices{^2_{\bar{1}}}-\frac{1}{(1+\tau)g_0\,\zeta}\left(\frac{\phi(\tau)}{1+\tau}+\phi'(\tau)-1\right)\diff_zA\indices{^2_{\bar{1}}}+\\
&+\left(\frac{\phi(\tau)}{1+\tau}+\phi'(\tau)-1\right)\frac{\diff_zt\,\diff_\zeta A\indices{^2_{\bar{1}}}}{(1+\tau)g_0}-\frac{\diff_zt\,A\indices{^2_{\bar{1}}}}{(1+\tau)\zeta\,g_0}\left(\phi'(\tau)-1+\frac{\phi(\tau)}{1+\tau}\right)=\\
=&\frac{-\diff_\zeta\diff_zA\indices{^2_{\bar{1}}}+\zeta\,\diff_zt\,\diff_{\zeta}\diff_{\zeta}A\indices{^2_{\bar{1}}}}{(1+\tau)g_0}-\frac{1}{(1+\tau)g_0}\left(-\frac{\diff_zA\indices{^2_{\bar{1}}}}{\zeta}+\diff_zt\,\diff_\zeta A\indices{^2_{\bar{1}}}-\frac{\diff_zt\,A\indices{^2_{\bar{1}}}}{\zeta}\right)+\\
&+\frac{1}{(1+\tau)g_0}\diff_\zeta\left[\mrm{log}\,\phi(\tau)(1+\tau)\right]\left(-\diff_zA\indices{^2_{\bar{1}}}+\diff_zt\,\zeta\diff_\zeta A\indices{^2_{\bar{1}}}-\diff_zt\,A\indices{^2_{\bar{1}}}\right).
\end{split}
\end{equation*}
This quantity vanishes precisely when
\begin{equation*}
\begin{split}
-\zeta\diff_\zeta\diff_zA\indices{^2_{\bar{1}}}&+\zeta^2\,\diff_zt\,\diff_{\zeta}\diff_{\zeta}A\indices{^2_{\bar{1}}}-\left(-\diff_zA\indices{^2_{\bar{1}}}+\diff_zt\,\zeta\,\diff_\zeta A\indices{^2_{\bar{1}}}-\diff_zt\,A\indices{^2_{\bar{1}}}\right)+\\&+\zeta\diff_\zeta\left[\mrm{log}\,\phi(\tau)(1+\tau)\right]\left(-\diff_zA\indices{^2_{\bar{1}}}+\diff_zt\,\zeta\,\diff_\zeta A\indices{^2_{\bar{1}}}-\diff_zt\,A\indices{^2_{\bar{1}}}\right)=0.
\end{split}
\end{equation*}
Notice that
\begin{equation*}
-\zeta\diff_\zeta\diff_zA\indices{^2_{\bar{1}}}+\zeta^2\,\diff_zt\,\diff_{\zeta}\diff_{\zeta}A\indices{^2_{\bar{1}}}=\zeta\diff_\zeta\left(-\diff_zA\indices{^2_{\bar{1}}}+\diff_zt\,\zeta\diff_\zeta A\indices{^2_{\bar{1}}}-\diff_zt\,A\indices{^2_{\bar{1}}}\right).
\end{equation*}
Thus, introducing the \emph{locally defined function} 
\begin{equation}\label{eq:def_k}
k:=-\diff_zA\indices{^2_{\bar{1}}}+\diff_zt\,\zeta\diff_\zeta A\indices{^2_{\bar{1}}}-\diff_zt\,A\indices{^2_{\bar{1}}},
\end{equation}
the complex moment map equation $\mrm{div}(\bdiff^*A^{1,0})=0$ may be expressed \emph{locally} as
\begin{equation*}
\zeta\diff_\zeta k+\left(\zeta\diff_\zeta\left[\mrm{log}\,\phi(\tau)(1+\tau)\right]-1\right)k=0.
\end{equation*}
This condition can be rewritten as
\begin{equation*}
\diff_\zeta k+\diff_\zeta\left(\mrm{log}\frac{\phi(\tau)(1+\tau)}{\zeta\,\bar{\zeta}}\right)k=0.
\end{equation*}
This equation can be integrated; so we see that the equation $\mrm{div}(\bdiff^*A^{1,0})=0$ is satisfied locally if and only if the function $k$ defined by equation \eqref{eq:def_k} satisfies
\begin{equation}\label{eq:complex_mm_k_sol}
k=c\,\frac{\zeta\,\bar{\zeta}}{\phi(\tau)(1+\tau)}
\end{equation}
for some function $c=c(z,\zeta)$ such that $\diff_\zeta c=0$.

\paragraph*{Choosing $c=0$.} Let's consider the case in which the function $c$ in \eqref{eq:complex_mm_k_sol} is identically $0$. In this case, $A^{1,0}$ satisfies 
\begin{equation}\label{eq:complessa_c0}
-\diff_zA\indices{^2_{\bar{1}}}+\zeta\,\diff_zt\,\diff_\zeta A\indices{^2_{\bar{1}}}-\diff_zt\,A\indices{^2_{\bar{1}}}=0.
\end{equation}
If we now choose $A=A(\beta)$, i.e.
\begin{equation*}
A\indices{^2_{\bar{1}}}=2\I\left(\zeta(\beta\indices{_{\bar{1}}^2_2}-\beta\indices{_{\bar{1}}^1_1})-\zeta^2\beta\indices{_{\bar{1}}^1_2}+\beta\indices{_{\bar{1}}^2_1}\right)
\end{equation*}
for $\beta\indices{^1_1},\beta\indices{^2_2}\in\m{A}^{0,1}(\Sigma,\bb{C})$, $\beta\indices{^1_2}\in\m{A}^{0,1}(L^*)$, $\beta\indices{^2_1}\in\m{A}^{0,1}(L)$, we can get an interesting consequence from equation \eqref{eq:complessa_c0}. Indeed, on the divisor $\Sigma=\Sigma_0=\set{\zeta=0}$ we get, from equation \eqref{eq:complessa_c0}
\begin{equation*}
-\diff_z\beta\indices{_{\bar{1}}^2_1}-\diff_zt\,\beta\indices{_{\bar{1}}^2_1}=0
\end{equation*}
and recalling that $\diff_zt=\diff_z\mrm{log}(a(z))$, were $a(z)$ is the local representative of the fibre metric on $L$, this tells us that
\begin{equation*}
\beta\indices{_{\bar{1}}^2_1}=\frac{q(z)}{a(z)}
\end{equation*}
for some function $q$ over $\Sigma$ such that $\diff_zq=0$. Consider instead what equation \eqref{eq:complessa_c0} tells us for $\zeta=\infty$, i.e. on the zero-set of $\eta=\zeta^{-1}$; after the change of coordinates, equation \eqref{eq:complessa_c0} becomes  
\begin{equation*}
\diff_zA\indices{^2_{\bar{1}}}(\eta)+\diff_zt\left(\eta\diff_\eta A\indices{^2_{\bar{1}}}-A\indices{^2_{\bar{1}}}(\eta)\right)=0
\end{equation*}
where $A\indices{^2_{\bar{1}}}(\eta)=-2\I\left(\eta(\beta\indices{_{\bar{1}}^2_2}-\beta\indices{_{\bar{1}}^1_1})-\beta\indices{_{\bar{1}}^1_2}+\eta^2\,\beta\indices{_{\bar{1}}^2_1}\right)$. Setting $\eta=0$ we find
\begin{equation*}
\diff_z\beta\indices{_{\bar{1}}^1_2}-\diff_zt\,\beta\indices{_{\bar{1}}^1_2}=0
\end{equation*}
and so
\begin{equation*}
\beta\indices{_{\bar{1}}^1_2}=a(z)\,\tilde{q}(z)
\end{equation*}
for some function $\tilde{q}$ over $\Sigma$ such that $\diff_z\tilde{q}=0$. With these choices, the matrix associated to $\beta\in\m{A}^{0,1}(\mrm{End}(\m{O}\oplus L))$ in a local holomorphic frame for $L$ is
\begin{equation*}
\begin{pmatrix}
\beta\indices{^1_1} & \tilde{q}(z)\,a(z)\,\mrm{d}\bar{z}\\
\frac{q(z)}{a(z)}\,\mrm{d}\bar{z} & \beta\indices{^2_2}
\end{pmatrix}.
\end{equation*}
It is useful to notice the identity $\zeta\,\diff_\zeta A\indices{^2_{\bar{1}}}-A\indices{^2_{\bar{1}}}=-2\I\left(\zeta^2\,\beta\indices{_{\bar{1}}^1_2}+\beta\indices{_{\bar{1}}^2_1}\right)$. Plugging this into \eqref{eq:complessa_c0} the equation can be rewritten as
\begin{equation*}
-\zeta\diff_z\left(\beta\indices{_{\bar{1}}^2_2}-\beta\indices{_{\bar{1}}^1_1}\right)-\zeta^2\,\diff_z\beta\indices{_{\bar{1}}^1_2}+\diff_z\beta\indices{_{\bar{1}}^2_1}-\diff_zt\left(\zeta^2\,\beta\indices{_{\bar{1}}^1_2}+\beta\indices{_{\bar{1}}^2_1}\right)=0,
\end{equation*}
which reduces to
\begin{equation*}
\diff_z\left(\beta\indices{_{\bar{1}}^2_2}-\beta\indices{_{\bar{1}}^1_1}\right)=0.
\end{equation*}
So equation \eqref{eq:complessa_c0} is satisfied if and only if
\begin{equation}\label{eq:condizione_coeff_c0}
\begin{split}
\beta\indices{_{\bar{1}}^1_2}&=a(z)\,\tilde{q}(z)\quad\mbox{with }\diff_z\tilde{q}=0;\\
\beta\indices{_{\bar{1}}^2_1}&=\frac{q(z)}{a(z)}\quad\mbox{with }\diff_zq=0;\\
\diff\big(\beta\indices{^2_2}&-\beta\indices{^1_1}\big)=0.
\end{split}
\end{equation}

The first two conditions in equation \eqref{eq:condizione_coeff_c0} are still just local ones. However we can globalise them by choosing $L$ to be the anticanonical bundle of $\Sigma$, $L=T^{1,0}\Sigma$. Indeed, recall that $\beta\indices{^1_2}\in\m{A}^{0,1}(L^*)$, $\beta\indices{^2_1}\in\m{A}^{0,1}(L)$, so that if $L=T^{1,0}\Sigma$ then $\beta\indices{^1_2}$ must be an element of $\m{A}^{0,1}({T^{1,0}}^*\Sigma)$, while $\beta\indices{^2_1}$ must be an element of $\m{A}^{0,1}(T^{1,0}\Sigma)$. Then we can choose the quantity $\tilde{q}$ of equation \eqref{eq:condizione_coeff_c0} to be a constant, and the local condition on $\beta\indices{^1_2}$ becomes the the global condition $\beta\indices{^1_2}=\tilde{q}\,h$. This is compatible with $\beta\indices{^1_2}\in\m{A}^{0,1}({T^{1,0}}^*\Sigma)$, since $h$ is a Hermitian metric on $T^{1,0}\Sigma$. In the same way, if $q$ is the local representative of a global holomorphic quadratic differential on $\Sigma$ (that we denote still by $q$), then the local condition on $\beta\indices{^2_1}$ globalises to $\beta\indices{^2_1}=q^{\sharp_h}$, i.e. $\beta\indices{^2_1}$ should be the quadratic differential with one index raised by $h$.\\

Let us summarise the results of this Section. Suppose that $L=K_\Sigma^*=T^{1,0}\Sigma$ and that $\beta$ satisfies the globally defined equations
\begin{equation}\label{eq:cond_globali_coeff_c0}
\begin{split}
\beta\indices{^1_2}&=\tilde{q}\,h\quad\mbox{ for some constant }\tilde{q};\\
\beta\indices{^2_1}&=q^{\sharp_h}\quad\mbox{ for some holomorphic quadratic differential }q;\\
\diff\big(\beta\indices{^2_2}&-\beta\indices{^1_1}\big)=0.
\end{split}
\end{equation}
The the complex moment map equation is satisfied. From now we always assume that $L$, $\beta$ are of this form.

\subsection{The real moment map}\label{sec:ruled_surface_realmm}

In this section we will prove that there exists a solution to the HcscK equations on our ruled surface, at least when the fibres have sufficiently small volume. We will work with the two possible choices of formal complexification given by the expressions in \eqref{eq:low_rank}. First we reformulate Theorem \ref{teorema:solution_mm_reale} using the notation introduced in the last few sections.
\begin{teorema}\label{teorema:HCSCK_rigata_esistenza}
Let $\Sigma$ be a Riemann surface of genus $g\geq 2$, and consider the ruled surface $M=\bb{P}(\m{O}\oplus K_\Sigma^*)$. Then, for all sufficiently small $m>0$, there exists a K\"ahler metric $\omega$ in the class dual to $2\,\pi\left(\m{C}+m\,\Sigma_\infty\right)$ (see Lemma \ref{lemma:classe_phi}) and a ``Higgs field" $\alpha\in\mrm{Hom}({T^{1,0}}^*M,{T^{0,1}}^*M)$ such that the complex and real moment map equations
\begin{equation*}
\mrm{div}\left(\bdiff^*\bar{\alpha}^\transpose\right)=0
\end{equation*}
\begin{equation*}
2\,s(\omega)-2\,\hat{s}(\omega)+\f{m}(\omega,\mrm{Re}(\alpha^\transpose))=0
\end{equation*}
 are satisfied, with $\f{m}$ given by one of the expressions in \eqref{eq:low_rank}.
\end{teorema}
We will choose $A=\mrm{Re}(\alpha^\transpose) = A(\beta)$, for a form $\beta\in\m{A}^{0,1}(\mrm{End}(\m{O}\oplus L))$. Then the complex moment map equation holds provided $\beta$ satisfies the conditions \eqref{eq:cond_globali_coeff_c0}. 

Note that, for any $\beta$ and with $A=A(\beta)$, we know that $A^{1,0}=A\indices{^2_{\bar{1}}}\mrm{d}\bar{z}\otimes\diff_\zeta$ and so the matrix associated to $A^{1,0}$ in a system of bundle-adapted coordinates has the form $\begin{pmatrix} 0 & 0\\ * & 0\end{pmatrix}$. In particular we are in the low-rank situation described at the end of Section \ref{sec:moment_map_complex_surface}, as required by the statement of Theorem \ref{teorema:HCSCK_rigata_esistenza}.

We will present the details of the proof of Theorem \ref{teorema:HCSCK_rigata_esistenza} for the choice of complexification given by the first expression in \eqref{eq:low_rank}, namely
\begin{equation*}
\begin{split}
\f{m}(\omega_\phi,A)=&\mrm{div}_\phi\left[
-\frac{\mrm{grad}_\phi\left(\frac{1}{4}\mrm{Tr}(A^2)\right)}{1+\sqrt{1-\frac{1}{4}\mrm{Tr}(A^2)}}+\frac{g_\phi(\nabla_\phi^aA^{0,1},A^{1,0})\diff_a+\mrm{c.c.}}{1+\sqrt{1-\frac{1}{4}\mrm{Tr}(A^2)}}-\nabla_\phi^*\left(\frac{A^2}{1+\sqrt{1-\frac{1}{4}\mrm{Tr}(A^2)}}\right)\right].
\end{split}
\end{equation*}
The proof for the alternative complexified equation \eqref{eq:low_rank_norm} is essentially the same, but some of the computations are more involved. We will point out the key differences in the course of the proof.

We note that $A(\beta)$ is nilpotent, with $A(\beta)^2 = 0$, so with our current choice of complexification we find
\begin{equation*}
\begin{split}
\f{m}(\omega_\phi,A(\alpha))=&\frac{1}{2}\mrm{div}_{\phi}\left[g_\phi(\nabla_\phi^aA^{0,1},A^{1,0})\diff_a+\mrm{c.c.}\right].
\end{split}
\end{equation*}

In the rest of this Section we fix $L = K^*_{\Sigma}$ and choose $\beta$ so that the complex moment map vanishes, i.e. we assume that $A^{1,0}$ satisfies equation \eqref{eq:complessa_c0}. Notice that if we fix a point $p\in\bb{P}(\m{O}\oplus L)$ and a system of transversally normal coordinates around this point, equation \eqref{eq:complessa_c0} at the point $p$ simply reads as $\diff_zA\indices{^2_{\bar{1}}}=0$.

\begin{lemma}\label{lemma:calcolo_div}
Assume that $\beta\indices{^2_2}=\beta\indices{^1_1}$ and $\beta\indices{^2_1}=0$, so that the matrix of $1$-forms associated to $\beta$ is upper triangular. Then
\begin{equation*}
\mrm{div}_{\phi}\left[g_\phi(\nabla^aA^{0,1},A^{1,0})\diff_a+\mrm{c.c.}\right]=2\,\norm{A^{1,0}}^2_\phi\left(\phi''(\tau)+\frac{(\phi'(\tau)+1)^2}{\phi(\tau)}\right).
\end{equation*}
\end{lemma}

\begin{proof}
We fix a point $p\in\bb{P}(\m{O}\oplus L)$ and a system of transversally normal coordinates $(z,\zeta)$ at this point. All of the following computations will be carried out at $p$.

From the definition we have
\begin{equation*}
\begin{split}
\mrm{div}_{\phi}\left[g_\phi(\nabla^aA^{0,1},A^{1,0})\diff_a\right]=&\nabla_a\left[g^{a\bar{b}}\nabla_{\bar{b}}A\indices{^{\bar{c}}_d}\,A\indices{^e_{\bar{f}}}g^{d\bar{f}}g_{e\bar{c}}\right]=\\
=&
g^{a\bar{b}}g^{d\bar{f}}g_{e\bar{c}}\nabla_a\nabla_{\bar{b}}A\indices{^{\bar{c}}_d}\,A\indices{^e_{\bar{f}}}+
g^{a\bar{b}}g^{d\bar{f}}g_{e\bar{c}}\nabla_{\bar{b}}A\indices{^{\bar{c}}_d}\,\nabla_aA\indices{^e_{\bar{f}}}
\end{split}
\end{equation*}
We proceed to examine the two terms. 

Using the fact that we are in transversally normal coordinates and that the only possibly non-vanishing component of $A^{1,0}$ is $A\indices{^2_{\bar{1}}}$, we can write the first term as
\begin{equation*}
\begin{split}
g^{a\bar{b}}g^{d\bar{f}}g_{e\bar{c}}\nabla_a\nabla_{\bar{b}}A\indices{^{\bar{c}}_d}\,A\indices{^e_{\bar{f}}}=&g^{a\bar{b}}g^{1\bar{1}}g_{2\bar{2}}\nabla_a\nabla_{\bar{b}}A\indices{^{\bar{2}}_1}\,A\indices{^2_{\bar{1}}}=\\
=&g^{1\bar{1}}g^{1\bar{1}}g_{2\bar{2}}\nabla_1\nabla_{\bar{1}}A\indices{^{\bar{2}}_1}\,A\indices{^2_{\bar{1}}}+g^{1\bar{1}}\nabla_2\nabla_{\bar{2}}A\indices{^{\bar{2}}_1}\,A\indices{^2_{\bar{1}}}.
\end{split}
\end{equation*}
A quick computation using equation \eqref{eq:complessa_c0} and the properties of the special system of coordinates gives
\begin{equation*}
\nabla_1\nabla_{\bar{1}}A\indices{^{\bar{2}}_1}=g_0\left(\bar{\zeta}\,\diff_{\bar{\zeta}}A\indices{^{\bar{2}}_1}-A\indices{^{\bar{2}}_1}+A\indices{^{\bar{2}}_1}\left(\phi'(\tau)-\frac{\phi(\tau)}{1+\tau}\right)\right)
\end{equation*}
and
\begin{equation*}
\nabla_2\nabla_{\bar{2}}A\indices{^{\bar{2}}_1}=A\indices{^{\bar{2}}_1}\frac{\phi(\tau)}{\zeta\,\bar{\zeta}}\left(\phi''(\tau)-\frac{\phi'(\tau)-1}{1+\tau}\right)-\diff_{\bar{\zeta}}A\indices{^{\bar{2}}_1}\frac{\phi(\tau)}{(1+\tau)\zeta}.
\end{equation*}
Hence the first term is
\begin{equation*}
\begin{split}
g^{a\bar{b}}g^{d\bar{f}}g_{e\bar{c}}\nabla_a\nabla_{\bar{b}}A\indices{^{\bar{c}}_d}\,A\indices{^e_{\bar{f}}}=&\frac{\phi(\tau)}{(1+\tau)^2\,\zeta\,\bar{\zeta}\,g_0}\left(\bar{\zeta}\,\diff_{\bar{\zeta}}A\indices{^{\bar{2}}_1}-A\indices{^{\bar{2}}_1}+A\indices{^{\bar{2}}_1}\left(\phi'(\tau)-\frac{\phi(\tau)}{1+\tau}\right)\right)\,A\indices{^2_{\bar{1}}}+\\
&+\frac{1}{(1+\tau)g_0}\left(A\indices{^{\bar{2}}_1}\frac{\phi(\tau)}{\zeta\,\bar{\zeta}}\left(\phi''(\tau)-\frac{\phi'(\tau)-1}{1+\tau}\right)-\diff_{\bar{\zeta}}A\indices{^{\bar{2}}_1}\frac{\phi(\tau)}{(1+\tau)\zeta}\right)\,A\indices{^2_{\bar{1}}}=\\
=&\norm{A^{1,0}}^2_\phi\left(\phi''(\tau)-\frac{\phi(\tau)}{(1+\tau)^2}\right).
\end{split}
\end{equation*}
On the other hand for the second term we have
\begin{equation*}
\begin{split}
g^{a\bar{b}}g^{d\bar{f}}g_{e\bar{c}}\nabla_{\bar{b}}A\indices{^{\bar{c}}_d}\,\nabla_aA\indices{^e_{\bar{f}}}=&g^{1\bar{1}}\nabla_{\bar{1}}A\indices{^{\bar{1}}_1}\nabla_1A\indices{^1_{\bar{1}}}+g^{1\bar{1}}\nabla_{\bar{2}}A\indices{^{\bar{2}}_1}\nabla_2A\indices{^2_{\bar{1}}}=\\
=&\frac{A\indices{^{\bar{2}}_1}\,A\indices{^2_{\bar{1}}}}{(1+\tau)g_0}\frac{\phi(\tau)^2}{(1+\tau)^2\,\zeta\,\bar{\zeta}}+\frac{1}{(1+\tau)g_0}\left(\diff_{\bar{\zeta}}A\indices{^{\bar{2}}_1}+A\indices{^{\bar{2}}_1}\Gamma^{\bar{2}}_{\bar{2}\bar{2}}\right)\left(\diff_\zeta A\indices{^2_{\bar{1}}}+A\indices{^2_{\bar{1}}}\Gamma^2_{22}\right)=\\
=&\norm{A^{1,0}}^2_\phi\frac{\phi(\tau)}{(1+\tau)^2}+\frac{\diff_\zeta\diff_{\bar{\zeta}}\left(A\indices{^2_{\bar{1}}}\,A\indices{^{\bar{2}}_1}\right)}{(1+\tau)g_0}+\\
&+\frac{\phi'(\tau)-1}{(1+\tau)g_0\,\zeta\,\bar{\zeta}}\left(\zeta\,\diff_\zeta\left(A\indices{^{\bar{2}}_1}\,A\indices{^2_{\bar{1}}}\right)+\bar{\zeta}\,\diff_{\bar{\zeta}}\left(A\indices{^{\bar{2}}_1}\,A\indices{^2_{\bar{1}}}\right)\right)+\frac{(\phi'(\tau)-1)^2}{\phi(\tau)}\,\norm{A^{1,0}}^2_\phi.
\end{split}
\end{equation*}

Up to this point of the proof, no assumption was made on the components of $\beta$. However, if we assume that $\beta$ is of the form $\begin{pmatrix} * & **\\ 0 & *\end{pmatrix}$ then
\begin{equation*}
\begin{split}
\zeta\diff_\zeta A\indices{^2_{\bar{1}}}=&2\,A\indices{^2_{\bar{1}}}.
\end{split}
\end{equation*}
So in this case we find
\begin{equation*}
\begin{split}
g^{a\bar{b}}g^{d\bar{f}}g_{e\bar{c}}\nabla_{\bar{b}}A\indices{^{\bar{c}}_d}\,\nabla_aA\indices{^e_{\bar{f}}}=\norm{A^{1,0}}^2_\phi\left(\frac{\phi(\tau)}{(1+\tau)^2}+\frac{(\phi'(\tau)+1)^2}{\phi(\tau)}\right).
\end{split}
\end{equation*}
Putting everything together we get
\begin{equation*}
\begin{split}
\mrm{div}_{\phi}\left[g_\phi(\nabla^aA^{0,1},A^{1,0})\diff_a\right]=&\norm{A^{1,0}}^2_\phi\left(\phi''(\tau)-\frac{\phi(\tau)}{(1+\tau)^2}\right)+\norm{A^{1,0}}^2_\phi\left(\frac{\phi(\tau)}{(1+\tau)^2}+\frac{(\phi'(\tau)+1)^2}{\phi(\tau)}\right)=\\
=&\norm{A^{1,0}}^2_\phi\left(\phi''(\tau)+\frac{(\phi'(\tau)+1)^2}{\phi(\tau)}\right)
\end{split}
\end{equation*}
and 
\begin{equation*}
\begin{split}
\mrm{div}_{\phi}\left[g_\phi(\nabla^aA^{0,1},A^{1,0})\diff_a+\mrm{c.c.}\right]=&2\,\norm{A^{1,0}}^2_\phi\left(\phi''(\tau)+\frac{(\phi'(\tau)+1)^2}{\phi(\tau)}\right).
\end{split}
\end{equation*}
\end{proof}
Notice that, under the assumption of Lemma \ref{lemma:calcolo_div},
\begin{equation*}
\norm{A^{1,0}}^2_\phi=4\frac{\phi(\tau)}{(1+\tau)g_0}\zeta\,\bar{\zeta}\,\card{\beta\indices{_{\bar{1}}^1_2}}^2
\end{equation*}
so that
\begin{equation*}
\frac{1}{2}\mrm{div}_{\phi}\left[g_\phi(\nabla^aA^{0,1},A^{1,0})\diff_a+\mrm{c.c.}\right]=4\frac{\phi(\tau)}{(1+\tau)g_0}\zeta\,\bar{\zeta}\,\card{\beta\indices{_{\bar{1}}^1_2}}^2\left(\phi''(\tau)+\frac{(\phi'(\tau)+1)^2}{\phi(\tau)}\right).
\end{equation*}
However, since we are assuming that $A$ satisfies equation \eqref{eq:complessa_c0}, $\beta$ should satisfy the conditions in equation \eqref{eq:cond_globali_coeff_c0}. So $\beta\indices{_{\bar{1}}^1_2}=\tilde{q}\,a(z)$ for some constant $\tilde{q}$, and
\begin{equation*}
\frac{1}{2}\mrm{div}_{\phi}\left[g_\phi(\nabla^aA^{0,1},A^{1,0})\diff_a+\mrm{c.c.}\right]=4\card{\tilde{q}}^2\frac{\phi(\tau)}{(1+\tau)g_0}a(z)^2\,\card{\zeta}^2\left(\phi''(\tau)+\frac{(\phi'(\tau)+1)^2}{\phi(\tau)}\right).
\end{equation*}

Now recall that we are assuming $L=K(\Sigma)^*=T^{1,0}\Sigma$, and $\zeta$ is a linear coordinate on $L$. We also have the Hermitian metric on the fibres of $L$ whose local representative is $a(z)$. If we choose this metric to be K\"ahler-Einstein, i.e. $a(z)=\lambda\,g_0(z)$ for some positive constant $\lambda$, the equation becomes
\begin{equation*}
\begin{split}
\frac{1}{2}\mrm{div}_{\phi}\left[g_\phi(\nabla^aA^{0,1},A^{1,0})\diff_a+\mrm{c.c.}\right]=&4\card{\tilde{q}}^2\frac{1}{1+\tau}\lambda\,a(z)\,\card{\zeta}^2\left(\phi(\tau)\,\phi''(\tau)+(\phi'(\tau)+1)^2\right)=\\
=&\frac{c}{m^2}\,\frac{\mrm{e}^t}{1+\tau}\left(\phi(\tau)\,\phi''(\tau)+(\phi'(\tau)+1)^2\right)
\end{split}
\end{equation*}
where we are collecting in $\frac{c}{m^2}$ all the various constants. 

We can finally write the zero-locus equation of the real moment map, using Proposition \ref{prop:curv_scal_phi} and Lemma \ref{lemma:media_curvatura}: since we are choosing a metric on $\Sigma$ that has constant scalar curvature equal to $-1$, the equation is 
\begin{equation}\label{eq:mappa_reale}
\phi''(\tau)+\frac{2}{1+\tau}\phi'(\tau)+\frac{1}{1+\tau}+\frac{4}{m(2+m)}=\frac{c}{m^2}\,\frac{\mrm{e}^t}{1+\tau}\left(\phi(\tau)\,\phi''(\tau)+(\phi'(\tau)+1)^2\right).
\end{equation}
(dividing throughout by a factor of $2$). 

The reason for introducing the factor $m^{-2}$ in the equation is that in the next sections we will find a solution of equation \eqref{eq:mappa_reale} in the \emph{adiabatic limit} when $m\to 0$, and to do this we will have to expand the equation with respect to $m$. This $m^{-2}$ factor has been chosen precisely in such a way that the expansion  in $m$ will have the appropriate form.\\

Let us summarise our computations so far. We showed that with all our assumptions, in particular those of Lemma \ref{lemma:calcolo_div}, the complex moment map vanishes automatically, while the real moment map equation reduces to the problem 
\begin{equation}\label{eq:HCSCK_original}
\begin{split}
\phi''(\tau)+2\frac{\phi'(\tau)}{1+\tau}+\frac{1}{1+\tau}+\frac{4}{m(2+m)}&=\frac{c}{m^2}\,\frac{\mrm{e}^t}{1+\tau}\left((\phi'(\tau)+1)^2+\phi(\tau)\,\phi''(\tau)\right)\\
\phi(0)=\phi(m)&=0\\
\phi'(0)=-\phi'(m)&=1
\end{split}
\end{equation}
to be solved for a positive function $\phi(\tau)$ on $[0,m]$ and a positive real number $c$. Here the function $t$ is a primitive of $\frac{1}{\phi(\tau)}$; we might fix the starting point of integration as $m/2$, since the choice of a different point can be absorbed by the constant $c$. From now on then we'll consider $t$ as
\begin{equation*}
t(\tau)=\int_{\frac{m}{2}}^{\tau}\frac{1}{\phi(x)}\mrm{d}x
\end{equation*}
hence equation \eqref{eq:HCSCK_original} becomes an ordinary integro-differential equation for $\phi$ and $c$.
\begin{nota} Essentially the same computations show that for the alternative choice of complexification \eqref{eq:low_rank_norm}, the real moment map equation reduces to the problem 
\begin{equation*}
\begin{split}
\sqrt{4-2\frac{c}{m^2}\frac{\phi(\tau)}{1+\tau}\mrm{e}^t}&\left(\phi''(\tau)+\frac{2\,\phi'(\tau)}{1+\tau}+\frac{1}{1+\tau}\right)+\\
&+\frac{8}{m(2+m)}+\frac{\frac{c}{m^2}\frac{\phi(\tau)}{1+\tau}\mrm{e}^t}{\sqrt{4-2\frac{c}{m^2}\frac{\phi(\tau)}{1+\tau}\mrm{e}^t}}\left(\frac{\phi(\tau)}{(1+\tau)^2}-\frac{(1+\phi'(\tau))^2}{\phi(\tau)}\right)=0
\end{split}
\end{equation*}
with the same boundary and positivity conditions, and the same definition of $t(\tau)$.
\end{nota}
\subsubsection{Approximate solutions}

We may regard the problem \eqref{eq:HCSCK_original} as a family of integro-differential equations parametrized by $m\in\bb{R}_{>0}$. Our aim is to show that for sufficiently small values of this parameter (i.e. in the limit when the fibres of $\bb{P}(\m{O}\oplus L)$ are very small) there is a solution to the equation. Notice however that $m$ appears both in the equation and in the domain of definition of $\phi(\tau)$, since $\tau$ takes values in $[0,m]$. It will then more convenient to first change variables, letting $\tau= m\,\lambda$, so that $\lambda$ takes values in the fixed interval $[0,1]$. If we rewrite the problem \eqref{eq:HCSCK_original} in terms of $\phi(\lambda)$ we get
\begin{equation*}
\begin{split}
\frac{\phi''(\lambda)}{m^2}+2\frac{\phi'(\lambda)}{m(1+m\,\lambda)}+\frac{1}{1+m\,\lambda}+\frac{4}{m(2+m)}&=\frac{c}{m^2}\,\frac{\mrm{exp}\left(\int_{\frac{1}{2}}^{\lambda}\frac{m}{\phi(x)}\mrm{d}x\right)}{1+m\,\lambda}\left(\left(\frac{\phi'(\lambda)}{m}+1\right)^2+\phi(\lambda)\frac{\phi''(\lambda)}{m^2}\right)\\
\phi(0)=\phi(1)&=0\\
\phi'(0)=-\phi'(1)&=m
\end{split}
\end{equation*}
which is of course equivalent to the problem
\begin{equation*}
\begin{split}
\phi''(\lambda)+2\,m\frac{\phi'(\lambda)}{1+m\,\lambda}+\frac{m^2}{1+m\,\lambda}+\frac{4\,m}{2+m}&=\frac{c}{m^2}\,\frac{\mrm{exp}\left(\int_{\frac{1}{2}}^{\lambda}\frac{m}{\phi(x)}\mrm{d}x\right)}{1+m\,\lambda}\left((\phi'(\lambda)+m)^2+\phi(\lambda)\,\phi''(\lambda)\right)\\
\phi(0)=\phi(1)&=0\\
\phi'(0)=-\phi'(1)&=m,
\end{split}
\end{equation*}
to be solved for a momentum profile $\phi(\lambda)$ and a constant $c >0$. 
\begin{nota} The corresponding equation for \eqref{eq:low_rank_norm} is given by
\begin{equation*}
\begin{split}
&\left(4-2\frac{c}{m^2}\frac{\phi}{1+\lambda\,m}\mrm{exp}\left(\int_{1/2}^\lambda\frac{m}{\phi}\mrm{d}x\right)\right)^{\frac{1}{2}}\left(\phi''+\frac{2\,m\,\phi'}{1+\lambda\,m}+\frac{m^2}{1+\lambda\,m}\right)+\\
&\quad\quad\quad+\frac{8\,m}{2+m}+\frac{\frac{c}{m^2}\frac{\phi}{1+\lambda\,m}\mrm{exp}\left(\int_{1/2}^\lambda\frac{m}{\phi}\mrm{d}x\right)}{\left(4-2\frac{c}{m^2}\frac{\phi}{1+\lambda\,m}\mrm{exp}\left(\int_{1/2}^\lambda\frac{m}{\phi}\mrm{d}x\right)\right)^{\frac{1}{2}}}\left(\frac{m^2\,\phi}{(1+\lambda\,m)^2}-\frac{(m+\phi')^2}{\phi}\right)=0
\end{split}
\end{equation*}
with the same boundary conditions.
\end{nota}
Introduce the space
\begin{equation*}
\m{V}_m:=\left\lbrace\phi\in\m{C}^\infty([0,1])\mid \phi>0\mbox{ in }(0,1),\ \phi(0)=\phi(1)=0\mbox{ and }
\phi'(0)=-\phi'(1)=m\right\rbrace.
\end{equation*}
Our problem is equivalent to showing that the integro-differential operator
\begin{equation*}
\m{F}_m:\m{V}_m\times\bb{R}_{>0}\to\m{C}^{\infty}_0([0,1])
\end{equation*}
defined by
\begin{equation}\label{eq:operatore_F_m}
\begin{split}
\m{F}_m(\phi,c):=&\phi''(\lambda)+2\,m\frac{\phi'(\lambda)}{1+m\,\lambda}+\frac{m^2}{1+m\,\lambda}+\frac{4\,m}{2+m}-\\&
-\frac{c}{m^2}\,\frac{\mrm{exp}\left(\int_{\frac{1}{2}}^{\lambda}\frac{m}{\phi(x)}\mrm{d}x\right)}{1+m\,\lambda}\left((\phi'(\lambda)+m)^2+\phi(\lambda)\,\phi''(\lambda)\right)
\end{split}
\end{equation}
has a zero. The reason why the image of $\m{F}_m$ lies inside the space of zero-average functions is that in its original form the real HcscK equation is of the form
\begin{equation*}
\mbox{scalar curvature }-\mbox{ its average }+\mbox{ divergence of a vector field }=0.
\end{equation*}
In fact we will show that $\m{F}_m$ has a zero for all sufficiently small $m>0$. 

We follow the well-developed approach of \emph{adiabatic limits} and in particular the excellent reference \cite{Fine_phd}. In this approach one first constructs a sufficiently good approximate solution and then perturbs this to a genuine solution by using a suitable quantitative versione of the Implicit Function Theorem.

Thus our first step is to find an approximate solution, i.e. $(\phi_0,c_0)\in\m{V}_m\times\bb{R}_{>0}$ such that
\begin{equation*}
\m{F}_m(\tilde{\phi},\tilde{c})=O(m^n)
\end{equation*}
for some $n>0$, in a purely formal sense. 
It is in fact possible to find approximate solutions up to every order, but we'll just need the first one
\begin{equation*}
\begin{split}
\phi_0(\lambda)=&\frac{m}{2(2+m)}\lambda(1-\lambda)\left(4+2\,m-m(4+3\,m)\lambda(1-\lambda)\right);\\
c_0=&2\,m^2.
\end{split}
\end{equation*}
For this choice of $\phi$, $c$, we have
\begin{equation*}
\m{F}_m(\phi_0,c_0)=O(m^3)
\end{equation*}
moreover,
\begin{equation}\label{eq:prima_approssimazione}
\begin{split}
\phi_0''(\lambda)+2\,m\frac{\phi_0'(\lambda)}{1+m\,\lambda}+\frac{m^2}{1+m\,\lambda}+\frac{4\,m}{2+m}=O(m^2)\\
\frac{\mrm{exp}\left(\int_{\frac{1}{2}}^{\lambda}\frac{m}{\phi_0(x)}\mrm{d}x\right)}{1+m\,\lambda}\left((\phi_0'(\lambda)+m)^2+\phi_0(\lambda)\,\phi_0''(\lambda)\right)=O(m^2).
\end{split}
\end{equation}
\begin{nota} Precisely the same choice of approximate solution works for the more complicated equation corresponding to \eqref{eq:low_rank_norm}. 
\end{nota}
\paragraph*{Linearization around the approximate solution.} We wish to study the differential of $\m{F}_m$ around our approximate solution, $(\phi_0,c_0)$. Introduce the space
\begin{equation*}
\m{V}:=\left\lbrace\phi\in\m{C}^\infty([0,1])\mid \phi(0)=\phi'(0)=\phi(1)=\phi'(1)=0\right\rbrace.
\end{equation*}
The tangent space to $\m{V}_m\times\bb{R}_{>0}$ is $\m{V}\times\bb{R}$. The linearization 
\begin{equation*}
\left(D\m{F}_m\right)_{(\phi,c)}:\m{V}\times\bb{R}\to\m{C}^{\infty}_0([0,1])
\end{equation*}
around a point $(\phi,c)\in\m{V}_m\times\bb{R}_{>0}$ is given by
\begin{equation}\label{eq:linearizzazione}
\begin{split}
\left(D\m{F}_m\right)_{(\phi,c)}&(u,k)=u''(\lambda)+2\,m\frac{u'(\lambda)}{1+\lambda\,m}-\frac{k}{m^2}\,\frac{\mrm{exp}\left(\int_{\frac{1}{2}}^{\lambda}\frac{m}{\phi(x)}\mrm{d}x\right)}{1+m\,\lambda}\left((\phi'(\lambda)+m)^2+\phi(\lambda)\,\phi''(\lambda)\right)+\\
&+\frac{c}{m^2}\,\frac{\mrm{exp}\left(\int_{\frac{1}{2}}^{\lambda}\frac{m}{\phi(x)}\mrm{d}x\right)}{1+m\,\lambda}\left(\int_{\frac{1}{2}}^\lambda\frac{m\,u(x)}{\phi(x)^2}\mrm{d}x\right)\left((\phi'(\lambda)+m)^2+\phi(\lambda)\,\phi''(\lambda)\right)-\\
&-\frac{c}{m^2}\,\frac{\mrm{exp}\left(\int_{\frac{1}{2}}^{\lambda}\frac{m}{\phi(x)}\mrm{d}x\right)}{1+m\,\lambda}\left(2(\phi'(\lambda)+m)u'(\lambda)+u(\lambda)\,\phi''(\lambda)+\phi(\lambda)\,u''(\lambda)\right).
\end{split}
\end{equation}

Now consider the linearization around the approximate solution $(\phi_0,c_0)$. Taking into account \eqref{eq:prima_approssimazione} and the fact that $\phi_0(\lambda)=O(m)$, we have for the various terms in the linearized operator:
\begin{equation*}
\begin{split}
&u''(\lambda)+2\,m\frac{u'(\lambda)}{1+\lambda\,m}=u''(\lambda)+O(m);\\
&\frac{k}{m^2}\frac{\mrm{exp}\left(\int_{\frac{1}{2}}^{\lambda}\frac{m}{\phi_0(x)}\mrm{d}x\right)}{1+m\,\lambda}\left((\phi_0'(\lambda)+m)^2+\phi_0(\lambda)\,\phi_0''(\lambda)\right)=-2\,k\,(3\,\lambda^2-2\,\lambda)+O(m);\\
&\frac{c_0}{m^2}\frac{\mrm{exp}\left(\int_{\frac{1}{2}}^{\lambda}\frac{m}{\phi_0(x)}\mrm{d}x\right)}{1+m\,\lambda}\left(\int_{\frac{1}{2}}^\lambda\frac{m\,u(x)}{\phi_0(x)^2}\mrm{d}x\right)\left((\phi_0'(\lambda)+m)^2+\phi_0(\lambda)\,\phi_0''(\lambda)\right)=O(m);\\
&\frac{c_0}{m^2}\,\frac{\mrm{exp}\left(\int_{\frac{1}{2}}^{\lambda}\frac{m}{\phi_0(x)}\mrm{d}x\right)}{1+m\,\lambda}\left(2(\phi_0'(\lambda)+m)u'(\lambda)+u(\lambda)\,\phi_0''(\lambda)+\phi_0(\lambda)\,u''(\lambda)\right)=O(m).
\end{split}
\end{equation*}

Hence we see that the differential of $\m{F}_m$ at the point $(\phi_0,c_0)$ is
\begin{equation*}
\left(D\m{F}_m\right)_{(\phi_0,c_0)}(u,k)=u''(\lambda)+2\,k(3\,\lambda^2-2\,\lambda)+O(m).
\end{equation*}

\begin{lemma}
The map
\begin{equation}\label{eq:linearizzazione_primo_ordine}
\begin{split}
D:\m{V}\times\bb{R}&\to\m{C}^\infty_0([0,1])\\
(u,k)&\mapsto u''(\lambda)+2\,k(3\,\lambda^2-2\,\lambda)
\end{split}
\end{equation}
is an isomorphism.
\end{lemma}

\begin{proof}
Fix $f\in\m{C}^\infty([0,1])$ and consider
\begin{equation*}
u''(\lambda)+2\,k(3\,\lambda^2-2\,\lambda)=f(\lambda)
\end{equation*}
as a differential equation for $u(\lambda)$. The general solution is given by
\begin{equation*}
u(\lambda)=\int_0^\lambda\left(\int_0^yf(x)\mrm{d}x\right)\mrm{d}y-2\,k\left(\frac{\lambda^4}{4}-\frac{\lambda^3}{3}\right)+C_1\,\lambda+C_2
\end{equation*}
for constants $C_1$, $C_2$. There is a \textit{unique} choice of $k$, $C_1$, $C_2$ such that this solution $u$ lies in $\m{V}$, and this choice is
\begin{equation*}
C_1=C_2=0\mbox{ and }k=-6\,\int_0^1\left(\int_0^yf(x)\mrm{d}x\right)\mrm{d}y.
\end{equation*}
\end{proof}
So we have found an explicit inverse to the zeroth-order part of $\left(D\m{F}_m\right)_{(\phi_0,c_0)}$.
\begin{nota} The linearisation of the more complicated equation corresponding to \eqref{eq:low_rank_norm} is in fact just the same as $D\m{F}_m$, up to $O(m)$ terms, so Lemma \ref{eq:linearizzazione_primo_ordine} also applies to that case.  
\end{nota}
\paragraph*{Some estimates.} We recall two results that are essential to obtain an exact solution from the approximate one. The first one is a quantitative version of the usual fact that invertibility is an open property, while the second is a quantitative version of the Inverse Function Theorem.

\begin{lemma}[Lemma $7.10$ in \cite{Fine_phd}]\label{lemma:inversione}
Let $D:B_1\to B_2$ be a bounded linear map between Banach spaces, with bounded inverse $D^{-1}$. Then any other linear bounded operator $L$ such that $||D-L||\leq (2\,||D^{-1}||)^{-1}$ is also invertible, and $||L^{-1}||\leq 2\,||D^{-1}||$.
\end{lemma}

\begin{lemma}[Theorem $5.3$ in \cite{Fine_phd}]\label{lemma:funzione_inversa_quantitativo}
Let $F:B_1\to B_2$ be a differentiable map between Banach spaces, with derivative $DF:B_1\to B_2$ at $0$. Assume that $DF$ is an isomorphism, with inverse $P$, and let $\delta$ be such that $F-DF$ is Lipschitz on the ball $B(0,\delta)$ with a Lipschitz constant $l\leq (2||P||)^{-1}$. Then, for any $y\in B_2$ such that $||y-F(0)||<\delta\,(2||P||)^{-1}$ there is a unique $x$ in $B_1$ such that $||x||<\delta$ and $F(x)=y$.  
\end{lemma}
In order to apply these results we embed $\m{V}\times\bb{R}$ and $\m{C}^{\infty}_0([0,1])$ into Banach spaces as follows:
\begin{itemize}
\item the first Banach space is the closure $\overline{\m{V}}$ of $\m{V}$ in $\m{C}^{l+2,\beta}([0,1])$, with the usual H\"older norm, for $l$ large enough and $0<\beta< 1$. We can then take the direct sum of this space with $\bb{R}$, and we let $\left(\overline{\m{V}}\times{\bb{R}},||.||\right)$ be the resulting Banach space;
\item for $\m{C}^{\infty}_0([0,1])$, we'll just consider it as a subset of $\m{C}^{l,\beta}_0([0,1])$.
\end{itemize}

Then we have the following estimate for the norm of the operator $D$ defined in equation \eqref{eq:linearizzazione_primo_ordine} (that is the zeroth-order part of the linearization of $\m{F}_m$ around the approximate solution $(\phi_0,c_0)$):
\begin{equation*}
\norm{D(u,k)}_{\m{C}^{l,\beta}}\leq\norm{u''}_{\m{C}^{l,\beta}}+2\,|k|\,\norm{3\,\lambda^2-2\,\lambda }_{\m{C}^{l,\beta}}\leq\norm{u}_{\m{C}^{l+2,\beta}}+22\,\card{k}\leq 22\,\norm{(u,k)}.
\end{equation*}
In order to prove a similar estiamate for the inverse, fix $f\in \m{C}^{l,\beta}_0([0,1])$ and let $(u_0,k_0):=D^{-1}(f)$. Then
\begin{equation*}
\card{k_0}=\left|6\,\int_0^1\left(\int_0^yf(x)\mrm{d}x\right)\mrm{d}y\right|\leq 3\,\sup{f}\leq 3\,\norm{f}_{\m{C}^{l,\beta}}
\end{equation*}
\begin{equation*}
\begin{split}
\norm{u_0}_{\m{C}^{l+2,\beta}}=&\norm*{\int_0^\lambda\left(\int_0^yf(x)\mrm{d}x\right)\mrm{d}y+2\,k_0\left(\frac{\lambda^4}{4}-\frac{\lambda^3}{3}\right)}_{\m{C}^{l,\beta}}\leq\\
&\leq\norm*{\int_0^\lambda\left(\int_0^yf(x)\mrm{d}x\right)\mrm{d}y}_{\m{C}^{l,\beta}}+2\,\card{k_0}\,\norm*{\frac{\lambda^4}{4}-\frac{\lambda^3}{3}}_{\m{C}^{l,\beta}}< 70\,\norm{f}_{\m{C}^{l,\beta}}
\end{split}
\end{equation*}
This shows
\begin{equation*}
\norm{D^{-1}(f)}< 73\,\norm{f}_{\m{C}^{l,\beta}}.
\end{equation*}

\begin{lemma}
For all sufficiently small $m > 0$ the map $\left(D\m{F}_m\right)_{(\phi_0,c_0)}$ is a linear isomorphism of Banach spaces. Moreover the norm of its inverse is less than $146$.
\end{lemma}
\begin{proof}
We can use Lemma \ref{lemma:inversione}; indeed, we know that $\left(D\m{F}_m\right)_{(\phi_0,c_0)} - D =O(m)$ so for $m$ small enough we'll have that the norm of the difference is less than $\frac{1}{146}$, as is needed to apply the Lemma.
\end{proof}
\begin{nota}
In fact precise estimates for the norm of $\left(D\m{F}_m\right)_{(\phi_0,c_0)}$ and its inverse are not needed. We only require that the norm of the inverse can be controlled by a quantity which is independent of $m$ and $l$. In what follows we'll write simply $N$ for the norm of $\left(D\m{F}_m\right)_{(\phi_0,c_0)}^{-1}$.
\end{nota}

\subsubsection{Proof of Theorem \ref{teorema:HCSCK_rigata_esistenza}}

We showed that for $m$ small enough we have an approximate solution $(\phi_0,c_0)$, depending on $m$, to the equation $\m{F}_m=0$, such that
\begin{equation*}
\m{F}_m(\phi_0,c_0)=O(m^3). 
\end{equation*}
Moreover, we know that the differential of $\m{F}$ around this approximate solution is an isomorphism of Banach spaces. Our next step is to use Lemma \ref{lemma:funzione_inversa_quantitativo} to show that for small enough $m$ we have a genuine solution to $\m{F}_m=0$.

Let $\m{G}_m:\overline{\m{V}}\times\bb{R}\to L^\infty_0([0,1])$ be defined as
\begin{equation*}
\m{G}_m(u,c):=\m{F}_m(\phi_0+u,c_0+c).
\end{equation*}
The differential of $\m{G}_m$ at $0$ is just $\left(D\m{F}_m\right)_{(\phi_0,c_0)}$, so it is an isomorphism. Then Lemma \ref{lemma:funzione_inversa_quantitativo} tells us that, if $\delta$ is the radius of a ball over which $\m{G}_m-D\m{G}_m$ is Lipschitz with a constant that is less than $\frac{1}{N}$, then for any $y$ such that $\norm{y-\m{G}_m(0)}\leq\frac{\delta}{N}$ there is a unique $x$ such that $\norm{x}<\delta$ and $\m{G}_m(x)=y$.

As $\m{G}_m(0)=O(m^3)$, in order to apply the result, we need to show that $\delta$ can be chosen to vanish slower than $m^3$ as $m \to 0$. 

However we also want $(\Phi,C)$ to satisfy some \emph{positivity conditions}: $\Phi$ should be strictly positive in the interior of $[0,1]$, and $C$ should be positive. The approximate solution satisfies these conditions, however $\phi_0(\lambda)=O(m)$ and $c_0=O(m^2)$; so in order to preserve positivity we need to choose a radius $\delta$ that goes to $0$ faster than $m^2$ as $m \to 0$ 

The next result shows that we can choose $\delta$ as required.

\begin{lemma}\label{lemma:costante_lipschitz}
Let $k\geq 2$. If $\delta\in O(m^k)$ then for $m$ small enough $\m{G}_m-D\m{G}_m$ is Lipschitz on $B(0,\delta)\subset\overline{\m{V}}\times\bb{R}$ with Lipschitz constant smaller than $\frac{1}{N}$.
\end{lemma}

This tells us that for a small enough $m$ we can choose $\delta$ in such a way that the solution of the equation that we have found satisfies the positivity conditions; it is enough to use Lemma \ref{lemma:costante_lipschitz} for $k=2+\frac{1}{2}$.

\begin{proof}[Proof of Lemma \ref{lemma:costante_lipschitz}]
Let $\m{N}_m:=\m{G}_m-D\m{G}_m$ be the nonlinear part of $\m{G}_m$. For $a,b\in B(0,\delta)$, the Mean Value Theorem implies
$\norm{\m{N}_m(a)-\m{N}_m(b)}_{\m{C}^{l,\beta}}\leq \norm{a-b}_{\m{C}^{l+2,\beta}}\cdot\mrm{sup}_{z\in B(0,\delta)}\norm{\left(D\m{N}_m\right)_z}$.
For $z\in B(0,\delta)$,
\begin{equation*}
\begin{split}
\left(D\m{N}_m\right)_z(\varphi)=&\left(D\m{G}_m\right)_z(\varphi)-\left(D\m{G}_m\right)_0(\varphi)=\\
=&\left(D\m{F}_m\right)_{(\phi_0,c_0)+z}(\varphi)-\left(D\m{F}_m\right)_{(\phi_0,c_0)}(\varphi).
\end{split}
\end{equation*}
We will show that this quantity is $O(m)$, if $\delta\in O(m^2)$. Since $O(m^k)\subseteq O(m^2)$ for $k\geq 2$, this will give us the thesis.

To prove the claim, let $z=:(\tilde{y},\tilde{c})$; if $\delta$ is $O(m^2)$, since $\norm{z}\leq\delta$ also $\tilde{y}=O(m^2)$ and $\tilde{c}=O(m^2)$. The linearization of $\m{F}_m$ at $(\phi,c):=(\phi_0,c_0)+z$ is given by (recall equation \eqref{eq:linearizzazione})
\begin{equation*}
\begin{split}
\left(D\m{F}_m\right)_{(\phi,c)}&(u,k)=u''(\lambda)+2\,m\frac{u'(\lambda)}{1+\lambda\,m}-\frac{k}{m^2}\frac{\mrm{exp}\left(\int_{\frac{1}{2}}^{\lambda}\frac{m}{\phi(x)}\mrm{d}x\right)}{1+m\,\lambda}\left((\phi'(\lambda)+m)^2+\phi(\lambda)\,\phi''(\lambda)\right)+\\
&+\frac{c}{m^2}\frac{\mrm{exp}\left(\int_{\frac{1}{2}}^{\lambda}\frac{m}{\phi(x)}\mrm{d}x\right)}{1+m\,\lambda}\left(\int_{\frac{1}{2}}^\lambda\frac{m\,u(x)}{\phi(x)^2}\mrm{d}x\right)\left((\phi'(\lambda)+m)^2+\phi(\lambda)\,\phi''(\lambda)\right)-\\
&-\frac{c}{m^2}\,\frac{\mrm{exp}\left(\int_{\frac{1}{2}}^{\lambda}\frac{m}{\phi(x)}\mrm{d}x\right)}{1+m\,\lambda}\left(2(\phi'(\lambda)+m)u'(\lambda)+u(\lambda)\,\phi''(\lambda)+\phi(\lambda)\,u''(\lambda)\right).
\end{split}
\end{equation*}

Let us consider the series expansions:
\begin{equation*}
\begin{split}
\mrm{exp}\left(\int_{\frac{1}{2}}^{\lambda}\frac{m}{\phi(x)}\mrm{d}x\right)=&\mrm{exp}\left(\int_{\frac{1}{2}}^{\lambda}\frac{m}{\phi_0+\tilde{y}}\mrm{d}x\right)=\mrm{exp}\left(\int_{\frac{1}{2}}^{\lambda}\frac{m}{\phi_0}-\frac{m}{\phi_0^2}\tilde{y}+\dots\mrm{d}x\right)=\\
=&\mrm{exp}\left(\int_{\frac{1}{2}}^{\lambda}\frac{m}{\phi_0}+O(m)\,\mrm{d}x\right)=\mrm{exp}\left(\int_{\frac{1}{2}}^{\lambda}\frac{m}{\phi_0}\mrm{d}x\right)+O(m),
\end{split}
\end{equation*}
\begin{equation*}
\begin{split}
(\phi'(\lambda)+m)^2&+\phi(\lambda)\,\phi''(\lambda)=(\phi_0'+\tilde{y}'+m)^2+(\phi_0+\tilde{y})(\phi_0''+\tilde{y}'')=\\
=&(\phi_0'+m)^2+(\tilde{y}')^2+2(\phi_0'+m)\tilde{y}'+\phi_0\,\phi_0''+\tilde{y}\,\phi_0''+\phi_0\,\tilde{y}''+\tilde{y}\,\tilde{y}''=\\
=&(\phi_0'+m)^2+\phi_0\,\phi_0''+O(m^3).
\end{split}
\end{equation*}
So we have
\begin{equation*}
\begin{split}
\frac{k}{m^2}&\frac{\mrm{exp}\left(\int_{\frac{1}{2}}^{\lambda}\frac{m}{\phi(x)}\mrm{d}x\right)}{1+m\,\lambda}\left((\phi'(\lambda)+m)^2+\phi(\lambda)\,\phi''(\lambda)\right)=\\
&=\frac{k}{m^2}\frac{\mrm{exp}\left(\int_{\frac{1}{2}}^{\lambda}\frac{m}{\phi_0}\mrm{d}x\right)}{1+m\,\lambda}\,\left((\phi_0'+m)^2+\phi_0\,\phi_0''\right)+O(m).
\end{split}
\end{equation*}
For the other terms, recalling that $c=c_0+\tilde{c}=O(m^2)$, we have simply
\begin{equation*}
\frac{c}{m^2}\frac{\mrm{exp}\left(\int_{\frac{1}{2}}^{\lambda}\frac{m}{\phi(x)}\mrm{d}x\right)}{1+m\,\lambda}\left(\int_{\frac{1}{2}}^\lambda\frac{m\,u(x)}{\phi(x)^2}\mrm{d}x\right)\left((\phi'(\lambda)+m)^2+\phi(\lambda)\,\phi''(\lambda)\right)=O(m)
\end{equation*}
\begin{equation*}
\frac{c}{m^2}\,\frac{\mrm{exp}\left(\int_{\frac{1}{2}}^{\lambda}\frac{m}{\phi(x)}\mrm{d}x\right)}{1+m\,\lambda}\left(2(\phi'(\lambda)+m)u'(\lambda)+u(\lambda)\,\phi''(\lambda)+\phi(\lambda)\,u''(\lambda)\right)=O(m).
\end{equation*}
As a consequence
\begin{equation*}
\begin{split}
\left(D\m{F}_m\right)_{(\phi,c)}(u,k)=&u''(\lambda)+2\,m\frac{u'(\lambda)}{1+\lambda\,m}-\frac{k}{m^2}\frac{\mrm{exp}\left(\int_{\frac{1}{2}}^{\lambda}\frac{m}{\phi_0}\mrm{d}x\right)}{1+m\,\lambda}\,\left((\phi_0'+m)^2+\phi_0\,\phi_0''\right)+O(m)=\\
=&\left(D\m{F}_m\right)_{(\phi_0,c_0)}(u,k)+O(m).
\end{split}
\end{equation*}
Then for $z\in B(0,\delta)$, $\norm{\left(D\m{N}_m\right)_z}$ is $O(m)$. Hence for $m$ small enough, on a ball of radius $m^2$ the Lipschitz constant of $\m{N}_m$ will be smaller than $\frac{1}{N}$.
\end{proof}

This settles the problem of existence of a solution $(\Phi,C)\in\m{C}^{l+2,\beta}([0,1])\times\bb{R}$ of $\m{F}_m(y,c)=0$. To prove smoothness we consider the existence result we just showed for increasing values of $l$, with corresponding solutions $\Phi_l$. The uniqueness statement in Lemma \ref{lemma:funzione_inversa_quantitativo}, together with the fact that $\norm{u}_{\m{C}^{l,\beta}}\leq\norm{u}_{\m{C}^{l+1,\beta}}$, implies that actually all the various $\Phi_l$ are the same function, that is of course smooth.
\begin{nota} Given our previous remarks, it is straightforward to check that the same proof works for the more complicated equation corresponding to \eqref{eq:low_rank_norm}.  
\end{nota}

\addcontentsline{toc}{section}{References}
\bibliographystyle{alpha}
\bibliography{../bibliografia_HCSCK}
\vskip1cm
\noindent\rm{SISSA, via Bonomea 265, 34136 Trieste, Italy}\\
\noindent\tt{cscarpa@sissa.it}\\
\noindent\tt{jstoppa@sissa.it}
\end{document}